\documentclass[11pt]{article}
\usepackage{amssymb, amsthm, amsmath, amscd}
 \usepackage[usenames,dvipsnames]{color}

\newtheorem{thm}{Theorem}[section]
\newtheorem{lem}[thm]{Lemma}
\newtheorem{prop}[thm]{Proposition}
\newtheorem{cor}[thm]{Corollary}

\theoremstyle{definition}\newtheorem{df}[thm]{Definition}
\theoremstyle{definition}\newtheorem{rem}[thm]{Remark}
\theoremstyle{definition}

\renewcommand{\phi}{\varphi}

\newcommand{\Z}{\mathbb{Z}}

\newcommand{\C}{\mathbb{C}}
\newcommand{\T}{\mathbb{T}}
\newcommand{\U}{\mathcal{U}}

\newcommand{\Aff}{\operatorname{Aff}}

\newcommand{\hm}{homomorphism}
\newcommand{\dt}{\delta}
\newcommand{\ep}{\epsilon}

\newcommand{\andeqn}{\,\,\,{\rm and}\,\,\,}
\newcommand{\rforal}{\,\,\,{\rm for\,\,\,all}\,\,\,}
\newcommand{\CA}{$C^*$-algebra}
\newcommand{\SCA}{$C^*$-subalgebra}

\newcommand{\af}{{\alpha}}

\newcommand{\dist}{{\rm dist}}

\newcommand{\beq}{\begin{eqnarray}}
\newcommand{\eneq}{\end{eqnarray}}
\newcommand{\tforal}{\,\,\,\text{for\,\,\,all}\,\,\,}
\newcommand{\tand}{\,\,\,\text{and}\,\,\,}

\title{Distance between  unitary  orbits of normal elements
in simple \CA s of real rank zero }
\author{Shanwen Hu and Huaxin Lin\footnote{corresponding author: hlin@uoregon.edu}
 }
\date{}

\begin{document}

\maketitle

\begin{abstract}
Let $x, y$ be two normal elements in a unital simple \CA\, $A.$  We introduce a  function  $D_c(x, y)$ and  show that in a unital simple
AF-algebra there is a constant $1>C>0$ such that
$$
C\cdot D_c(x, y)\le {\rm dist}({\cal U}(x),{\cal U}(y))\le D_c(x,y),
$$
where ${\cal U}(x)$ and ${\cal U}(y)$ are the closures of the unitary orbits of $x$ and of $y,$ respectively.
We also generalize this to unital simple \CA s with real rank zero, stable rank one
and weakly unperforated $K_0$-group.
More complicated estimates are given in the presence of
non-trivial $K_1$-information.

\end{abstract}

\section{Introduction}

Let $H$ be a Hilbert space and let $B(H)$ be the \CA\, of all bounded
operators. The study of normal operators in $B(H)$ has a long
history. It has been an interesting and important problem to
determine when two normal operators are unitarily equivalent in a
subalgebra $A$ of $B(H).$  Any detailed account of  history will
inevitably involve an enormous amount of literature. We will choose to
limit ourselves to the immediate concerns of this paper. We study the distance between
unitary orbits of normal elements. Some of the {pioneer works on  this
subject have  been made by Ken Davidson}  in \cite{Dv}, \cite{Dv2} and \cite{Dv1}. More recent work on the subject
can be found in \cite{SK} and \cite{Sher}.

 Let $A$ be a unital \CA\,
and let $x\in A$ be a normal element. The unitary orbit of $x$ is
defined to be the set $\{u^*xu: u\in U(A)\},$ where $U(A)$
is the unitary group of $A.$ {Denote by ${\cal U}(x)$ the closure of the unitary orbit of $x.$} Suppose that $y\in A$ is another
normal element. Denote by $X$ and $Y$ the
spectrum of $x$ and of $y,$ respectively. Let $\phi_X, \phi_Y: C(X\cup
Y)\to A$ be the \hm s defined by $\phi_X(f)=f(x)$ and
$\phi_Y(f)=f(y)$ for all $f\in C(X\cup Y).$ Suppose that $A$ has a
unique tracial state $\tau.$
 Denote by $\mu_{\tau\circ \phi_X}$ and $\mu_{\tau\circ \phi_Y}$ the
two probability measures on $X\cup Y$ defined by the positive linear functionals
$\tau\circ \phi_X$ and $\tau\circ \phi_Y,$ respectively.
For each open subset  $O\subset \C$ and $r>0,$   denote by
$O_r=\{\xi\in \C: {\rm dist}(\xi, O)<r\}.$
Define
$$
r_O=\inf\{ r: \mu_{\tau\circ \phi_X}(O)\le \mu_{\tau\circ \phi_Y}(O_r)\}
$$
and
define
$$
D_T(x, y)=\sup \{ r_O: O \,\,\, {\rm open\,\,\, subset\,\,\, of }\,\,\, \C\}.
$$
For finite dimensional Hilbert spaces, when $A=M_n,$ the $n\times n$ matrix algebra over  $\C,$ an application of  the Marriage Theorem shows that
\beq\label{In1}
{\rm dist}({\cal U}(x), {\cal U}(y))\le D_T(x, y)
\eneq
(see \cite{HN} and \cite{Dv}, for example).

 We first realize that (\ref{In1}) also  holds for the case that $A$ is a UHF-algebra.
 Apart from the application of the Marriage Theorem, the proof  is also  based on an important fact that normal
elements in $A$ can be approximated by normal elements with finite spectrum
(\cite{Lnjot94}).   If we allow $A$ to be a general unital simple AF-algebra with a unique tracial
state, the above bound
no long works because of the presence of possible infinitesimal elements in $K_0(A).$
Even without infinitesimal elements in $K_0(A),$ in the case that  $A$ has infinitely many extremal tracial states,
the usage of the Marriage Theorem has to mix appropriately with the Riesz interpolation property.
With  the help of the Cuntz semigroups, which {are  more appropriate tools} to compare
positive elements, we are able to establish a modified upper bound
formula for the distance between unitary orbits of two normal elements in a general unital AF-algebra
(see \ref{TAF} below).   
The fact that  normal elements in an AF-algebra
can be approximated by normal elements with finite spectrum plays an essential role in the proof.
This  follows from the result that a pair of almost commuting self-adjoint
matrices is close to a pair of commuting  self-adjoint matrices (see \cite{Lnalm}
and \cite{FR}). In \cite{Lnoldjfa}, it was shown that, in a unital separable
simple \CA\, of real rank zero, stable rank one and with weakly unperforated $K_0(A),$
a normal element $x$ can be approximated by those normal elements with finite
spectrum if $\lambda-x\in {\rm Inv}_0(A)$ (the connected component
of invertible elements of $A$ containing the identity) for all $\lambda\not\in {\rm sp}(x).$
In this case  we also prove that the same upper bound works for distance between unitary orbits of two  normal elements in $A$ which have vanishing $K_1$ information (see  \ref{Toriginal}).

The distance between unitary orbits of normal elements in unital purely infinite simple \CA s were recently studied in \cite{SK}. In that case, one could get a precise formula for the distance at least for the case that $K_0(A)=0$ when no $K_1$-information involved. However, the distance is basically given by $d_H({\rm sp}(x), {\rm sp}(y)),$ the Hausdorff distance
between the spectra. It is the presence of the trace in the finite \CA s
that makes our upper bound more complicated and sophisticated. But it is exactly {this}  phenomenon that is  exciting.  However, it is more desirable, in this current
study, to include the cases that \CA s have non-trivial $K_1$-groups and normal elements
have non-trivial $K_1$-information.

To study the unitary orbits of normal elements, one has to know when two normal elements
are approximately unitarily equivalent.
When $A$ is a unital purely infinite simple \CA, or $A$ is a unital separable simple \CA\, with finite tracial rank, we know exactly when two normal elements are approximately unitarily equivalent (see \cite{Da}, \cite{Lntrans} and \cite{LnTAMS12}). These works actually deal with the problem of unitary orbits
of \hm s from $C(X),$ the \CA\, of continuous functions on a compact metric space
$X,$  to a unital purely infinite simple \CA\, or a unital separable simple
\CA\, with finite tracial rank. These studies are closely related to the Elliott
program of classification of amenable \CA s.

To consider the unitary orbits of normal elements in a unital separable simple
\CA\, $A$ with real rank zero, stable rank one and weakly unperforated $K_0(A),$ we first
present a theorem that two normal elements $x,\, y\in A$ are approximately unitarily equivalent
if and only if ${\rm sp}(x)={\rm sp}(y),$ $(\phi_x)_{*i}=(\phi_y)_{*i},$ $i=0,1$ and
$\tau\circ \phi_x=\tau\circ \phi_y$ for all $\tau\in T(A)$ (the tracial state space of $A$),
where $\phi_x, \phi_y: C({\rm sp}(x))\to A$ are defined by
$\phi_x(f)=f(x)$ and by $\phi_y(f)=f(y)$ for all $f\in C({\rm sp}(x)),$ respectively.
This is a generalization of the similar result in \cite{Lntrans} (which  only works for unital simple \CA\, with
tracial rank zero).

Let $A$ be a unital separable simple
\CA\,  with real rank zero, stable rank one and weakly unperforated $K_0(A)$ and let
$x,\, y\in A$ be two normal elements with ${\rm sp}(x)=X$ and ${\rm sp}(y)=Y,$ respectively.
We first consider the case that $(\lambda-x)^{-1}(\lambda-y)\in {\rm Inv}_0(A)$ for
all $\lambda\not\in X\cup Y.$  With the help of a Mayer-Vietoris Theorem, we are able to
present a reasonable upper bound for the distance between unitary orbits
of $x$ and $y$ (see  \ref{MT1} below) in the same sprit of \ref{Toriginal} mentioned above.  However, there are normal elements with ${\rm sp}(x)={\rm sp}(y)$ {which}  induce exactly the same map on the Cuntz semigroup, but, for any given $\lambda\not\in X\cup Y,$ $(\lambda-x)^{-1}(\lambda-y)\not\in {\rm Inv}_0(A).$ In this case the same upper bound mentioned above is zero. Nevertheless, as first found by Davidson (\cite{Dv2}),
\beq\label{In2}
{\rm dist}({\cal U}(x), {\cal U}(y))\ge \sup\{{\rm dist}(\lambda, {\rm sp}(x))+{\rm dist}(\lambda, {\rm sp}(y)) \},
\eneq
where the supremum is taken among those $\lambda\not\in {\rm sp}(x)\cup
{\rm sp}(y)$ such that
$(\lambda-x)^{-1}(\lambda-y)\not\in {\rm Inv}_0(A).$

Based on  \ref{MT1}, we also present an upper bound for distance between unitary orbits of normal elements in $A,$ which is a combination of the upper bound in  \ref{MT1} together {with} the sprit of the lower bound given by Davidson (see   \ref{MT2} below).

Then, of course, there is the  issue of the lower bound. It was shown
by Davidson (\cite{Dv}) that there is a universal constant $1>C>0$
such that
\beq\label{In3}
{\rm dist}({\cal U}(x), {\cal U}(y))\ge
C\cdot D_T(x,y)
\eneq
in the case that $A=M_n$ (or $A=B(H)$ for infinite dimensional Hilbert space $H$
with some modification).  The constant is
computed at least 1/3.
 It
was shown even in the case that $n\ge 3,$ the constant $C$ cannot be
made equal to 1 (\cite{Hk}). We show that a similar lower bound (with the
same constant $C$) holds for unital AF-algebras and, more generally,
for unital separable simple \CA s of tracial rank zero. A different
lower bound $d_c(x, y)$ is also presented. There are cases that
$d_c(x, y)=D_c(x,y)$  with $\dist(\U(x),\U(y))=d_c(x,y)$ and cases that $d_c(x,y)<D_c(x,y)$ with $\dist(\U(x),\U(y))=D_c(x,y)$.

Briefly, the paper is organized as follows:
{ Section 2}  serves as a preliminary of the paper. Some metrics
associated with the measure distribution and the Cuntz semigroups are introduced.
In Section 3, we present an upper bound for the distance between
unitary orbits of normal elements in unital AF-algebras and in
unital separable simple \CA s $A$ with real rank zero, stable rank one and weakly unperforated $K_0(A),$ when $K_1$-information of the relevant normal elements vanish.
In Section 4, we show, with the metric introduced in Section 2, that  normal elements can always be approximated by normal elements
with finite spectrum in a unital simple \CA\, with stable rank one, real rank zero and
weakly unperforated $K_0(A).$ Another function $D_c^e(x,y)$ which is a  modification of $D_c(x,y)$
is also introduced.
In Section 5, we show exactly when two normal elements in $A$ are approximately unitarily equivalent.
In Section 6,  we present an upper bound for the distance between unitary orbits of normal elements
$x,\, y\in A$ under the condition that $\lambda-x$ and $\lambda-y$ give the same
$K_1$ element (but not necessarily zero) when $\lambda\not\in {\rm sp}(x)\cup {\rm sp}(y).$  In Section 7,
we present a general  upper bound for the distance between unitary orbits of normal elements
in $A$ (without any assumption on $K_1$-information). In Section 8, we discuss the lower bound of the distance between unitary orbits of normal elements.

{\bf Acknowledgements}
The most  research of this work was done
when both authors were at the Research Center for Operator Algebras in
the East China Normal University and both were partially supported by the Center.  The second named author  also
acknowledges the support from a grant of NSF. The authors would also like to thank the referee 
for  carefully checking  and for numerous   suggestions.

\section{Preliminaries}

\begin{df}\label{Duorbit}
{\rm Let $A$ be a unital \CA. Denote by $U(A)$ the unitary group of
$A.$  Let $x\in A$ be a normal element. Define ${\cal U}(x)$ {to} be the
closure of
$$
\{u^*xu: u\in U(A)\}.
$$
}
\end{df}

\begin{df}\label{XUY}
Fix a compact metric space $\Omega.$ Let $r>0.$ For each subset
$S\subset \Omega,$ {define}
\beq\nonumber
S_r=\{t\in \Omega: {\rm dist}(t, S)<r\},\,\,\,
S_{-r}=\{t\in\Omega:{\rm dist}(t,S^c)>r\}
\eneq
and,
{denote by ${\bar S}$} the closure of
$S.$

\end{df}

\begin{df}\label{DTArho}
{\rm Let $A$ be a unital \CA. Denote by $T(A)$ the tracial state
space of $A.$ Let ${\rm Aff}(T(A))$ {be} the  space of all real affine
{continuous}
functions on $T(A).$ Denote by $\rho_A: K_0(A)\to  {\rm Aff}(T(A))$
the order preserving \hm\, defined by
$\rho_A([p])(\tau)=(\tau\otimes Tr)(p)$ for all projections in
$M_n(A),$ where $Tr$ is the standard trace on $M_n,$ $n=1,2,....$ }
\end{df}

\begin{df}\label{DK0K1}
{\rm Let $A$ be unital \CA.  Denote by ${\rm Inv}_0(A)$ the
connected component of the set of invertible elements which contains
the identity of $A.$ Let $x,\, y\in A$ and let $\lambda\not\in {\rm
sp}(x)\cup {\rm sp}(y).$ Then $\lambda-x$ and $\lambda-y$ are
invertible. Denote by $[\lambda-x]$ the corresponding element in
$K_1(A).$ So $[\lambda-x]=[\lambda-y]$ means that they represent the
same element in $K_1(A).$ In the case that $A$ is of stable rank
one, $[\lambda-x]=[\lambda-y]$ is equivalent to
$(\lambda-x)^{-1}(\lambda-y)\in {\rm Inv}_0(A).$ }
\end{df}

\begin{df}\label{Dcuntz1}
{\rm Let $A$ be a $C^*$-algebra   and let $a, b\in A_+$  be two
positive elements. We write $a\lesssim b$ if there is a sequence of
elements $\{x_n\}\subset A$ such that
$$
x_n^*bx_n\to a
$$
as $n\to\infty.$  If $a\lesssim b$ and $b\lesssim a,$ then we write
$[a]=[b]$ and say that $a$ and $b$ are equivalent in Cuntz
semi-group.

If $p, \, q\in A$ are two projections, then $p\lesssim q$  means
that there is a partial isometry $w\in A$ such that $w^*w=p$ and
$ww^*\le q.$

}
\end{df}

\begin{df}\label{Ddim}
{\rm Let $\ep>0.$ Denote by $f_\ep$ the continuous function on
$[0,\infty)$ such that $0\le f_{\ep}\le 1;$ $f(t)=1$ if $\xi\in
{[\ep,\infty)}$ and $f(t)=0$ if $\xi\in [0, \ep/2]$ and $f(t)$ is linear in
$(\ep/2, \ep).$

Let $b\in A_+$ defined
$$
d_{\tau}(b)=\lim_{\ep\to 0} \tau(f_\ep(b))
$$
for $\tau\in T(A).$

$A$ is said to have  strict comparison for positive elements, if
$$
d_\tau(a)<d_\tau(b)\tforal \tau\in T(A)
$$
implies that $a\lesssim b.$ In this paper, we mainly study those \CA s $A$ which have  real rank
zero, stable rank one and  weak unperforated $K_0(A).$  Such \CA s  always have
strict comparison for positive elements.  {When $A$ has tracial rank zero (see 3.6.2 of \cite{Lnbk}), we write $TR(A)=0.$ If $A$ is a unital simple \CA\, with $TR(A)=0,$ then
$A$ has real rank zero, stable rank one and weakly unperforated $K_0(A).$}
}
\end{df}

\begin{df}\label{Dfo}
{\rm Let $\Omega$ be a compact metric space and let $O\subset
\Omega$ be an open subset. Throughout this paper, $f_O$ denotes a
positive function  with $0\le f_O\le 1$ whose support is exactly
$O,$ i.e., $f_O(t)>0$ for all $t\in O$ and $f_O(t)=0$ for all
$t\not\in O.$ If $x,\, y\in A$ are two normal elements with $X={\rm
sp}(x)$ and
 $Y={\rm sp}(y),$ we let $\Omega=X\cup Y.$
 Let $\phi_X, \psi_Y: C(\Omega)\to A$ be unital \hm s
 defined by $\phi_X(f)=f(x)$ and $\psi_Y(f)=f(y)$ for all $f\in C(\Omega).$}
\end{df}

\begin{df}\label{Dcuntz}
{\rm Let $W(A)$  be the Cuntz semi-group which are equivalence
classes of positive elements in $M_{\infty}(A).$


}
\end{df}

 \begin{df}\label{DWW}
 {\rm
 Let $A$ be a unital \CA\, and $\Omega$ be a compact metric space.
 Denote by  $Hom_1(C(\Omega), A)$ the set of all unital
 \hm s $\kappa$ from $C(\Omega)$ into $A.$
 }
 \end{df}

 \begin{df}\label{Dcdist-0}
 {\rm
 Let $O\subset \Omega$ be an open subset.
 Given $\kappa\in Hom_1(C(\Omega), A),$
 $[\kappa(f_O)]$ does not depend on the choice of
 $f_O.$  If $\kappa_1, \kappa_2\in Hom_{1}(C(\Omega), A),$
 define
 \beq\label{DNdist-1}
 D_c(\kappa_1, \kappa_2)=\sup\{\inf\{d>0: \kappa_1(f_O)\lesssim \kappa_2(f_{O_d})\} : O\subset {\Omega}, \,\,\,{\rm open }\}.
\eneq

 }
 \end{df}

 \begin{rem}\label{Rdcdist}
{\rm The definition of $D_{c}(\cdot, \cdot)$ is not symmetric as a priori.
However, when $A$ is a unital simple \CA\, with stale rank one,
then the definition in (\ref{DNdist-1}) is in fact symmetric.

Moreover, in general, $W(C(\Omega))$ is not determined by  open
subsets of $\Omega.$
In the definition it would required that $\kappa_1([f_O])\lesssim
\kappa_2([f_{O_d}]), $ for all $f\in M_{\infty}(C(\Omega))_+$ whose
supports in $O$ and all $f_{{O_d}}\in M_{\infty}(C(\Omega))_+$ with
supports in $O_d.$  We will not study these in full generality.

}
\end{rem}

\begin{df}\label{Dmetricsp}
{\rm Let $A$ be a unital \CA\, and let $\Omega$ be a compact metric
space. {For}  $\kappa_1,\,\kappa_2\in  { Hom_{1}}(C(\Omega), A)$
we write $\kappa_1\sim \kappa_2$ if
\beq\label{Dmetricsp-1}
[\kappa_1(f_O)]=[\kappa_2(f_O) ]\,\,\, {\rm in}\,\,\, W(A)
\eneq
for all open subsets $O\subset
\Omega.$ It is easy to see that ``$\sim$'' is an equivalence
relation. Put
$$
H_{c,1}(C(\Omega), A)={ Hom_{1}}(C(\Omega), A)_+/\sim.
$$

It follows from \cite{lR1} that if  {$\kappa_1\sim \kappa_2$ then they
induce the same semi-group \hm s from
$W(C(\Omega))$ into $W(A)$}
if the covering dimension of $\Omega$ is at most two and the second
co Homology  subsets $\check{H}^2(X)=\{0\}$ for each compact subsets
$X.$  This is particular true when  $\Omega$ is a compact subset of the {plane}
which is the primary concern of this research.

 Let
$\phi\in { Hom_{1}}(C(\Omega), A).$ {Then
${\rm ker}\phi=\{f\in C(\Omega): f|_X=0\}$ for some
compact subset $X\subset \Omega.$}
The
compact subset $X$ is called the spectrum of $\phi.$ Sometimes, we
denote $\phi\in { Hom}_{1}(C(\Omega),A)$ with spectrum $X$ by
$\phi_X$. }
\end{df}


%

\begin{df}\label{DNdist}
{\rm Let $A$ be a unital  \CA \,{and} let $X$ and $Y$ be two compact
subsets of a compact metric space $\Omega.$ Suppose that
$\phi: C(X)\to A$ and $\phi: C(Y)\to A$ are two unital
monomorphisms. We define $\phi_X: C(\Omega)\to A$ by $\phi\circ
\pi_X$ and $\phi_Y: C(Y)\to A$ by $\phi\circ \pi_Y,$  {where $\pi_X$ and $\pi_Y$ are the projections onto $X$ and $Y$, respectively.}
}
\end{df}



{Let $a, d\in A_+$ with $a,d\le 1.$
In the following as well as in the proof of \ref{hu1} below, we will write
{$a<<d.$}
if there are $b,\, c\in A_+$  with $0\le b, c\le 1$ such that
\beq\label{D<<1}
ab=a,\,\,\, b\lesssim  c \andeqn dc=c.
\eneq
}
The following follows immediately from \cite{RrUHF} (see also  \cite{Rl2}, and  4.2 and
4.3 of \cite{RW}).
 \begin{lem}\label{Lcuntzcanc2}
 Suppose {that} $A$ is a  unital $C^*$-algebra with stable rank
 one,  $0\le a, d\le 1$ are  elements in $A$ such that
 \beq\label{Lcc-1}
 {a<<d.}
 \eneq
 Then
 $1-d\lesssim 1-a.$
 \end{lem}
 \begin{proof}
 {There are $b,c\in A$ such that  $0\le b,c\le 1,$ $ab=a,$ $b\lesssim c$ and  $cd=c.$}
 Since
{$ab=a,$} for any $1/2>\ep>0,$ $af_\ep(b)=a.$ Since $b\lesssim c$,   by Proposition 2.4  of \cite{RrUHF},
 there is a unitary $u\in A$ such that
 \beq\label{Lcc-1n}
 u^*f_\ep(b)u\in Her(c).
 \eneq
 It follows that $u^*au=d^{1/2}u^*aud^{1/2}\le d.$
 Therefore
\beq\label{Lcc-2n} 1-d\le 1-u^*au=u^*(1-a)u.
 \eneq
 It follows that $1-d\lesssim 1-a.$
 \end{proof}

\begin{prop}\label{hu1}
Let $A$ be a unital simple \CA\, with stable rank one and let
$\Omega$ be a compact metric space.
{Then}  $(H_{c,1}(C(\Omega), A), D_c)$ is a metric space.
That is: for any $\phi_X,\phi_Y,\phi_Z\in H_{c,1}(C(\Omega),
A)$,
 \beq\label{hu1-1}
  D_c(\phi_X, \phi_Y)=0\Longleftrightarrow \phi_X\sim\phi_Y,
\eneq \beq\label{hu1-2}
  D_c(\phi_X,\phi_Y)= D_c(\phi_Y,\phi_X),
 \eneq
 \beq\label{hu1-3}
  D_c(\phi_X, \phi_Z)\le  D_c(\phi_X,
\phi_Y)+ D_c(\phi_Y, \phi_Z). \eneq
\end{prop}

\begin{proof}
{Let $d= D_c(\phi_X,\phi_Y).$  We will  show that $D_{c}(\phi_Y,
\phi_X)=d.$} Suppose $O\subset \Omega$ is an open subset. For any
$d>\ep>0,$ let
$$
F=\{t:{\rm dist} (t,O)\ge d+\ep\}\andeqn K=F_{d+\ep}.
$$

Define  $f,\,g\in C(\Omega)$ as the following:
\beq
f(t)=0\,\,{\rm if}\,\, t\not\in F_{{\ep/8}},\,\,0<f(t)<1\,\,{\rm
if}\,\,t\in F_{{\ep/8}}\setminus F \,\,{\rm and}\,\, f(t)=1\,\,{\rm
if}\,\,  t\in
F \andeqn\\
g(t)=0 \,\,{\rm if} \,\,t\not\in K, 0<g(t)<1\,\,{\rm if}\,\,
t\in K \setminus F_{d+\ep/2}\,\,{\rm and}\,\, g(t)=1,\,\,{\rm
if}\,\, t\in F_{d+\ep/2}.
\eneq
Since
{
\beq\label{hu1-n1n}
 &&\hspace{-0.6in}F_{\ep/8}\subset {\overline F_{\ep/8}}\subset F_{\ep/4}\subset
(F_{\ep/4})_{d+\ep/16}\subset F_{d+5\ep/16}\subset \overline{F_{d+5\ep/16}}
\subset F_{d+\ep/2}\subset \overline{F_{d+\ep/2}}\subset K\\
&&\hspace{0.5in}\andeqn
\phi_X(f_{F_{\ep/4}})\lesssim \phi_Y(f_{(F_{\ep/4})_{d+\ep/16}}),
\eneq
}
we have
 \beq\label{hu1-n2}
\phi_X({f})<<\phi_Y(g).
\eneq
 By  \ref{Lcuntzcanc2},
\beq\label{hu1-n4}
1-\phi_Y(g)\lesssim 1-\phi_X({f}).
\eneq
{Note that }
$$
O\subset \{t:{\rm dist}(t,F)\ge d+\ep\}=\{t:{\rm
dist}(t,F)<d+\ep\}^c\subset K^c.
$$
It follows that,  for any $t\in O$, $g(t)=0,$ which  implies
 \beq\label{hu1-n5-}
  f_O\le 1-g.
  \eneq

Hence
\beq\label{hu1-n5}
\phi_Y(f_O)\lesssim 1-\phi_Y(g).
\eneq
On
the other hand, if $f(t)\not=1$ or
$t\not\in F$, {{then}} ${\rm dist}(t,O)<d+\ep,$ or $t\in O_{d+\ep}$. Therefore  ${\rm Supp}(1-f)\subset O_{d+\ep}$, so we may assume that 
 \beq\label{hu1-n5+}
 1-{f}\le f_{O_{d+\ep}}
 \eneq
 by choosing a right representation of $f_{O_{d+\ep}}.$
 It follows that
 \beq\label{hu1-n5++}
1-\phi_X({f})\lesssim \phi_X(f_{O_{d+\ep}}).
\eneq

Combining (\ref{hu1-n5}),(\ref{hu1-n4}) and (\ref{hu1-n5++}), we obtain

\beq\label{hu1-n6}
\phi_Y(f_O)\lesssim \phi_X(f_{O_{d+\ep}})
\eneq
for all $\ep>0$ and for all open subsets $O\subset \Omega.$

This implies that
\beq
D_{c}(\phi_{{Y}}, \phi_{{X}})\le d.
\eneq
By 
symmetry, this proves that (\ref{hu1-2}) { holds.}

Suppose that $D_{c}(\phi_X, \phi_Y)=0.$ For any non-empty open subset $O$,
any $r>0$,
  recall that
$$
O_{-r}=\{t:{\rm dist}(t,O^c)>r\}.
$$
Then there is $\dt>0$ such that for any
$r\in(0,\dt),O_{-r}\not=\emptyset$. It is easy to show, for any
$S\subset \Omega,$ $ (S_{-r})_r\subset S$. For any $\ep\in (0,\dt)$,
then
$$
\phi_X(f_{O_{-\ep}})\lesssim \phi_Y(f_{(O_{-\ep})_\ep})\lesssim
\phi_Y(f_O).
$$
This shows that, for any $\sigma\in(0,\dt),$
\beq\label{hu1-n1}
f_{{\sigma}}(\phi_X(f_O))\lesssim \phi_Y(f_{O}).
\eneq
It follows from
2.4 of \cite{RrUHF} that
$$
\phi_X(f_O)\lesssim \phi_Y(f_{O}).
$$
Similarly, {by (\ref{hu1-2}), we have $D_c(\phi_Y,\phi_X)=0$, so }
$$
{\phi_Y(f_O)\lesssim \phi_X(f_{O}),}
$$
so $[\phi_X(O)]=[\phi_Y(O)]$ and (\ref{hu1-1}) holds.


Finally, suppose that
$$
 D_c(\phi_X, \phi_Y)=d_1, D_c(\phi_Y, \phi_Z)=d_2, D_c(\phi_X,
\phi_Z)=d_3.
$$
Then for any open $O$ and any $\ep>0$,
$$
\phi_X(f_{O})\lesssim \phi_Y(f_{O_{d_1+\ep}})\lesssim
\phi_Z(f_{O_{d_1+d_2+2\ep}}).
$$
Therefore
$$
d_3\le d_1+d_2
$$
and (\ref{hu1-3}) holds.

\end{proof}

\begin{prop}\label{Pcsameclose}
Let $A$ be a {simple unital  \CA\ with stable rank one,} and let $\Omega$ be a compact metric space.
 Then, for
any finite subset of projections ${\cal P}\subset C(\Omega),$ there
exists a $\dt>0$ satisfying the following:
if  $\phi, \psi: C(\Omega)\to A$ are two unital \hm s such that
 \beq\label{Pcs-1}
 D_c(\phi, \psi)<\dt,
\eneq
then
 \beq\label{Pcs-2}
 [\phi(p)]=[\psi(p)] \,\,\,{\rm in}\,\,\, W(A)\tforal \,\,\,
p\in {\cal P}.
\eneq
Moreover, 
\beq\label{Pcs-2+}
 [\phi(p)]=[\psi(p)] \,\,\,{\rm in}\,\,\, K_0(A)\tforal \,\,\,
p\in {\cal P}.
\eneq
\end{prop}

\begin{proof}
Without loss of generality, we may assume
that ${\cal P}$ consists of  mutually orthogonal  non-zero projections.
There are mutually disjoint clopen subsets $\{E_i:i=1,2...,m\}$ such that {${\cal
P}=\{p_i=\chi_{E_i},i=1,2,...,m\}$.} Let
$$
d=\min_{{1\le i\le m}}\{{\rm dist}(E_i,X\backslash E_i)\}.
$$
Now choose $0<\dt<d.$
Note that, for any $d>r >0,$
$(E_i)_r=E_i,$ $i=1,2,...,m.$
If $ D_c(\phi, \psi)<\dt$, then for any $i$ and $d>r>\dt,$
$$
{\phi}(f_{E_i})\lesssim {\psi(f_{(E_i)_r})=\psi(f_{E_i})}
$$
{which} implies
$$
[\phi(p_i)]\lesssim[\psi(p_i)].
$$
{By  symmetry}, we have
$$
{[\psi(p_i)]\lesssim[\phi(p_i)].}
$$
Thus we get $[\phi(p_i)]=[\psi(p_i)],$ {in $W(A)$}  $i=1,2,...,m$.


\end{proof}

\begin{lem}\label{Lappdcu}
Let $A$ be a unital simple \CA\, of (stable rank one) and let
$\Omega$ be a compact metric {space}.
 For any $\ep>0,$  there exists $\dt>0$ and a finite subset ${\cal F}\subset
 C({\Omega})$ satisfying the following:

Suppose that {$\phi, \psi, \rho: C(\Omega)\to A$}
are  {three}
unital \hm s  such that
\beq\label{Lappdcu-1}
 \|\phi(f)-\psi(f)\|<\dt\tforal f\in
{\cal F},
\eneq
then
\beq\label{Lappdcu-2}
|D_{c}(\phi, \rho)-
D_c(\psi, \rho)|<\ep
 \eneq
\end{lem}

\begin{proof}
Let $d= D_c(\phi, \rho)\ge 0.$
Let $\ep>0$ be given. 

Since $\Omega$ is compact, there are only finitely many open subsets
$\{O_{1},O_{2},...,O_{n}\}$ such that, for any open subset $G\subset
\Omega,$ there is an integer $i,$ \beq\label{Lappdcu-4}
{ G\subset O_i\subset G_\ep}.
 \eneq

 Let $g_i=f_{{(O_i)_\ep}},i=1,2,...,n.$

It follows from \cite{RrUHF} that there is $\dt>0$ satisfying the
following: If $h_1, h_2: C(\Omega)\to A$ are two unital \hm s such
that \beq\label{Lappdcu-8} \|h_1(g_i)-h_2(g_i)\|<\dt,
\,\,\, i=1,2,...,n, \eneq then \beq\label{Lappdcu-9}
f_{\ep}(h_1(g_i))\lesssim h_2(g_i)
\eneq $i=1,2,...,n.$

{Choose this $\dt.$}
Let ${\cal F}=\{g_i: i=1,2,...,n\}.$
Let $G\subset \Omega$ be  an
open subset. There is an integer $i$ such that
\beq\label{Lappdcu-10}
{ G\subset O_i\subset G_\ep}.
\eneq If
\beq\label{Lappdcu-11}
\|\phi(g_i)-\psi(g_i)\|<\dt,\,\,\, i=1,2,...,n, \eneq
then
\beq\label{Lappdcu-12}
f_{\ep}(\psi(g_i))\lesssim \phi(g_i).\\
\eneq
Since the support of $f_{\ep}(g_i)$
contains ${O_{i}}, $ $G\subset {O_{i}}$ {{and}} $O_{i}\subset
G_\ep,$ we obtain that
\beq\label{Lappdcu-16}
{\psi}(f_G)\lesssim
{\psi}(f_{\ep}(g_i))\lesssim{\phi}(g_i){\lesssim \phi(f_{G_{2\ep}})}\lesssim
{\rho}(f_{G_{{d+3\ep}}}),
 \eneq
Since this holds for all open sets $G\subset \Omega,$ we conclude
that \beq\label{Lappdcu-17}
D_{c}(\psi, \rho)\le
d+{3\ep}=D_{c}(\phi, \rho)+{3\ep}.
\eneq
{By  symmetry,}
\beq\label{Lappdcu-18}
 D_c(\phi,\rho)\le  D_c(\psi, \rho)+{3\ep}.
\eneq Lemma follows.

\end{proof}

\begin{df}\label{Depapp}
{\rm Let $A$ be a unital \CA\, and let $\Omega$ be a compact metric
space. Fix {$\kappa\in H_{c,1}(C(\Omega), A)$} and $\ep>0.$
Let $X$ be the spectrum of $\kappa.$ We say $\kappa$ has a {\it
finite $\ep$-approximation {in $H_{c, 1}(C(\Omega), A)$}}, if there is a finite subset  $F\subset
X$  and $\phi_F\in Hom_{1}(C(\Omega), A)_+$ with the spectrum
$F$ such that
\beq\label{Depapp-1}
D_{c}(\kappa ,
{\phi_F})<\ep.
\eneq
Note that $[\phi_F(f_O)]$ can be represented by a
projection $p\in A.$

Let $A$ be a unital separable simple \CA\, with real rank zero,
stable rank one and with weakly unperforated $K_0(A).$
{We show  that (see  {\ref{Lfepapp}} below), if
$\kappa\in H_{c,1}(C(\Omega), A)$ is induced by a \hm\, $h: C(\Omega)\to A,$}
then $\kappa$ has a finite {$\ep$-approximation} for any
$\ep>0.$
{Warning: a \hm\, $\phi$ has a finite { $\ep$-approximation}
in $H_{c, 1}(C(\Omega), A)$ does not imply that
it is close to a \hm s with finite spectrum.} }
\end{df}

\begin{df}\label{Ddt}
{\rm Let $A$ be a unital \CA\, with $T(A)\not=\emptyset.$ Let
$\Omega$ be a compact metric space, {and $X, Y\subset \Omega$ be two
compact subsets.} Let $\phi_X: C(X)\to A$ and $\psi_Y: C(Y)\to A$ be two
unital monomorphisms. Let $\pi_X: C(\Omega)\to C(X)$ and $\pi_Y:
C(\Omega)\to C(Y)$ be the quotient maps. Define $\phi_X=\phi\circ
\pi_X$ and $\psi_Y=\psi\circ \pi_Y.$ For each open subset $O\subset
\Omega,$ define \beq\label{Ddf-1-}
r_O(\phi_X,\psi_Y)=\inf\{r>0: d_\tau(\phi_X(f_O))\le d_\tau(\psi_Y(f_{O_r}))\tforal \tau\in T(A)\}\andeqn\\
r_O^+(\phi_X,\psi_Y)=\inf\{r>0: d_\tau(\phi_X(f_O))<
d_\tau(\psi_Y(f_{O_r}))\tforal \tau\in T(A)\}. \eneq

Define
\beq\label{Ddf-1}
D_T(\phi_X,
\psi_Y)=\sup\{r_O(\phi_X,\psi_Y): O \,\,\, {\rm open}\}.
\eneq

 Put
$$
a(\phi_X,\psi_Y)=\sup\{{\rm dist}(\zeta, X): \zeta\in Y\}\andeqn
b(\phi_X, \psi_Y)=\sup\{{\rm dist}(\xi, Y): \xi\in X\}.
$$
Define
\beq\label{Ddf-2}
D^T(\phi_X,\psi_Y)=\max\{a(\phi_X, \psi_Y),
\sup\{r_O^+(\phi_X,\psi_Y): O\,\,\, {\rm open\,\,\, and }\,\,\,
O\cap X\not=X\}\}.
 \eneq}
\end{df}
Note that, if $X\subset O,$ then $d_\tau(\phi_X(f_O))=1$ for all
$\tau\in T(A).$ Therefore
\beq\label{Ddf-3}
D_T(\phi_X, \psi_Y)\ge
a(\phi_X, \psi_Y).
\eneq


Since $X$ is compact, there is $\xi\in X$ such that $b(\phi_X,
\psi_Y)={\rm dist}(\xi, Y).$ If $\ep>0$ and $O(\xi, \ep)$ is the
open ball with center at $\xi$ and radius $\ep,$ then
\beq\label{Ddf-4}
r_{O(\xi, \ep)}{(\phi_X,\psi_Y)
}\ge b(\phi_X, \psi_Y)-\ep.
\eneq
 It follows that $D_T(\phi_X,\psi_Y)\ge {\max}\{a(\phi_X,\psi_Y),b(\phi_X,
 \psi_Y)\}=d_H(X,Y),$ where $d_H(X,Y)$ is the Hausdorff distance  between
 $X$ and $Y$.

\begin{lem}\label{d<d}
Let $A$ be a unital simple \CA\, with $T(A)\not=\emptyset$ and let
 $O\subset \Omega$ be an open set {with $O\cap X\not= X$}. If $r>D^T(\phi_X, {\psi}_Y),$
then \beq\label{d<d-1} \inf\{
d_\tau({\psi}_Y(f_{O_r}))-d_\tau(\phi_X(f_O)): \tau\in T(A)\}>0. \eneq
\end{lem}

\begin{proof}
Put $d=D^T(\phi_X, \psi_Y)$ and  $\eta=(1/4) (r-d)>0.$ Let $f_1\in
C(\Omega)$ with $0\le f_1\le 1$ such that $f_1(t)=1$ if $t\in O$ and
$f_1(t)=0$ if $t\not\in O{_{\eta}}.$ Let $f_2\in C(\Omega)$ with $0\le
f_2\le 1$ such that $f_2(t)=1$ if $t\in O_{d+2\eta}$ and $f_2(t)=0$
if $t\not\in O_r$.
Then  \beq\label{d<d-2} &&d_\tau(\psi_Y(f_{O_r}))\ge
{\tau}(\psi_Y(f_2))\ge d_\tau(\psi_Y(f_{O_{d+2\eta}}))\\
&>&d_\tau(\phi_X(f_{O_\eta}){)}\ge \tau(\phi_X(f_1))\ge
d_\tau(\phi_X(f_O)) \eneq for all $\tau\in T(A).$ Since $T(A)$ is
compact, we conclude that \beq\label{d<d-3}
\inf\{\tau(\psi_Y(f_2))-\tau(\phi_X(f_1)): \tau\in T(A)\}>0. \eneq
Therefore \beq\label{d<d-4}
&&\inf\{ d_\tau(\psi_Y(f_{O_r}))-d_\tau(\phi_X(f_O)): \tau\in T(A)\}\\
&\ge& \inf\{\tau(\psi_Y(f_2))-\tau(\phi_X(f_1)): \tau\in T(A)\}>0.
\eneq
\end{proof}

\begin{prop}\label{dtP}
Let $A$ be a unital \CA\, with $T(A)\not=\emptyset$. If $A$ has the
strict comparison for positive elements.  Let ${\phi_X, \phi_Y,
\phi_Z}$ be three unital homomorphisms from $C(\Omega)$ into $A$.
Then
\beq\label{dtP-4}
D_T(\phi_X, \phi_Y)\le D_c(\phi_X, \phi_Y)\le
D^T(\phi_X, \phi_Y),
\\\label{dtP-1} D_T(\phi_X,
\phi_Y)=D_T(\phi_Y,\phi_X),
\\\label{dtP-2}
D^T(\phi_X,\phi_Y)=D^T(\phi_Y,\phi_X),
\\\label{dtP-5}
D_T(\phi_X,\phi_Z)\le
D_T(\phi_X,\phi_Y)+D_T(\phi_Y,\phi_Z),\\\label{dtP-3}
D^T(\phi_X,\phi_Z)\le D^T(\phi_X,\phi_Y)+D^T(\phi_Y,\phi_Z).
\eneq If
$X$ or $Y$ is connected, then
\beq\label{dtP-6}
D_T(\phi_X,\phi_Y)=D_c(\phi_X,\phi_Y)=D^T(\phi_X,\phi_Y).
\eneq

\end{prop}
\begin{proof}

If $D_c(\phi_X,\phi_Y)<d$, then for any open $O,$
$$
\phi_X(f_O)\lesssim \phi_Y(f_{O_d}){.}
$$
{S}o for any $\tau\in
T(A)$,
$$
d_\tau(\phi_X(f_O))\le d_\tau(\phi_Y(f_{O_d})).
$$
This implies the first inequality of (\ref{dtP-4}).

If $D^T(\phi_X,\phi_Y)<d$, then for any open subset $O$ with $O\cap
X\not=X$, by  \ref{d<d},
\beq\label{d<d-4n}
\inf\{d_\tau(\phi_Y(f_{O_d}))-d_\tau(\phi_X(f_O)): \tau\in T(A)\}>0.
\eneq
Since $A$ has strict comparison for positive elements,  we have
$$\phi_X(f_O)\lesssim \phi_Y(f_{O_d}).$$  If $O\cap X=X$, then $X\subset O$.
Since $d>a(\phi_X,\psi_Y)$, {we have that } $Y\subset X_d\subset O_d$. It
follows  $\phi_X(f_O)\lesssim \phi_Y(f_{O_d})$. Thus we get the
second inequality of (\ref{dtP-4}).

 To show (\ref {dtP-1}), let $D_T(\phi_X,\phi_Y)\le d {.}$ {It suffices to} show that
$D_{T}(\phi_Y,\phi_X)\le d.$ Suppose $O\subset \Omega$ is an open
subset{.} 
For any $\ep>0,$ let
$$
F=\{t:{\rm dist}(t,O)\ge d+\ep\}\andeqn K=F_{d+{\ep}}.
$$

  Define $f,g\in
C(\Omega)$ with $f(t)=0$ if $t\not\in F_{\ep/2}$, $0<f(t)<1$, if
$t\in F_{\ep/2}\setminus F_{\ep/4}$, $f(t)=1$ if $t\in F_{\ep/4}$
and $g(t)=0$ if $t\not\in K$,$0<g(t)<1$, if $t\in K\setminus
F_{d+\ep/2}$, $g(t)=1$ if $t\in F_{d+\ep/2}$.
So, if $1-f(t)\not=0,$ then $t\not\in F.$ Hence $t\in O_{d+\ep}.$ 
Therefore, by choosing a right $f_{O_{d+\ep}},$ 
we may assume $1-f\le f_{O_{d+\ep}}$.

Then \beq\label{hu3-n1}
(F_{\ep/2})_{d+\ep/2}\subset F_{d+\ep}=K.
\eneq
By the definition,
 \beq\label{hu3-n2}
d_\tau(\phi_X(f))= d_\tau(\phi_X(f_{F_{\ep/2}}))\le
d_\tau(\phi_Y(f_{(F_{\ep/2})_{d+\ep/2}}))\le d_\tau(\phi_Y(f_{K}))
=d_\tau(\phi_Y(g)).
\eneq
{Thus}
\beq\label{hu3-n4}
1-d_\tau(\phi_Y(g))\le
1-d_\tau(\phi_X(f)).
\eneq
Since $f_O\le 1-g$,  we have 
\beq
\phi_Y(f_O)\lesssim
1-\phi_Y(g).
\eneq
Hence
 \beq\label{hu3-n9}
1-d_\tau(\phi_X(f))\le d_\tau(\phi_X(f_{O_{d+\ep}})).
\eneq
Thus {w}e have
\beq\label{hu3-n6}
d_\tau(\phi_Y(f_O))\le
d_\tau(\phi_X(f_{O_{d+\ep}}))
\eneq
for all $\ep>0$ and for all open
subsets $O\subset \Omega$. Thus we  get $D_T(y,x)\le d$.

 To show {(\ref{dtP-2})}, let $d>D^T(\phi_X, \phi_Y)$ {and let $O\subset \Omega.$}
{ If $O\cap Y\not=Y, $}
{ we have
\beq\label{hu3-n6+}
d_\tau(\phi_X(f_{F_{\ep/2}}))<
d_\tau(\phi_Y(f_{(F_{\ep/2})_{d+\ep/2}})).
\eneq
}
{ We will follow} the proof
 of (\ref{dtP-1}).
{By (\ref{hu3-n6+}), instead of $``\le " $ we will have $``<" $ in (\ref{hu3-n2}).}
 It follows, {as in the proof of  (\ref{dtP-1}),}
$$
\sup\{r_O^+(\phi_Y,\phi_X): O\,\,\, {\rm open\,\,\, and }\,\,\,
O\cap {Y}\not={Y}\}\le d.
$$

 Since $a(\phi_Y,\phi_X)\le
 D_T(\phi_Y,\phi_X)=D_T(\phi_X,\phi_Y)<d,$ then
 \beq
 D^T(\phi_Y,\phi_X)=\max\{a(\phi_Y, \phi_X),
\sup\{r_O^+(\phi_Y,\phi_X): O\,\,\, {\rm open\,\,\, and }\,\,\,
O\cap {Y}\not={Y}\}\}\le d.
\eneq
We get {(\ref{dtP-2})}.

{Now we turn to (\ref{dtP-5}).}
If $c_1=D_T(\phi_X,\phi_Y),c_2=D_T(\phi_Y,\phi_Z)$, 
$c_3=D^T(\phi_X, \phi_Y)$ and $c_4=D^T(\phi_X, \phi_Y),$ then for any $\ep>0$, any open set $O$, 
$$
{d_\tau(\phi_X(f_O))\le d_\tau(\phi_Y(f_{O_{c_1+\ep}}))\le d_\tau(\phi_Z(f_{O_{c_1+c_2+2\ep}}))}
$$
and
$$
{d_\tau(\phi_X(f_O))< d_\tau(\phi_Y(f_{O_{c_3+\ep}}))< d_\tau(\phi_Z(f_{O_{c_3+c_4+2\ep}}))}
$$

If
$$a(\phi_X,\phi_Y)=d_1,a(\phi_Y,\phi_Z)=d_2,$$
then  $Z\subset Y_{d_{{2}}}\subset (X_{d_1})_{d_2}\subset X_{d_1+d_2}$,
so
$$
a(\phi_X,\phi_Z)=\inf\{r>0:Z\subset X_r\}\le d_1+d_2=
a(\phi_X,\phi_Y)+a(\phi_Y,\phi_Z).
$$
{From these,}
 we obtain (\ref{dtP-5}) {and (\ref{dtP-3})}.

To show (\ref{dtP-6}), assume $X$ is connected {and}
$D_T(\phi_X,\phi_Y)=d.$ It is suffices  to  show $D^T(\phi_X,\phi_Y)\le
d$. For any open $O$ with $O\cap X\not=X {,}$ {s}ince $X$ is connected,
there is $\dt>0$ such that for any $0<\ep<\dt$, $O\cap
X\not=O_{\ep/2}\cap X{.}$ {S}o, since $A$ is simple, for any $\tau\in T(A)$,
\beq
\tau(\phi_X(f_O))<\tau(\phi_X(f_{O_{\ep/2}}))\le
\tau(\phi_Y(f_{O_{d+\ep}})).
\eneq

{On the other hand, s}ince $a(\phi_X,\phi_Y)\le D_T(\phi_X,\phi_Y)$, by definition,
$D^T(\phi_X,\phi_Y)\le d$.  {This} end{s}  the proof of
(\ref{dtP-6}).
\end{proof}

{ Note there exists } \CA\, $A,$ $D_T(\phi, \psi)=D_c(\phi, \psi),$ even {when} neither $X$ nor $Y$ are connected.

\begin{prop}\label{D_T=D_c}
Let $A$ be a unital simple \CA\,  {with stable rank one,} ${\rm ker}\rho_A(K_0(A))=\{0\}$ and
with strict comparison for positive elements. Suppose that $A$ has a unique tracial state.
Then
$$
D_T(\phi, \psi)=D_c(\phi, \psi).
$$

\end{prop}

\begin{proof}
{ Let $\phi, \psi: C(\Omega)\to A$ be two unital \hm s with
spectrum $X$ and $Y,$ respectively, and let $\tau$ be the unique
tracial state of $A.$
By  (\ref{dtP-4}) of   \ref{dtP}, it suffices to show that $D_T(\phi, \psi)\ge D_c(\phi, \psi).$
Let $d=D_T(\phi, \psi)$ and $d_1>d.$
For any open subset $O\subset \Omega,$}
\beq\label{DT=DC1}
d_\tau(\phi(f_O))\le d_\tau(\psi(f_{O_{d_1}})).
\eneq

{If inequality holds in (\ref{DT=DC1}), then by the strict
comparison,} \beq\label{DT=DC2} \phi(f_O)\lesssim \psi(f_{O_{d_1}}).
\eneq {Otherwise, {suppose} that equality holds in (\ref{DT=DC1})}.

If, for every $1/2>\ep>0,$
\beq\label{DT=DC3}
d_\tau
(f_\ep(\phi(f_O)))=d_\tau(\phi(f_\ep(f_O)))<d_\tau(\phi(f_O))=d_\tau({\psi}(f_{O_{d_1}}){)},
\eneq
by the strict comparison again,
\beq\label{DT=DC4}
{\phi (}f_\ep(f_O){)}\lesssim {\psi}(f_{O_{d_1}})\rforal 1/2>\ep>0.
\eneq
It
follows that $\phi(f_O)\lesssim \psi(f_{O_{d_1}}).$

{{Otherwise  there is  an $1/2>\ep>0$ such that 
\beq\label{DT=DC5}
d_\tau(f_\ep(\phi(f_O)))=d_\tau({\phi}(f_O)).
\eneq
Since $A$ is
simple, we conclude that, for  $\dt>0,$ 
\beq\label{DT=DC6}
O\cap X=O_{-\dt}{\cap} X=\{ \xi\in X: {\rm dist}(x, O^c)>\dt\},
\eneq 
which implies that $O\cap X$ is a clopen set relative to $X$.  Let
$q=\phi(f_O)$ {which is then  a projection} in this case.
We also have}, for any $d<d_2<d_1,$} 
\beq\label{DT=DC7}
d_\tau(\phi(f_O))=d_\tau(\phi_Z(f_{O_{d_2}}){)}=d_\tau(\phi_Y(f_{O_{d_1}}){)}.
\eneq
 {The same argument above shows that $O_{d_2}\cap
Y=O_{d_1}\cap Y.$} 
{It follows that $p=\psi(f_{O_{d_1}})$ is a
projection. 
Since ${\rm ker}\rho_A(K_0(A))=\{0\},$ and
$\tau(p)=\tau(q)$,} 
\beq\label{DT=DC8} \phi(f_O)=q\sim
p=\psi(f_{O_{d_1}}). \eneq
 {It follows that $D_c(\phi, \psi)\le
d_1$ for all $d_1>d.$ This completes the proof.}
\end{proof}

{\begin{df}\label{lowerb}For $\phi,\psi\in Hom(C(\Omega),A),$
 let
$$
\dt_T(\phi,\psi)=\inf\{r>0:d_\tau(\phi(f_{O}))\le
d_\tau(\psi(f_{O_r})),{\rm for\,\,all}\,\,O=O(\lambda,d)\,\,,\tau\in
T(A)\},
$$
$$
\dt_c(x,y)=\inf\{r>0:\phi(f_{O}){\lesssim} \psi(f_{O_r}),{\rm
for\,\,all}\,\,O=O(\lambda,d)\,\,,\}
$$
and
$$
d_T(\phi,\psi)=\max\{\dt_T(\phi,\psi),\dt_T(\psi,\phi)\},
$$
$$
d_c(\phi,\psi)=\max\{\dt_c(\phi,\psi),\dt_c(\psi,\phi)\}.
$$
For any normal elements $x,y\in A$,
 if $X={\rm sp}(x),Y={\rm
sp}(y),\Omega=X\cup Y$ and $\phi_X,\phi_Y$ are defined by
$\phi_X(f)=f(x),\phi_Y(f)=f(y)$ for any $f\in C(\Omega)$, {define}
$$
d_T(x,y)=d_T(\phi_X,\phi_Y), \,\,\,\,d_c(x,y)=d_c(\phi_X,\phi_Y).
$$
\end{df}
}

\begin{prop}\label{d=distance}
Let $\Omega$ be a compact metric space and $A$ be a unital simple
\CA.  Suppose that $\phi_X, \phi_Y, \phi_Z: C(\Omega)\to A$ are unital \hm s
with spectrum $X, Y, Z\subset \Omega,$ respectively.
Then
    \beq\label{Dsmall-14}
   { d_T}(\phi_X,\phi_Y)={d_T}(\phi_Y,\phi_X),\,\,{d_c}(\phi_X,\phi_Y)={d_c}(\phi_Y,\phi_X)\\\label{Dsmall-14+}
   { d_T}(\phi_X, \phi_Y)\ge d_H(X, Y),\\
      \label{Dsmall-10}
     { d_T}(\phi_X,\phi_Y)\le d_c(\phi_X,\phi_Y)\le D_c(\phi_X,\phi_Y),\\
     \label{Dsmall-12} { d_T}(\phi_X,\phi_Y)\le { d_T}(\phi_X,\phi_Z)+{ d_T}(\phi_Z,\phi_Y),\\
     \label{Dsmall-13} {d_c}(\phi_X,\phi_Y)\le {d_c}(\phi_X,\phi_Z)+{d_c}(\phi_Z,\phi_Y).
     \eneq
    \end{prop}

\begin{proof}
  The identities in (\ref{Dsmall-14}) follows from the definition.
 The inequality in (\ref{Dsmall-14+}) also follows from the definition immediately since $A$ is assumed to be simple.
 If ${d_c(\phi_X,\psi_Y)=r}$,
then for any $\xi\in \Omega$, any $d>0$, any $\ep>0$,
$$
\phi_X(f_{O(\xi,d)})\lesssim \phi_Y(f_{O(\xi,d+r+\ep)}).
$$
 It follows that, for any $\tau\in T(A)$,
$$
d_\tau(\phi_X(f_{O(\xi,d)}))\le d_\tau(\phi_Y(f_{O(\xi,d+r+\ep)})).
$$
This implies ${d_T(\phi_X,\psi_Y)\le d_c(\phi_X,\psi_Y)}$. It is
obvious that ${d_c}(\phi_X,\psi_Y)\le D_c(\phi_X,\psi_Y).$ So
(\ref{Dsmall-10}) holds. 
The
proofs of  (\ref {Dsmall-12}) and  (\ref{Dsmall-13}) are similar, we
show  (\ref{Dsmall-12}) only.

 If $d_T(\phi_X,\phi_Z)<d_1,d_T(\phi_Z,\phi_Y)<d_2$, then
for any $\xi\in\mathbb {C},$ any $d>0$, { any $\ep>0$,
$$
d_T(\phi_X(f_{O(\xi,d)}){)}\le d_T(\rho_Z(f_{O(\xi,d+d_1+\ep)}){)}\le
d_T(\psi_Y(f_{O(\xi,d+d_1+d_2+2\ep)}){)},
$$}
therefore ${d_T}(\phi_X,\psi_Y)\le d_1+d_2$, this implies (\ref
{Dsmall-12}) holds. 

\end{proof}

\begin{rem}\label{Rd=D}
{\rm There are examples such that ${d_T}(x, y)=D_c(x,y).$

Let $n\ge 4$ be an integer. Let $X=\{e^{k\pi i/n}: 0\le k\le n\}$
and $Y=rX=\{ re^{k\pi i/n}: 0\le k\le n\}$ for some $0<r<1.$ Let $A$
be any unital simple \CA\, with $T(A)\not=\emptyset$ which has $n$
mutually orthogonal non-zero projections $\{e_1,e_2,...,e_n\}$ such
that $\sum_{k=1}^n e_i=1.$ Define $x=\sum_{k=1}^n e^{(k-1)\pi i/n}
e_i$ and $y=\sum_{k=1}^n re^{(k-1)\pi i/n} e_i.$ Then one computes
that \beq\label{Rd=D-1} D_T(x, y)={1-r}={d_T}(x, y)=d_c(x,
y)=D_c(x,y). \eneq

On the other hand, of course, there are also examples that  $d_T(x,y)<D_T(x,y).$
Let  $\{e_1,e_2,e_3\}$ be mutually orthogonal and equivalent
projections {with $e_1+e_2+e_3=1$,}
$$
{x=-e_1+e_3,y=-ie_1+ie_3.}
$$
Then
$$
{d_T}(x,y)={d_c}(x,y)=1<\sqrt 2=D_T(x,y)=D_c(x,y).
$$
%

{{In the case that $A$ has stable rank one, strict  comparison and ${\rm ker}\rho_A=\{0\},$
{then}  $D_T(\cdot, \cdot)$ is a distance on $H_{c, 1}(C(\Omega), A).$
In general, because of the definition of $H_{c, 1}(C(\Omega), A),$
$D^T, d_T, d_c$ are not a distance on $H_{c, 1}(C(\Omega), A).$}}
}

\end{rem}

\section{Distance between unitary orbits of normal elements with trivial $K_1$}

Let $\Z^k$ be the direct sum of $n$ copies of the abelian group $\Z.$
Put
\beq\label{DZn+}
\Z^k_+=\{(n_1, n_2,..., n_k): n_j\ge 0: j=1,2,...,k\}.
\eneq
It is well known that  $(\Z^k, \Z^k_+)$ is an unperforated (partially) ordered group
with the Riesz interpolation property.
Let {$R\subset \{1,2,...,m\}\times \{1,2,...,n\}$} be a subset and let $A\subset \{1,2,...,m\}.$
Define
{$R_A\subset \{1,2,...,n\}$}  to be the subset of those $j's$ such that $(i,j)\in R,$ {for some }
$i\in A.$

\begin{lem}\label{Lmarr1}
If $\{a_i\}_{i=1}^m,\{b_i\}_{j=1}^n\subset  \Z^k_+$ with
$\sum_{i=1}^ma_i=\sum_{j=1}^nb_j$, {and}  $R\subset \{1,...,m\}\times
\{1,...,n\}$ satisfying:
for any $A\subset \{1,...,m\}$,
\beq\label{Lmarr1-1}
\sum_{i\in A}a_i\le \sum_{j\in R_A}b_j,
\eneq
then there are $\{c_{ij}\}\subset \Z^k_+$ such that
\beq\label{Lmarr1-2}
\sum_{j=1}^nc_{ij}=a_i,\sum_{i=1}^mc_{ij}=b_j,\rforal i,j
\eneq
and
\beq\label{Lamrr1-3}
c_{ij}=0\,\,\,{ unless}\,\,\, (i,j)\in R.
\eneq

\end{lem}

\begin{proof}
Write
\beq\label{Lmarr1-4}
a_i=(a_i(1), a_i(2),...,a_i(k))\andeqn\,\, b_j=(b_j(1),b_j(2),...,b_j(k)),
\eneq
$i=1,2,...,m$ and $j=1,2,...,n.$

It follows from  Lemma 1.2 of \cite{HN} that, for each $s$ ($s=1,2,...,k$),
there are $c_{i,j}(s)\in \Z_+$ such that
\beq\label{Lmarr1-5}
\sum_{j=1}^nc_{ij}(s)=a_i(s),\sum_{i=1}^mc_{ij}(s)=b_j(s),\rforal i,j.
\eneq
and
\beq\label{Lmarr1-6}
c_{ij}(s)=0 \,\,\,{\rm unless}\,\,\, (i,j)\in R.
\eneq
Define
\beq\label{Lmarr1-7}
c_{ij}=(c_{ij}(1), c_{ij}(2),...,c_{ij}(k)),\,\,\,i=1,2,...,m\andeqn j=1,2,...,n
\eneq
Note that
\beq\label{Lmarr1-8}
c_{ij}=0 \,\,\,{\rm unless}\,\,\, (i,j)\in R.
\eneq
We also have
\beq\label{Lmarr1-9}
\sum_{i=1}^m c_{ij}=a_i\andeqn \sum_{j=1}^n c_{ij}=b_j.
\eneq
\end{proof}

\begin{lem}\label{Lnmarr}
Let $(G, G_+)$ be a countable torsion free  unperforated  partially ordered
abelian group with the Riesz interpolation property.
If $\{a_i\}_{i=1}^m,\{b_i\}_{j=1}^n\subset  G_+$ with
$\sum_{i=1}^ma_i=\sum_{j=1}^nb_j$, {and}  $R\subset \{1,...,m\}\times
\{1,...,n\}$ satisfying:
for any $A\subset \{1,...,m\}$,
\beq\label{Tmarr-1}
\sum_{i\in A}a_i\le \sum_{j\in R_A}b_j,
\eneq
then there are $\{c_{ij}\}\subset G_+$ such that
\beq\label{Tmarr-2}
\sum_{j=1}^nc_{ij}=a_i,\sum_{i=1}^mc_{ij}=b_j,\rforal i,j
\eneq
and
\beq\label{Tmarr-3}
c_{ij}=0\,\,\,{\rm unless}\,\,\, (i,j)\in R.
\eneq

\end{lem}

\begin{proof}
Let $G$ be  a countable  unperforated ordered Riesz group. It
follows from \cite{EHS} that $G=\lim_{n\to\infty} (G_n, h_n)$ with
$G_n$ is order isomorphic to $\Z^{r(n)}$ (with positive cone
$\Z^{r(n)}_+$), where $h_n: G_n\to G_{n+1}.$ Denote by
{ $h_{n,
n+k}=h_{n+k}\circ h_{n+(k-1)}\circ\cdots \circ h_n: G_n\to G_{n+k},$}
$k=1,2,...,n=1,2,...,$ and denote by $h_{n, \infty}: G_n\to G$ the
\hm\, induced by the inductive limit system. Moreover,
$G_+=\cup_{n=1}^{\infty} h_{n, \infty}((G_n)_+).$ For each $A{\subset}
\{1,2,...,m\},$ {denote}  by 
\beq\label{Tmarr-6} g_A=\sum_{j\in
R_A}b_j-\sum_{i\in A}a_i \eneq Note that there are no more than
$2^m$ many $A's.$ There exists $n_1>0$ such that there are
\beq\label{Tmarr-7}
a_{i,k}, b_{j,k}\in (G_k)_+\andeqn\\
g_{A,k}\in (G_k)_+\tforal i, j \andeqn A\subset \{1,2,...,m\}
\eneq
such that, for $k_1>k\ge n_1,$
$h_{k,k_1}(a_{i,k})=a_{i,k_1},$ $h_{k,k_1}(b_{j,k})=b_{j,k_1},$
$h_{k,k_1}(g_{A,k})=g_{A,k_1},$
$h_{k,\infty}(a_{i,k})=a_i,$ $h_{k, \infty}(b_{j,k})=b_j$
and $h_{k,\infty}(g_A')=g_A.$
Note, since each $G_k$ is isomorphic to $\Z^{r(n)},$
there is  an integer $n_2>n_1$ such that
\beq\label{Tmarr-7+}
\sum_{i=1}^m h_{n_1, n_2}(a_{i,n_1})-\sum_{j=1}^nh_{n_1, n_2}(b_{j,n_1})=0\andeqn\\
g_{A,n_2}=\sum_{j\in R_A}h_{n_1,n_2}(b_{j,n_1})-\sum_{i\in A}h_{n_1, n_2}(a_{i,n_1}).
\eneq
Thus, we obtain, by applying \ref{Lmarr1}, $c_{i,j,n_2}\in (G_{n_2})_+,$
$(i,j)\in R$ such that
\beq\label{Tmarr-9}
\sum_{j=1}^n c_{i,j,n_2}=a_{i, n_2}\andeqn \sum_{i=1}^m c_{i,j,n_2} =b_{j,n_2}.
\eneq
Moreover,
\beq\label{Tmarr-10}
c_{i,j, n_2}=0\,\,\,{\rm unless}\,\,\,\,\, (i,j)\in R.
\eneq
Define $c_{i,j}=h_{n_2, \infty}(c_{i,j,n_2}).$
Then, $c_{i,j}\ge 0$ and by (\ref{Tmarr-9}) and (\ref{Tmarr-10}),
\beq\label{Tmarr-11}
\sum_{j=1}^n c_{i,j}=a_{i},\,\,\, \sum_{i=1}^m c_{i,j} =b_{j}\andeqn\\
c_{i,j}=0\,\,\,{\rm unless}\,\,\,\,\, (i,j)\in R.
\eneq
\end{proof}

\begin{lem}\label{TmarrG}
Let $(G, G_+)$ be a countable weakly  unperforated  partially ordered
abelian group with the Riesz interpolation property.
If $\{a_i\}_{i=1}^m,\{b_i\}_{j=1}^n\subset  G_+$ with
$\sum_{i=1}^ma_i=\sum_{j=1}^nb_j$, $R\subset \{1,...,m\}\times
\{1,...,n\}$ satisfying:
for any $A\subset \{1,...,m\}$,
\beq\label{Tmarr-1n}
\sum_{i\in A}a_i\le \sum_{j\in R_A}b_j,
\eneq
then there are $\{c_{ij}\}\subset G_+$ such that
\beq\label{Tmarr-2n}
\sum_{j=1}^nc_{ij}=a_i,\sum_{i=1}^mc_{ij}=b_j,\rforal i,j
\eneq
and
\beq\label{Tmarr-3n}
c_{ij}=0\,\,\,{\rm unless}\,\,\, (i,j)\in R.
\eneq

\end{lem}

\begin{proof}
It follows from \cite{Ell} that one may write
$$
0 \to T\to G\to G_0\to 0,
$$
where $G_0$ is a countable unperforated ordered group with the Riesz
interpolation property and $T$ is a countable abelian torsion group.
Moreover,  $g\in G_+\setminus \{0\}$ if and only if $d(g)\in
(G_0)_+,$ where $d: G\to G_0$ is the quotient map.
 Furthermore,
there exists a sequence of abelian groups $S_n$ and $T_n$ such that
$S_n$ is order isomorphic to $\Z^{r(n)}$ and $T_n=\Z/k(1,n)\Z\oplus
\Z/k(2,n)\Z\oplus\cdots \oplus\Z/k(t(n),n)\Z$ such that
$G=\lim_{n\to \infty}( S_n\oplus T_n,\imath_n),$ {where  $\imath_n:
S_n\oplus T_n \to S_{n+1}\oplus T_{n+1}$. Denote by} $\imath_{n, \infty}:
S_n\oplus T_n\to G.$  Note
$$
(S_n\oplus T_n)_+=\{(s,f): s\in (S_n)_+\setminus \{0\}\}\cup \{(0,0)\},\,\,\, n=1,2,...,
$$
and $G_+=\lim_{\infty} (S_n\oplus T_n)_+.$  Let $\pi_n': S_n\oplus T_n\to S_n$ and let $\pi_n'': S_n\oplus T_n\to T_n$ be
the projection maps. Let  $\imath_n': S_n\to S_{n+1}$ be defined
by $\imath_n'=\pi_n'\circ \imath_n|_{S_n}.$
Let $F_n=\Z^{t(n)}$
and
$\pi_n: F_n\to T_n$ be the  {quotient}  map.
Define $H_n=S_n\oplus F_n.$ Since $F_n$ is free, there is a \hm\, $j_n: F_n\to F_{n+1}$ such that
$$
\begin{array}{ccc}
F_n &\overset{j_n}{\longrightarrow} & F_{n+1}\\
\Big\downarrow_{\pi_n} & &\Big\downarrow_{\pi_{n+1}}\\
T_n&\overset{\imath_n}{\longrightarrow} & T_{n+1}
\end{array}
$$
commutes.  Since $S_n$ is free,
there is $h_n': S_n\to F_{n+1}$ such that
$$
\begin{array}{ccc}
S_n &\overset{\imath_n}{\longrightarrow} & S_{n+1}\oplus T_{n+1}\\
\Big\downarrow_{h_n'} & &\Big\downarrow_{\pi_{n+1}''}\\
F_{n+1}&\overset{\pi_{n+1}}{\longrightarrow} & T_{n+1}\\
\end{array}
$$
{commutes.}
Define $h_n: H_n\to H_{n+1}$ by
$$
h_n|_{S_n}=\imath_n'\oplus h_n'\andeqn h_n|_{F_n}=j_n,\,\,\, n=1,2,....
$$
Define $(H_n)_+=\{(s,f): s\in (S_n)_+\setminus \{0\}\}\cup
\{(0,0)\}.$ Let $H=\lim_{n\to\infty}(H_n, h_n)$ and let
$F=\lim_{n\to\infty}(F_n,h_n|_{F_n}).$ Define
$H_+=\cup_{n=1}^{\infty} h_{n, \infty}(H_n)_+,$ where $h_{n,
\infty}: H_n\to H$ is the \hm\, induced by the inductive limit
system. Define $d_1: H\to H/F.$ Then it is clear that $H/F$ is order
isomorphic to $G_0.$ Moreover, if $h\in H,$ {then} $h\in H_+$ if
and only if $d_1(h)\in (G_0)_+.$ Therefore $H$ is also a torsion
free weakly unperforated ordered group with Riesz interpolation
property.

Define $q_n: H_n\to S_n\oplus T_n$ by $q_n|_{S_n}={\rm id}_{S_n}$ and
$q_n|_{F_n}=\pi_n,$ $n=1,2,....$
One has the following commutative diagram:
$$
\begin{array}{ccc}
H_n &\overset{h_n}{\longrightarrow} & H_{n+1}\\
\Big\downarrow_{q_n} & &\Big\downarrow_{q_{n+1}}\\
S_n\oplus T_n&\overset{\imath_n}{\longrightarrow} & S_{n+1}\oplus T_{n+1}\\
\end{array}
$$
It induces a quotient map $\Pi: H\to G.$ It is an order preserving map.

Now let $a_i, b_j\in G_+,$  $i=1,2,...,m$ and $j=1,2,...,n,$ as described.
Let $a_i', b_j'\in S$ such that $\Pi(a_i')=a_i$ and $\Pi(b_j')=b_j,$ $i=1,2,...,m$ and $j=1,2,...,n.$
Then $a_i', b_j'\in H_+$ and
$$
\sum_{i\in A}  a_i' \le \sum_{j\in R_A} b_j'
$$
in {$H_+.$} Since $H$ satisfies the assumption in  \ref{Lnmarr}, there are
$c_{ij}'\in H_+$ such that
\beq\label{Tnmar-2}
\sum_{j=1}^m c_{ij}'=a_i',\,\,\, \sum_{i=1}^n c_{ij}'=b_j'\andeqn\\
c_{ij}'=0\,\,\, {\rm unless}\,\,\, (i,j)\in R.
\eneq
Put $c_{ij}=\Pi(c_{ij}'),$ $i=1,2,...,m$ and $j=1,2,...,n.$ Then
\beq\label{Tnmarr-3}
\sum_{j=1}^m c_{ij}=a_i,\,\,\,\sum_{i=1}^n c_{ij}=b_j\andeqn\\
c_{ij}=0\,\,\,{\rm unless}\,\,\, (i,j)\in R.
\eneq

\end{proof}

\begin{lem}\label{MLXY}
{
Let $A$ be a unital separable simple \CA\, with stable rank one and weakly unperforated $K_0(A)$ which has the Riesz interpolation property  and let $\Omega$ be a compact metric space.
Suppose that $\phi_X(f)=\sum_{i=1}^mf(\xi_i)p_i$ and $\phi_Y(f)=\sum_{j=1}^n f(\zeta_j)q_j$
for all $f\in C(\Omega),$ where $\{p_1, p_2,...,p_m\}$ and $\{q_1, q_2,...,q_n\}$ are two
sets of mutually orthogonal non-zero projections in $A$ such that
$\sum_{i=1}^m p_i=\sum_{j=1}^nq_j=1_A$ and $\xi_i, \zeta_j\in \Omega.$
Let $d>0.$
Then $D_c(\phi_X, \phi_Y) {\le}\, d$ if and only if, for any $\ep>0,$ there are  projections $p_{i,j}, q_{i,j}\in A$ such that
\beq\label{MLXY-1}
p_i=\sum_{j=1}^n p_{i,j},\,\,\, q_j=\sum_{i=1}^n q_{i,j},\\
{\rm [}p_{i,j}{\rm ]}={\rm [}q_{i,j}{\rm ]}\,\,\, {\rm in}\,\,\, K_0(A)\,\,\,\tand\\
\max\{{\rm dist}(\xi_i,\zeta_j): q_{i,j}\not=0\}<d+\ep.
\eneq
}
\end{lem}

\begin{proof}
Suppose $d=D_c(\phi_X, \phi_Y).$  Let $\ep>0.$ Put
$$
R=\{(i,j):{\rm dist}({\xi_i},{\zeta_j})\le d+\ep\}.
$$
For any $A\subset\{1,...,m\}$, put $O_A=\{{\xi_i}: i\in A\}$ and
$O_{R_A}=\{{\zeta_j}: j\in R_A\}.$ Then
$$
\sum_{i\in A}[p_i]=\Big[\sum_{i\in A}p_i\Big] ={[\phi_X(f_{O_A})] \le
[\phi_Y(f_{ {(O_A)_{d+\ep}}})] =[\phi_Y(f_{O_{R_A}})]}=\sum_{j\in R_A}[q_j].
$$
It follows  from \ref {TmarrG}, there  projection $r_{ij}$ such that
$$
[p_i]=\sum_{j=1}^n[r_{ij}],[q_j]=\sum_{i=1}^m[r_{ij}],i=1,2,...,m;j=1,2,...,n,
$$
where $r_{ij}=0$ unless $(i,j)\in R$.
 Then there are  $\{p_{ij}\}$ and $\{q_{ij}\}$ with
$[p_{ij}]=[q_{ij}]=[r_{ij}]$, satisfying
$$
p_i={\sum_{j=1}^n}p_{ij},q_j=\sum_{i=1}^mq_{ij},i=1,2,...,m;j=1,2,...,n.
$$
Then
$$
\max\{{\rm dist}({\xi_i, \zeta_j}):q_{ij}\not=0\}\le d+\ep.
$$
The converse is obvious.
\end{proof}

\begin{lem}\label{MLfsp}
Let $A$ be a unital separable simple \CA\, with stable rank one and weakly unperforated $K_0(A)$ which has the Riesz interpolation property and let $x, \, y\in A$ be two normal
elements with finite spectrum.
Then,
\beq\label{MLfd-1}
{\rm dist}({\cal U}(x), {\cal U}(y))\le D_{c}(x,y).
\eneq
\end{lem}
\begin{proof}
Let $\ep>0.$ Put $d=D_c(x,y)+\ep.$ We assume
$x=\sum_{i=1}^m\lambda_ip_i,y=\sum_{j=1}^n\mu_jq_j$, { where
}$\{p_i\}_{i=1}^m$ and $\{q_j\}_{j=1}^n$ are mutually orthogonal
projections with $\sum_{i=1}^mp_i$ and $\sum_{j=1}^nq_j=1$.
It follows from  \ref{MLXY} that
there are $\{p_{ij}\}$ and $\{q_{ij}\}$ with
$[p_{ij}]=[q_{ij}]=[r_{ij}]$, satisfying
$$
p_i={\sum_{j=1}^n}p_{ij},q_j=\sum_{i=1}^mq_{ij},i=1,2,...,m;j=1,2,...,n.
$$
Let $u\in U(A)$ with
$u^*p_{ij}u=q_{ij},i=1,2,...,m;j=1,2,...,{n}$.
$$
\|u^*xu-y\|=\|\sum_{i,j}(\lambda_i-\mu_j)q_{ij}\|\le
\max\{|\lambda_i-\mu_j|:q_{ij}\not=0\}\le d.
$$

\end{proof}

\begin{thm}\label{Toriginal}
Let $A$ be a unital simple separable \CA\, with real rank zero,
stable rank one and with weakly unperforated $K_0(A).$ Suppose that
$x$ and $y$ are two normal elements in $A$ with \beq\label{Torig-1}
[\lambda-x]=0\tand [\mu-y]=0\,\,\,in \,\,\, K_1(A) \eneq for all
$\lambda\not\in {\rm sp}(x)$ and for all $\mu\not\in {\rm sp}(y).$
Then \beq\label{Torig-2} {\rm dist}({\cal U}(x), {\cal U}(y))\le
D_c(x,y). \eneq
\end{thm}

\begin{proof}
Let $\ep>0.$ The assumption (\ref{Torig-1}) implies that
$\lambda-x\in {{\rm Inv}}_0(A)$ for all $\lambda\not\in {\rm sp}(x).$ It follows from \cite{Lnoldjfa} that, for any $\dt>0,$ there is a normal element $x_1\in A$  with finite spectrum in ${\rm sp}(x)$ such that
\beq\label{Torig-3}
\|x-x_1\|<\min\{\dt, \ep/8\}.
\eneq
It follows from  \ref{Lappdcu} that, for sufficiently small $\dt,$ we may assume that
\beq\label{Torig-4}
 D_c(x, x_1)<\ep/8.
\eneq
Exactly the same argument shows that there is normal element with finite spectrum
in ${\rm sp}(y)$
such that
\beq\label{Torig-5}
\|y-y_1\|<\ep/8\andeqn D_{c}(y, y_1)<\ep/8.
\eneq
It follows that
\beq\label{Torig-6}
D_{c}(x_1, y_1)< D_c(x, y)+\ep/4.
\eneq
Since $x_1$ and $y_1$ both have finite spectrum,  {it follows from
 \ref{MLfsp} that} there exists a unitary $u\in A$ such that
\beq\label{Torig-7}
\|u^*x_1u-y_1\|< D_c(x_1, y_1){+\ep/8}<D_c(x, y)+\ep/4+\ep/8=D_c(x, y)+3\ep/8.
\eneq
It follows that
\beq\label{Torig-8}
{\rm dist}({\cal U}(x), {\cal U}(y){)}&<&\ep/8+\|u^*x_1u-y_1\|+\ep/8\\
&<& D_{c}(x, y)+5\ep/8
\eneq
for all $\ep>0.$

\end{proof}

\begin{thm}\label{TAF}
Let $A$ be a unital separable AF-algebra and let $x, \, y\in A$ be two
normal elements.
Then
$$
{\rm dist}({\cal U}(x), {\cal U}(y))\le D_c(x, y).
$$
\end{thm}

\begin{proof}
Fix two normal elements  $x, \, y\in A.$ Let $\ep>0.$   Let
${\ep}>\eta>0.$ For any $\dt>0$ with $\dt<\eta/4,$ there exists
a finite dimensional \SCA\, $B\subset A$ with $1_B=1_A$ and two
elements $x',\, y'\in B$ such that \beq\label{TAF-1}
\|x-x'\|<\dt\andeqn \|y-y'\|<\dt. \eneq It follows from  {\cite{Lnalm}}
that, for some sufficiently small $\dt,$ there are normal elements
$x_1, y_1\in B$ such that \beq\label{TAF-2} \|x'-x_1\|<\eta/4\andeqn
\|y'-y_1\|<\eta/4. \eneq Thus
\beq\label{TAF-3}
\|x-x_1\|<\eta/4+\dt\andeqn \|y-y_1 \|<\eta/4+\dt.
\eneq
So,
by \ref{Lappdcu},
\beq\label{TAF-3n}
D_c(x,x_1)<\ep/4,\,\,\,D_c(y,
y_1)<\ep/4 \andeqn D_c(x_1,x_2)<D_c(x, y)+\ep/2.
\eneq
Now $x_1$ and
$y_1$ have finite spectra. We then complete the proof as in
 \ref{Toriginal}.
\end{proof}

\section{Maps with finite dimensional ranges}
\begin{df}\label{Ddce}
{\rm Let $\Omega$ be a compact metric space, $X,\, Y\subset \Omega$
be two compact subsets. Let $A$ be a unital \CA\,  and let $\phi_X:
C(\Omega)\to A$ and $\phi_Y: C(\Omega)\to A$ be two \hm s with
spectrum $X$ and $Y,$ respectively.  Let $\{h_n: C(\Omega)\to
A\}$ be a sequence of \hm s with finite dimensional ranges, i.e,
$h_n(f)=\sum_{i=1}^{k(n)} f(\xi(i,n))p_{n,i}$ for all $f\in
C(\Omega),$ where $\xi(i,n)\in X\cap Y$ and where
$\{p_{n,1},p_{n,2},...,p_{n, k(n)}\}$ is a sequence of finite
subsets of mutually orthogonal projections. We assume that, for any
$\ep>0,$ there is $N\ge 1$ such that $\{\xi(1,n), \xi(2,n),...,
\xi(k(n),n)\}$ is $\ep$-dense in $X\cap Y.$ Denote by
$e_n=\sum_{i=1}^{k(n)}p_{n,i},$ $n=1,2,....$ Suppose that
\beq\label{Ddce-1} \lim_{n\to\infty}\sup\{\tau(e_n): \tau\in
T(A)\}=0. \eneq

Let $\{u_n\}\subset U(A)$ be a sequence of unitaries, let
$q_n=u_n^*e_nu_n,$ let $\phi_{X,n}: C(\Omega)\to (1-e_n)A(1-e_n)$
and let $\phi_{Y,n}: C(\Omega)\to (1-q_n)A(1-q_n)$ be two unital \hm
s with spectrum  in $X$ and $Y,$ respectively. Suppose that
\beq\label{Ddce-2}
 \lim_{n\to\infty}D_c(\phi_X,
h_n+\phi_{X,n})=0\andeqn \lim_{n\to\infty}D_c(\phi_Y,{\rm Ad}\,
u_n\circ h_n+\phi_{Y, n})=0.
\eneq
 Define $\phi_{Y,n}'(f)=u_n\phi_{Y, n}(f)u_n^*$ for all $f\in C(\Omega).$
Then $\phi_{Y, n}': C(\Omega)\to (1-e_n)A(1-e_n).$
Denote by
$D_c(\phi_{X,n}, \phi_{Y,n}')$ the distance defined in  {\ref{Dcdist-0}} for $A=(1-e_n)A(1-e_n).$

Now we defined 
\beq\label{Ddce-3} D_c^e(\phi_X, \phi_Y)=\inf\{
\liminf_{n\to\infty}D_c(\phi_{X, n}, \phi_{Y,n}') \}, 
\eneq where
the infimum  is taken among all possible such  non-zero $h_n,
\phi_{X,n}, u_n$ and $\phi_{Y, n}.$

{\it It is important to note that
\beq\label{DC2015}
D_c^e(\phi_X, {\rm Ad}\, u\circ \phi_X)=0
\eneq
for any compact subset $X\subset \Omega$ and any unitary $u\in A.$ }
 Note since 
$$
{\rm Ad}u_n^*\circ ({\rm Ad}\,u_n\circ h_n+\phi_{Y,n})=h_n+{\rm Ad}\,u^*_n\circ \phi_{Y,n},
$$
when we consider  $D_c^e(\phi_X, \phi_Y)$ for a pair  $\phi_X$ and $\phi_Y,$ we may always 
replace   $\phi_{Y}$  by ${\rm Ad}\,u_n^*\circ \phi_{Y}$  by $\phi_{Y,n}$
and 
${\rm Ad}\,u_n^*\circ \phi_{Y,n}$ respectively, and write 
 \beq\label{Ddce-2+1}
 \lim_{n\to\infty}D_c(\phi_X,
h_n+\phi_{X,n})=0\andeqn \lim_{n\to\infty}D_c(\phi_Y,\, h_n+\phi_{Y, n})=0
\eneq
{\it This will be used throughout the paper without further notice.}

 If no such non-zero maps
$\{h_n\}$ exists, we  define $D_c^e(\phi_X,
\phi_Y)=D_c(\phi_X, \phi_Y).$ In particular, if $X\cap Y=\emptyset,$
$D_c^e(\phi_X, \phi_Y)=D_c(\phi_X, \phi_Y).$ When $X\cap
Y\not=\emptyset$ and $A$ is a unital separable simple \CA\, of real
rank zero, stable rank one and weakly unperforated $K_0(A),$
non-zero $\{h_n\}$ can {always}
 be found (see \ref{Rmdce}
below). In general, 
\beq\label{dc<dce} D_c(\phi_X, \phi_Y)\le
D_c^e(\phi_X, \phi_Y). 
\eneq 
From the definition, by  \ref{hu1},
\beq\label{dcesym} D_c^e(\phi_X, \phi_Y)=D_c^e(\phi_Y, \phi_X).
\eneq We also note that \beq\label{dcesup} D_c^e(\phi_X,
\phi_Y)=\inf\{\limsup_{n\to\infty}D_c(\phi_{X, n}, \phi_{Y,n}') \}.
\eneq To see this, take a subsequence $\{n_k\}$ such that
$$
\lim_{k\to\infty} D_c(\phi_{X, n_k},\phi_{Y,n_k}')=
\limsup_{n\to\infty}D_c(\phi_{X, n_k}, \phi_{Y,n_k}').
$$
Then
$$
\liminf_{k\to\infty} D_c(\phi_{X,n_k}, \phi_{Y, {n_k}}')=
\lim_{k\to\infty} D_c(\phi_{X, n_k},\phi_{Y,n_k}').
$$
Thus, by the definition of $D_c^e(\cdot,\, \cdot),$
(\ref{dcesup}) holds.
{Furthermore, there exists a subsequence $\{n_k\}$ such that
$$
\lim_{k\to \infty} D_c(\phi_{X, n_k},\phi_{Y, n_k})=D_c^e(\phi_X, \phi_Y).
$$
To see this, for each $k,$  there exists such sequence $\{h_{n,k}\},$ $\phi_{X, {n,k}}$ and $\phi_{Y,n,k}$  such that
$$
\liminf_{n\to\infty}D_c(\phi_{X, n,k}, \phi_{Y, n,k})\le D_c^e(\phi_X, \phi_Y)+1/k.
$$
Choose $\{n_k\}$ such that
$$
D_c(\phi_{X, n_k,k}, \phi_{Y,n_k,k})\le D_c^e(\phi_X, \phi_Y)+1/k.,
$$
Then
$$
\lim_{k\to\infty}D_c(\phi_{X, n_k,k}, \phi_{Y, n_k, k})=D_c^e(\phi_X, \phi_Y).
$$
}
}

\end{df}

\begin{rem}
{\rm
 Let $A$ be a unital simple \CA\, with real rank zero, stable rank one and
 weakly unperforated $K_0(A).$
 Suppose that  $\phi_X(f)=\sum_{i=1}^m f(\xi_i)p_i+\sum_{i=m+1}^{n_1} f(\zeta_i)p_i$
and $\phi_Y(f)=\sum_{i=1}^m f(\xi_i)q_i+\sum_{i=m+1}^{n_2}
f(\eta_i)q_i$ for all $f\in C(\Omega),$ where
$\{p_1,p_2,....,p_{n_1}\}$ and $\{q_1, q_2,...,q_{n_2}\}$ are two
sets of mutually orthogonal non-zero projections such that
$\sum_{i=1}^{n_1}p_i=1_{A}=\sum_{i=1}^{n_2}q_i,$ $\{\xi_1,
\xi_2,...,\xi_m\}=X\cap Y$ and $\zeta_i\in X\setminus Y$ and
$\eta_i\in Y\setminus X.$ Let $e_{i,n}\le p_i$  and $e_{i,n}'\le
q_i$ { be} non-zero projections such that
$[e_{1,n}]=[e_{i,n}]=[e_{i,n}']$ for all $i$ and $n,$ and
\beq\label{Ddce-4-1}
\lim_{n\to\infty}\sup\{\tau(\sum_{i=1}^me_{i,n}):\tau\in T(A)\}=0.
\eneq
Then
\beq\label{Ddce-4}
\phi_X(f)&=&\sum_{i=1}^m f(\xi_{i}{)}e_{i,n} +(\sum_{i=1}^m f(\xi_i)(p_i-e_{i,n})+\sum_{i=m+1}^{n_1}f(\zeta_i)p_i{)}\andeqn\\
\phi_Y(f)&=&\sum_{i=1}^m
f(\xi_i)e_{i,n}'+(\sum_{i=1}^mf(\xi_i)(q_i-e_{i,n}')+\sum_{i=m+1}^{n_2}f({\eta_i})q_i{)}
\eneq
for all $f\in C(\Omega).$ It makes sense  that  one insists
that $e_{i,n}$ pairs with $e_{i,n}'$ and the rest of them pairs
according to the distance $D_c$ defined in  {\ref{Dcdist-0}}. After all,
$e_{i,n}$ and $e_{i,n}'$ corresponds to the same point $\xi_i\in
X\cap Y.$
}
\end{rem}

\begin{prop}\label{Pclosem}
Let {$\Omega$ be a compact metric space and $X\subset\Omega.$}
Let $\Delta: (0,1)\to (0,1)$ be an increasing function (with
$\lim_{r\to 0}\Delta(r)=0$).
For any $\ep>0,$ let $r_0=\ep/16.$
 There is a finite subset of mutually
orthogonal projections $\{f_1, f_2,...,f_n\}\subset C(\overline{X_{\ep/64}}),$ a finite
subset ${\cal H}\subset C(\overline{X_{\ep/64}})_+$ and  $\sigma>0$  satisfying the following:
Suppose that $A$ is a unital simple \CA\, with stable rank one and with strict comparison for positive elements and suppose that ${\phi, \psi}: C(\overline{X_{\ep/64}})\to A$ are  two \hm s such that
\beq\label{Pclose-1}
\phi_{*0}([f_i])=\psi_{*0}([f_i]),\,\,\,i=1,2,...,n,\\\label{Pclose-1+}
|\tau\circ \phi(g)-\tau\circ \psi({g})|<\sigma\tforal g\in {\cal H}\tand\\\label{Pclose-1++}
\mu_{\tau\circ \phi}(O)\ge \Delta(r) \eneq for all open balls $O\in
\overline{X_{\ep/64}}$ with radius $r>r_0$ and for all $\tau\in
T(A),$ then \beq\label{Pclose-2}
 D_c({\phi, \psi})<\ep.
\eneq

Moreover, $\sigma$ can be chosen to be
\beq\label{Pclose-3}
\sigma=(1/4)\Delta(\ep/16).
\eneq
\end{prop}

\begin{proof}

Let $\ep>0.$ Let $\Omega_1=\overline{X_{\ep/64}}.$ Since $\Omega_1$
is compact, there are $\xi_1, \xi_2,...,\xi_{m_1}\in \Omega_1$ such
that $\cup_{i=1}^{m_1}O(\xi_i, \ep/16)\supset \Omega_1.$ Let
$G_1,G_2,...,G_K$ be all possible finite union of those $O(\xi_i,
\ep/16)$'s. If $O\subset \Omega_1$ is any open subset, there is
$G_j$ such that \beq\label{Ltpclo-4} O\subset G_j\andeqn d_H(O,
G_j)\le \ep/8. \eneq Without loss of generality, we may assume that
$(G_1)_{5\ep/64}, (G_2)_{5\ep/64},..., (G_{K_0})_{5\ep/64}$ are
clopen sets and $ (G_{K_0+1})_{5\ep/64},
(G_{K_0+2})_{5\ep/64},...,(G_K)_{5\ep/64}$ are not closed.
Therefore, for $j>K_0,$ there is $\zeta'_j\in \Omega_1$ such that
{ $\zeta'_j\not\in (G_j)_{5\ep/64}$ and ${\rm dist}(\zeta'_j,
G_j)=5\ep/{64}.$  This implies there exists $\zeta_j\in X$ such that
$4\ep/64<{\rm dist}(\zeta_j,G_j)<6\ep/64$. }There is $\eta>0$ such
that $\zeta_j\not\in {(G_j)_{\eta+\ep/16}},$ $K_0<j\le K.$ 

Note that

 \beq\label{Ltplo-4+}
O({\zeta_j}, \ep/16)\subset (G_j)_{{10}\ep/{64}}\setminus
(G_j)_{\eta}. \eneq Therefore

\beq\label{Ltpclo-5}
\inf\{d_\tau({\phi}(f_{(G_j)_{{10}\ep/64}})-d_\tau({\phi}(f_{G_j}))):
\tau\in T(A)\}\ge \Delta(\ep/16).
\eneq
Let $S_1, S_2,...,S_n$ be  clopen subsets of $\Omega_1$
such that $S_j=(G_j)_{5\ep/64},$ $j=1,2,...,K_0,$ and $n\ge K_0.$
Let $f_j$ be the characteristic function of $S_j$ $ j=1,2,...,n.$
Let ${\cal F}\subset C(\Omega_1)_+$ be the finite subset which contains
$g_{j},$
$j=1,2,...,K,$ where
$0\le {g}_{j}\le 1,$
$g_{j}(t)=1$ on $(G_j)_{{{10}\ep/64}}$ and $g_{j}(t)=0$ if
$t\not\in (G_j)_{\ep/{4}}.$

Choose $0<\sigma=(1/4)\Delta(\ep/16).$

Now suppose that $\psi: C(\Omega_1)\to A$ is a unital \hm s which satisfies the assumption
(\ref{Pclose-1}), (\ref{Pclose-1+}) and (\ref{Pclose-1++}).

Let $O\subset \Omega_1$ be an open subset. Then we may assume
that (\ref{Ltpclo-4}) holds. In particular,

 \beq\label{Ltpclo-9}
(G_j)_{\ep/{4}}\subset O_{\ep}.
\eneq

If $1\le
j\le K_0,$
{by} the assumption {(\ref{Pclose-1}),}
\beq\label{Ltpclo-9+}
[\phi(f_{(G_j)_{5\ep/64}})]=[\psi(f_{(G_j)_{5\ep/64}})],\,\,\,j=1,2,...,K_0.
\eneq
Therefore
\beq\label{Ltpclo-9++}
\phi(f_O)\lesssim
\phi(f_{(G_j)_{5\ep/64}})\lesssim \psi(f_{(G_j)_{5\ep/64}}) \lesssim
\psi(f_{(G_j)_{\ep/{4}}})\lesssim  \psi(f_{O_\ep}).
\eneq

If $K_0<j\le K,$
%

\beq\label{Ltpclo-10}
&&d_\tau(\psi(f_{O_\ep}))\ge
\tau(\psi(g_{j}))>\tau( \phi(g_{j}))-\sigma \\
&&\ge d_\tau(\phi(g_{(G_j)_{{{10}\ep/64
}}}))-\sigma\\
&&> d_\tau(\phi(g_{G_j}))\ge d_\tau(\phi(f_O))
\eneq
for all
$\tau\in T(A).$
Combining (\ref{Ltpclo-9++}) and (\ref{Ltpclo-10}), since $A$ has the strict comparison for positive elements, we obtain
\beq\label{Ltpclo-11}
D_c(\phi, \psi)<\ep.
\eneq

\end{proof}

\begin{cor}\label{Cabs}
Let $\ep>0.$
 Let $X$ be a compact subset of the plane.
Suppose that $X=\sqcup_{i=1}^n S_i,$ where each $S_i$ is an $\ep/64$-connected component, $i=1,2,...,n,$
suppose that $A$ is a unital simple \CA\, of stable rank one, real rank zero and
weakly unperforated $K_0(A).$
  Suppose that $\{\lambda_1, \lambda_2,...,\lambda_{m_0}\},$  $\{\mu_1, \mu_2,...,\mu_{m_1}\}$  and $\{\zeta_1, \zeta_2,...,\zeta_{m_2}\}$ are  $\ep/16$-dense in $X,$
and suppose that $\{e_{0,1},e_{0,2},...,e_{0,m_0}\},$
$\{e_{1,1},e_{1,2},...,e_{1,m_1}\}$ and $\{e_{2,1},e_{2,2},...,e_{2,m_2}\}$ are mutually
orthogonal non-zero projections in $A$ such that
\beq\label{Cabs-1}
P=\sum_{j=1}^{m_1}e_{1, j}=\sum_{j=1}^{m_2}e_{2, j},\\
e_{0,j}\in (1-P)A(1-P),\,\,\, j=1,2,...,m_0,\\
\label{Cabs-3}8\tau(P)< \tau(e_{0,j})\tforal \tau\in T(A),\,\,\, j=1,2,...,m_0\andeqn\\
\phi_{*0}([\chi_{S_i}])=\psi_{*0}([\chi_{S_i}]),\,\,\,i=1,2,...,n,
\eneq where
$\phi(f)=\sum_{j=1}^{m_0}f(\lambda_j)e_{0,j}+\sum_{j=1}^{m_1}f(\mu_j)e_{1,j}$
and
$\psi(f)=\sum_{j=1}^{m_0}f(\lambda_j)e_{0,j}+\sum_{j=1}^{m_2}f(\zeta_j)e_{2,j}.$
Then there is a unitary $u\in A$ such that \beq\label{Cabs-2}
\|u^*\phi(z)u-\psi(z)\|<\ep, \eneq where $z\in C(X)$ is the identity
function on $X.$
\end{cor}

\begin{proof}
{The}  proof is virtually  contained in that of  \ref{Pclosem}.
  We can keep all notations and proof of  \ref{Pclosem}
up to the definition of $\sigma.$ 

We define
$$
\sigma=\min\{\inf\{\tau(e_{0,j}): \tau\in T(A)\}: j=1,2,...,m_0\}.
$$
Since $A$ is unital and simple, $\sigma>0.$
{By the proof of \ref{Pclosem}, for any $O$, there is  $G_j$ such that 
$$
O\subset G_j\subset (G_j)_{\ep/8}\subset (G_j)_{5\ep/64}=S_j.$$
If $O_{\ep/16}=S_j$,} then
\beq\label{Cabs-10}
\psi(f_O)\lesssim \psi({\chi_{S_j}})\lesssim \phi({\chi_{S_j}})\lesssim \phi(f_{O_{\ep/16}}).
\eneq
In general,
\beq\label{Cabs-11}
d_\tau(\psi(f_O))<d_\tau(\phi(f_O))+(1/4)\sigma
\eneq
for all $\tau\in T(A).$
If $O_{\ep/16}\not=S_j,$ there is $\xi\not\in O_{\ep/16}$ such that ${\rm dist}(\xi, O_{\ep/16})<\ep/64.$
  There is $\lambda_j$ such that
\beq\label{Cabs-12}
{\rm dist}(\xi, \lambda_j)<\ep/16.
\eneq
Then
\beq\label{Cabs-13} \lambda_j\not\in O\andeqn \lambda_j \in
O_{\ep/8}.
\eneq
Then, by (\ref{Cabs-11}),
\beq\label{Cabs-12n}
d_\tau(\psi(f_O)) &\le &d_\tau(\phi(f_O))+(1/4)\sigma\\
&<& d_\tau(\phi(f_O))+d_\tau(\phi(O(\lambda_j, \ep/16)))\le
d_\tau(\phi(f_{O_{\ep}})). \eneq It follows that
\beq\label{Cabs-13n}
\psi(f_O)\lesssim \phi(f_{O_{\ep}}).
\eneq
This holds for any open
set $O\subset X.$ Therefore \beq\label{Cabs-14} D_c(\psi(z),
\phi(z))<\ep. \eneq The corollary then follows from  \ref{MLfsp}.
\end{proof}

{The following follows from Proposition 11.1 of \cite{LnTAMS12} immediately.}

\begin{prop}\label{PDelta}
Let $A$ be a unital simple \CA\, with $T(A)\not=\emptyset$ and
let $X$ be a compact metric space.
Suppose that $\phi: C(X)\to A$ is a unital monomorphism.
Then there is an increasing function $\Delta: (0,1)\to (0,1)$ with
$\lim_{r\to 0} \Delta(r)=0$ such that
\beq\label{PDelta-1}
\mu_{\tau\circ \phi}(O)\ge \Delta(r)\tforal \tau\in T(A)
\eneq
where $O$ is an open ball of $X$ with radius $r\in (0,1).$
\end{prop}

\begin{cor}\label{Ltp=close}
Let $A$ be a unital simple \CA\, with stable rank one and with strict comparison for positive elements, let $\Omega$ be a compact metric space, let $X\subset \Omega$ be a compact subset and let $\phi_X\to A$ be a unital \hm\, with spectrum $X.$
For any $\ep>0,$ there is $\dt>0,$ a finite set of
clopen subsets $S_1,S_2,...,S_k$  in $\overline {X_{\ep/64}}$ and a finite subset ${\cal F}\subset C(\overline{X_{\ep/64}})_+$ such that
for any unital \hm\, $\phi_Y: C(\overline{X_{\ep/64}})\to A$ with the property
that
\beq\label{Ltpclo-1}
|\tau(\phi_X(f))-\tau(\phi_Y(f))|<\dt\tforal f\in {\cal F}\tand\\\label{Ltpclo-1+}
{\rm [}\phi_X(\chi_{S_i}){\rm ]}={\rm [}\phi_Y(\chi_{S_i}){\rm ]},\,\,\,i=1,2,...,k,
\eneq
then
\beq\label{Ltpclo-2}
D_c(\phi_X, \phi_Y)<\ep.
\eneq
\end{cor}

\begin{prop}\label{Pdced00}
Let $\Omega$ be a compact metric space, let $X\subset \Omega$ be a
compact subset, let $A$ be a unital simple \CA\, of stable rank one
and with strict comparison for positive elements and let $\phi_X:
C(\Omega)\to A$  be a unital \hm\, with spectrum $X.$ Suppose that
${\{}h_n: C(\Omega)\to A{\}}$ is a sequence of unital \hm s
such that
\beq\label{Pdced00-1}
\lim_{n\to\infty} D_c({\phi_X}, h_n)=0.
\eneq
Then
\beq\label{Pdced00-2}
\lim_{n\to\infty}D_c^e({\phi_X},
h_n)=0.
\eneq
\end{prop}

\begin{proof}
Fix $\phi_X.$ Let $\ep>0.$   {Without loss of generality, we may assume that
$\Omega=\overline{X_{\ep/64}}.$ }
Let $S_1, S_2,...,S_k$ be a finite set of
clopen subsets of $\Omega,$
${\cal F}\subset C(\Omega)_+$ be a
finite subset and $\dt>0$ be required by  \ref{Ltp=close} for
$\phi_X$ and $\ep.$  Let $Y_n$ be the spectrum of $h_n.$
By (\ref{Pdced00-1}), we assume that $Y_n\subset \overline{X_{\ep/64}}$
for all $n,$ without loss of generality. Furthermore, we may further assume
that $X\cap Y_n\not=\emptyset$ for all $n.$
 Suppose that $\{e_n\}\subset A$ is a sequence of projections, $\{\phi_{0,n}: C(\Omega) \to
e_nAe_n\}$ is a sequence of unital \hm s with { finite spectrum that is 
$\ep_n$-dense in} ${X\cap Y_n}$ and with
$\lim_{n\to\infty}\ep_n=0,$
 and two sequences of unital \hm s  $\{\phi_{1, n}, \phi_{2, n}: C(\Omega)\to  (1-e_n)A(1-e_n)\}$
  such that the spectrum of $\phi_{1,n}$ is in $X,$ the spectrum of $\phi_{2, n}$ is in ${X\cap Y_n},n=1,2,...$ and
\beq\label{Ccuapp-2}
\lim_{n\to\infty}\sup\{\tau(e_n):\tau\in T(A)\}=0,\\\label{Ccuapp-2+1}
\lim_{n\to\infty}D_c(\phi_X, \phi_{0,n}+\phi_{1,n})=0\andeqn\\\label{Ccuapp-2+2}
\lim_{n\to\infty}D_c(h_n, \phi_{0,n}+\phi_{2,n})=0.
\eneq
By  \ref{Pcsameclose}, we may assume that,
for all $n\ge 1,$
\beq\label{Ccuapp-3}
[\phi_X(\chi_{S_i})]=[\phi_{0,n}(\chi_{S_i})]+[\phi_{1, n}(\chi_{S_i})]=[\phi_{0,n}(\chi_{S_i})]+[\phi_{2, n}(\chi_{S_i})],
\eneq
$j=1,2,...,k$ and  $n=1,2,....$
Since $A$ has stable rank one, this implies that
\beq\label{Ccuapp-4}
[\phi_{1, n}(\chi_{S_i})]=[\phi_{2, n}(\chi_{S_i})],\,\,\,j=1,2,...,k\andeqn n=1,2,....
\eneq
By (\ref{Ccuapp-2}),
there exists $N\ge 1$ such that, for all $n\ge N,$
\beq\label{Ccuapp-5}
\sup\{\tau(e_n):\tau\in T(A)\}<\dt.
\eneq
It follows from (\ref{Ccuapp-2+1}), (\ref{Ccuapp-2+2}) and (\ref{Pdced00-1}) that
\beq\label{Ccuapp-6}
\sup\{|\tau(\phi_{1,n}(f))-\tau(\phi_{2,n}(f))|:\tau\in T(A)\}<\dt\tforal f\in {\cal F}.
\eneq
It follows from  \ref{Ltp=close} that
\beq\label{Ccuapp-7}
D_c(\phi_{1,n}, \phi_{2,n})<\ep\tforal n\ge N.
\eneq
It follows that
\beq\label{Ccuapp-8}
\limsup_{n\to\infty}D_c^e(\phi_X, h_n)=0.
\eneq

\end{proof}
{In fact, we proved the following:}

\begin{cor}\label{CPc0}
With the same assumption {above}, for a fixed unital \hm\, $\phi: C(\Omega)\to A,$
and for any $\ep>0,$ there exists $\dt>0$ such that
if $\psi: C(\Omega)\to A$ is another unital \hm\, such that
$D_c(\phi, \psi)<\dt,$ then
\beq\label{CPc0-1}
D_c^e(\phi, \psi)<\ep.
\eneq
\end{cor}

{Since for any $\phi$, $D_c(\phi,\phi)=0$,  we have }
\beq\label{Pdced0-1}
D_c^e(\phi, \phi)=0.
\eneq

 We also have the following:
\begin{prop}\label{dce<DT}
Let $A$ be a unital simple \CA\, with strict comparison for positive elements. Then
\beq\label{dce<DT-1}
D_c^e(\phi_X, \phi_Y)\le D^T(\phi_X, \phi_Y).
\eneq
\end{prop}
\begin{proof}

By (\ref{dtP-1}) in \ref{dtP} and (\ref{dc<dce}), we may assume $X\cap Y\not=\emptyset$ and let $\Omega=X\cup Y$.
{Suppose that $\{e_n\}\subset A$ is a sequence of projections,
$h_n: C(\Omega) \to e_nAe_n$ is a sequence of unital \hm s with
spectrum being $\ep_n$-dense in $X\cap Y$ and with
$\lim_{n\to\infty}\ep_n=0,$
 and two sequences of unital \hm s  $\phi_{X_n}, \phi_{Y_n}: C(\Omega)\to  (1-e_n)A(1-e_n)$
  such that  $X_n\subset X,$ 
  $n=1,2,...$ and
\beq\label{dceDT-2} \lim_{n\to\infty}\sup\{\tau(e_n):\tau\in
T(A)\}=0,\\\label{dceDT-3} \lim_{n\to\infty}D_c(\phi_X,
\phi_{X_n}+h_{n})=0\andeqn\\\label{dceDT-4}
\lim_{n\to\infty}D_c(\phi_{Y}, \phi_{Y_n}+h_{n})=0 \eneq}
(see the lines around  (\ref{Ddce-2+1})).

{For any fixed $\ep>0$, let $r_0=D^T(\phi_X,\phi_Y)$. We will
show $D_c^e(\phi_X,\phi_Y)<r_0+3\ep$. It is suffices to show there
exists integer $N$ such that for all $n>N$, \beq\label{dceDT-5
}D_c(\phi_{X_n},\phi_{Y_n})<r_0+3\ep.\eneq}

{There are finitely many  open subsets $\{O_k:k=1,2,...,K\}$ of $\Omega$ such that for
any open set $O\subset \Omega$, there exists $O_k$ such that \beq
\label{dceDT-6} O\subset O_k\subset O_{\ep}.\eneq We may assume that
$$
O_k\cap X\not= X,k=1,2,...,J;O_k\cap X=X,k=J+1,J+2,...,K.$$ Define
\beq\label{dceDT-1}
\dt=\inf\{d_\tau(\phi_Y(f_{(O_k)_{r_0+\ep}}))-d_\tau(\phi_X(f_{O_k})):\tau\in
T(A),k=1,2,...,J\}.\eneq By  \ref{d<d}, $\dt>0$.}

{Let $N$ be large such that for any $n>N$,\beq
\sup\{\tau(e_n):\tau\in T(A)\}<\frac{\dt}2\eneq and \beq
D_c(\phi_X,\phi_{X_n}+h_n)<\ep,\,\,\,D_c(\phi_Y,\phi_{Y_n}+h_n)<\ep.\eneq
}
{Since $D_T(\phi_X,\phi_{X_n}+h_n)\le D_c(\phi_X,\phi_{X_n}+h_n)<\ep$,
 for any open subset $O$ with $O_\ep\cap X\not=X$,
 we
 have
 \beq\label{dceDT-7} d_\tau(\phi_{X_n}(f_O))\le
 d_\tau(\phi_{X_n}(f_O))+d_\tau(h_n(f_O))\le
 d_\tau(\phi_X(f_{O_\ep})).\eneq}

{Suppose $O_\ep\subset O_k\subset (O_\ep)_{\ep}$, then $k\le J$
and by (\ref{dceDT-1}),  \beq\label{dceDT-8} d_\tau(\phi_X(f_{O_\ep}))\le
 d_\tau(\phi_X(f_{O_k}))<d_\tau(\phi_Y(f_{(O_k)_{r_0+\ep}}))-\frac\dt
2.\eneq}
Using the fact that $D_c(\phi_X,\phi_{X_n}+h_n)<\ep,$ 
{we also  have 
\beq\label{dceDT-9}
d_\tau(\phi_Y(f_{(O_k)_{r_0+\ep}}))\le
d_\tau(\phi_{Y_n}(f_{(O_k)_{r_0+2\ep}}))+d_\tau(h_n(f_{(O_k)_{r_0+2\ep}}))
\eneq
and
\beq\label{dceDT-10}d_{\tau}(\phi_{Y_n}(f_{(O_k)_{r_0+2\ep}}))\le
d_\tau(\phi_{Y_n}(f_{O_{r_0+3\ep}}))\eneq Further since
$d_\tau(h_n(f_{(O_k)_{r_0+2\ep}}))\le \tau(e_n)<\frac{\dt}2$,
combing (\ref{dceDT-7}),(\ref{dceDT-8}),(\ref{dceDT-9})and
(\ref{dceDT-10}), we obtain \beq\label{dceDT-11}
d_\tau(\phi_{X_n}(f_O))<d_\tau(\phi_{Y_n}(f_{O_{r_0+3\ep}})).\eneq
The strict comparison for positive elements implies
$[\phi_{X_n}(f_O)]\le[\phi_{Y_n}(f_{O_{r_0+3\ep}})]$.} 

By the exactly the same argument, 
$[\phi_{Y_n}(f_O)]\le[\phi_{X_n}(f_{O_{r_0+3\ep}})]$.

{ Since
$D^T(\phi_X,\phi_Y)\ge d_H(X,Y)$, for any open set $O$ with
$O_\ep\cap X=X$, we have $O_{r_0+3\ep}\cap Y=Y$, this implies
$$[\phi_{X_n}(f_O)]\le[\phi_{Y_n}(f_{O_{r_0+3\ep}})].$$} 
{
Therefore
$D_c(\phi_{X_n},\phi_{Y_n})\le r_0+3\ep$ holds for all $n>N$. That
is the end of  the proof.}
\end{proof}
\begin{cor}\label{dceconnect}
Let $A$ be a unital simple \CA\, with strict comparison for positive
elements { and with $T(A)\not=\emptyset$}. Suppose that $X,$ or  $Y$ are  connected compact subsets.
Then \beq\label{dceconn-1} D_c(\phi_X, \phi_Y)=D_c^e(\phi_X,
\phi_Y). \eneq
\end{cor}

\begin{proof}
This follows from  \ref{dce<DT} and the last part of  \ref{dtP}.
\end{proof}

\begin{prop}\label{dcedig}
Let $A$ be a unital simple \CA\, with stable rank one, with the strict comparison for positive elements. Suppose that $\phi_X,\, \phi_Y: C(X\cup Y)\to A$ are two
unital \hm s with spectrum $X$ and $Y,$ respectively.
Let $\{\xi_1, \xi_2,...,\xi_k\}\subset X\cap Y.$
Suppose that there is a {sequence of}  finite subsets
of mutually orthogonal non-zero projections
$\{e_{1,n}, e_{2,n},...,e_{k,n}\}$ of $A$ such that
\beq\label{dcedig-1}
&&\lim_{n\to\infty}D_c(\phi_{X,n},\phi_X)=0,\,\,\,
\lim_{n\to\infty} D_c(\phi_{Y,n}, \phi_Y)=0,\\\label{dcedig-2}
&&{\tand}\lim_{n\to\infty}\sup \{\sum_{i=1}^{k}\tau({e_{i,n}}): \tau\in T(A)\}=0,
\eneq
{where $\phi_{X,n}(f)=\sum_{i=1}^{k}f(\xi_{i})e_{i,n}+\psi_{X,n}(f)$, {
$\phi_{Y,n}(f)=\sum_{i=1}^{k}f(\xi_{i})e_{i,n}+\psi_{Y,n}(f)$}
for all $f\in C(X\cup Y),$ and }$\psi_{X,n}, \psi_{Y,n}: C(X\cup Y)\to (1-p_n)A(1-p_n)$
are a unital \hm s with the spectrum in $X$ and $Y,$ respectively,
where $p_n=\sum_{i=1}^{{k}}{e_{i,n}}$ (please see  lines around {\rm (\ref{Ddce-2+1})}).

Then
\beq\label{dcedig-3}
{\liminf}_{n\to\infty}D_c(\psi_{X,n},\psi_{Y,n})\le D_c^e(\phi_X, \phi_Y).
\eneq

\end{prop}

\begin{proof}
Let $d=D_c^e(\phi_X, \phi_Y)$ and let $\Omega=X\cup Y.$ 
Let $\{\xi_1^{(n)},\xi_2^{(n)},..., \xi_{k(n)}^{(n)}\}$ be a sequence of finite subsets of
$X\cap Y$ which are $\ep_n$-dense with $\lim_{n\to\infty}\ep_n=0,$
let $\{e_1^{(n)}, e_2^{(n)},...,e_{k(n)}^{(n)}\}$ be a sequence of mutually orthogonal  non-zero
projections in $A$  with
\beq\label{dcedig-4}
\lim_{n\to\infty} \sup\{\sum_{i=1}^{k(n)}\tau(e_i^{(n)}): \tau\in T(A)\}=0,
\eneq
{and $\{u_n\}$ be be a sequence of unitary}
such that
\beq\label{decdig-5}
&&\lim_{n\to\infty}D_c(\phi_X, h_{X, n,0}+h_{X,n,1})=0,
\lim_{n\to\infty}D_c(\phi_Y,h_{Y, n,0}+h_{Y,n,1}))=0\andeqn\\
&&\lim_{n\to\infty}D_c(h_{X, n,1}, h_{Y,n,1})=d=D_c^e(\phi_X, \phi_Y),
\eneq
where $h_{X,n,0}(f)=h_{Y, n,0}(f)=\sum_{i=1}^{k(n)}f(\xi_i^{(n)})e_i^{(n)}$ for all $f\in C(\Omega)$ and
$h_{X, n,1}, h_{Y, n,1}: C(\Omega)\to  (1-E_n)A(1-E_n)$ are unital
\hm s with spectrum in $X$ and $Y,$ respectively, and where
$E_n=\sum_{i=1}^{k(n)}e_i^{(n)},$ $n=1,2,....$
 {Without loss of generality, we may also assume} that $k(n)\ge k$ and $\xi_i^{(n)}=\xi_i,$
$i=1,2,...,k.$  
By (\ref{dcedig-2}), since $A$ has strict comparison, by passing to a subsequence of
$\phi_{i,n}$ ($\psi_{X,n}$ and $\psi_{Y,n}$),
if necessary, we may further assume that
\beq\label{dcedig-7}
e_{j,n}\le e_j^{(n)},\,\,\, j=1,2,...,k\andeqn n=1,2,....
\eneq
Define
\beq\label{dcedig-8}
&&h_{X,n,0}'(f)=\sum_{i=1}^kf(\xi_i^{(n)})e_{i,n},\\
&&h_{X,n,0}''(f)=\sum_{i=1}^kf(\xi_i^{(n)})(e_i^{(n)}-e_{i,n})+\sum_{i=k+1}^{k(n)}f(\xi_i^{(n)})e_i^{(n)}
\tforal f\in C(\Omega)\\
&&h_{X, n,1}'=h_{X, n,0}''+h_{X, n,1}\andeqn {h_{Y, n,1}'=h_{X,n,0}''}+h_{Y,n,1}.
\eneq
Then, by assumption,
\beq\label{dcedig-9}
\lim_{n\to\infty} D_c(h_{X, n,0}'+\psi_{X,n}, h_{X, n,0}'+h_{X, n,1}')=0\eneq
and
\beq\label{dcedig-10}
\lim_{n\to\infty}D_c(h_{X, n,0}'+\psi_{Y,n}, h_{X, n,0}'+h_{Y,n,1})=0.
\eneq
By the proof of  \ref{Pdced00}, (\ref{dcedig-9}) and (\ref{dcedig-10}) imply that
\beq\label{dcedig-11}
\lim_{n\to\infty} D_c(\psi_{X,n}, h_{X, n,1}')=0
\andeqn
\lim_{n\to\infty} D_c(\psi_{Y,n},h_{Y, n,1}')=0.
\eneq
It follows that
\beq\label{dcedig-12}
&&\hspace{-1in}\liminf_{n\to\infty} D_c(\psi_{X,n}, \psi_{Y,n})\le
\lim_{n\to\infty}D_c(\psi_{X,n}, h_{X, n,1}')+\\
&&\limsup_{n\to\infty}D_c(h_{X, n,1}', h_{{Y}, n,1}')+
\lim_{n\to\infty}D_c(\psi_{Y,n}, h_{Y, n,1}')\\
&&=\limsup_{n\to\infty}D_c(h_{X, n,1}', h_{Y, n,1}')\\
&&\le \limsup_{n\to\infty}D_c(h_{X, n,1}, h_{Y, n,1})=d.
\eneq

\end{proof}

\begin{lem}\label{Lfepapp}
Let $A$ be a unital separable simple \CA\, with real rank zero,
stable rank one and weakly unperforated $K_0(A).$
Let $\phi_X: C(X)\to A$ be a unital \hm.
Then, for any $\ep>0,$ any $\sigma>0,$ any $\eta>0$ and
any finite $\eta$-dense subset $\{\xi_1, \xi_2,...,\xi_m\}\subset X,$
there is a projection $e\in A$ with $\tau(e)<\sigma$ for all
$\tau\in T(A),$ a unital \hm\, $\phi_0: C(X)\to eAe$ with
spectrum $\{\xi_1, \xi_2,...,\xi_m\}$ and
a unital \hm\, $\phi_1: C(X)\to (1-e)A(1-e)$ with finite spectrum
such that
\beq\label{Lfepapp-1}
D_c(\phi_X, \phi_0+\phi_1)<\ep.
\eneq
\end{lem}

\begin{proof}
Since $A$ is simple and has real rank zero and stable rank one with
weakly unperforated $K_0(A),$ $K_0(A)$ has Riesz interpolation property
by a theorem of Zhang (\cite{Zh}). Moreover $\rho_A(K_0(A))$ is
dense in $\Aff(T(A)).$
By \cite{EG},  there exists a
unital simple AH-algebra of no dimension growth $B$ of real rank zero
(therefore $TR(B)=0$ --see \cite{Lnproc})
such that
\beq\label{Lfepapp-2}
(K_0(B), K_0(B)_+, [1_B])&=&(K_0(A), K_0(A)_+, [1_A])\andeqn\\
K_1(B)&=&\{0\}.
\eneq
It follows from \cite{LnK} that there exists a unital \hm\, $h: B\to A$ such that
$h_{*0}$ gives the identity in {(\ref{Lfepapp-2}).}

It follows from \cite{LnMZ} that there exist  unital monomorphisms
$\psi_X': C(X)\to B$ such that
\beq\label{Lfepapp-3}
(h\circ \psi_X')_{*0}=(\phi_X)_{*0} \andeqn \tau\circ h\circ \psi_X'=\tau\circ \phi_X
\eneq
for all $\tau\in T(A).$
Define $\psi_X=h\circ \psi_X'.$
Then
\beq\label{Lfepapp-4}
(\psi_X)_{*0}=(\phi_X)_{*0} \andeqn
\tau\circ \psi_X=\tau\circ \phi_X
\eneq
for all $\tau\in T(A).$
These, in particular, by \ref{Ltp=close}, imply
that
\beq\label{Lfepapp-5}
D_c(\phi_X, \psi_X)=0.
\eneq
So, without loss of generality, we may assume now that $A=B.$ In particular,
$B$ has tracial rank zero.

Let $\ep>0.$ Let $\dt>0$ be a positive number, $S_1, S_2,...,S_k$ be a finite set of mutually disjoint clopen subsets of $X$ and let ${\cal F}\subset C(X)_+$ be a finite subset required by
\ref{Ltp=close} for $\ep>0$ and $\phi_X.$ We may assume
that $X=\sqcup_{i=1}^mS_i$ and $1_{C(X)}\in {\cal F}.$  By Lemma 4.3 of \cite{LnMZ}, there is a projection $p\not=1_A,$ a unital \hm\, $h: C(X)\to pAp$ with finite spectrum
such that
\beq\label{Lfepapp-6}
|\tau\circ h(f)-\tau\circ \phi_X(f)|<\dt/2\tforal f\in {\cal F}\andeqn \\\label{Lfepapp-7}
\tau\circ h(\chi_{S_i})<\tau\circ \phi_X(\chi_{S_i}),\,\,\,i=1,2,...,k,
\eneq
for all $\tau\in T(A),$ {and}
\beq\label{Lfeapapp-8}
h(f)=\sum_{i=1}^mf(\xi_i)e_i+h_1(f)\tforal f\in C(X),
\eneq
where $\{e_1,e_2,...,e_m\}\subset pAp$ is a set of mutually orthogonal
non-zero projections and $h_1: C(X)\to (p-\sum_{i=1}^m e_i)A(p-\sum_{i=1}^m e_i)$ is a unital \hm\, with finite spectrum in $X.$
Note that (\ref{Lfepapp-6}) implies that
\beq\label{Lf-n}
\tau(1-p)<\dt/2\tforal \tau\in T(A).
\eneq

By (\ref{Lfepapp-7}), there are mutually orthogonal projections
$q_1, q_2,...,q_k\in (1-p)A(1-p)$ such that
$[\phi_X(\chi_{S_i})]=[q_i]+[h(\chi_{S_i})],$ $i=1,2,...,k.$
Since $\sum_{i=1}^k \chi_{S_i}=1_{C(X)}$ and $\phi_X$ is unital,
$p+\sum_{i=1}^kq_i=1_A.$
Define $\psi_X: C(X)\to A$ by
$\psi_X(f)=\sum_{j=1}^k f(\zeta_j)q_j+h(f)$ for all $f\in C(X),$
where ${\zeta_j}\in S_j$ is a point, $j=1,2,...,k.$
We compute that,
\beq\label{Lf-n10}
[\psi_X(\chi_{S_i})]=[\phi_X(\chi_{S_i})],\,\,\,i=1,2,...,k.
\eneq
Moreover, by (\ref{Lf-n}) and (\ref{Lfepapp-6}),
\beq\label{Lf-n11}
|\tau(\phi_X(f))-\tau(\psi_X(f))|<\dt\tforal \tau\in T(A).
\eneq
It follows from  \ref{Ltp=close} that
\beq\label{Lf-n12}
D_c(\phi_X, \psi_X)<\ep.
\eneq
Since $A$ is simple and has (SP), we can find non-zero projections
$e_i'\le e_i$ such that
$\sum_{i=1}^m \tau(e_i')<\sigma.$ Put $e=\sum_{i=1}^me_i'.$
Define $\phi_0(f)=\sum_{i=1}^m f(\xi_i)e_i$ for all $f\in C(X)$ and
defined
\beq\label{Lf-n13}
\phi_1(f)=\sum_{j=1}^k f(\zeta_j)q_j+\sum_{i=1}^mf(\xi_i)(e_i-e_i')+h_1(f)
\eneq
for all $f\in C(X).$  Note that $\phi_0+\phi_1=\psi_X.$
Lemma follows.

\end{proof}

\begin{cor}\label{Pepfinite}
Let $A$ be a unital separable simple \CA\, with real rank zero, stable rank one and with weakly unperforated $K_0(A)$ and let $X$ be a compact metric space.
Suppose that $\phi_X: C(X)\to A$ {is a unital \hm.} Then, there exists a sequence of unital \hm s
$\phi_n: C(X)\to A$ with finite dimensional range such that
\beq\label{Pepfinite-1}
\lim_{n\to\infty}D_c^e(\phi_X, \phi_n)=0.
\eneq
\end{cor}

\begin{rem}\label{Rmdce}
{\rm
In the case that $A$ has real rank zero, stable rank one and weakly unperforated $K_0(A),$
 \ref{Lfepapp} shows that, in the definition of $D_c^e(\phi_X, \phi_Y),$ if $X\cap Y\not=\emptyset,$
we can also assume that the sequence of non-zero $\{h_n\}$ exists.
}
\end{rem}

\begin{prop}\label{Pdced0}
Let $\Omega$ be a compact metric space, let $A$ be a unital simple
\CA\,  with  real rank zero, stable rank one and weakly unperforated $K_0(A),$  and let $\phi_X,
\phi_Y, \phi_Z: C(\Omega)\to A$  be unital \hm s with spectrum $X,$
$Y$ and $Z,$ respectively. If, in addition,
$X\cap Y\subset Z$
\beq\label{Pdced0-3} D_c^e(\phi_X, \phi_Y)\le
D_c^e(\phi_X,\phi_Z)+D_c^e(\phi_Z, \phi_Y). \eneq
\end{prop}

\begin{proof}

If $X\cap Y=\emptyset,$ then
$$
D_c^e(\phi_X, \phi_Y)=D_c(\phi_X, \phi_Y)\le D_c(\phi_X, \phi_Z)+D_c(\phi_Z, \phi_Y)
\le D_c^e(\phi_X, \phi_Z)+D_c^e(\phi_Z, \phi_X).
$$
So, we assume that $X\cap Y\not=\emptyset.$

By {the} definition and from  \ref{Rmdce} above,  we have
nonzero sequences of projections {$\{e(n,i)\}$} of
$A,$ unital \hm s $h(n,i): C(\Omega)\to e(n,{i})Ae(n,i)$ and unital
\hm s $\phi(n,i), \phi(Z,n,i),  : C(\Omega)\to
(1-e(n,i))A(1-e(n,i))$ such that \beq\label{Pdcen-1}
&&\lim_{n\to\infty}{\sup_{\tau\in T(A)}}\tau(e(n,i))=0 \\\label{Pdcen-1+1}
&&\lim_{n\to\infty}D_c(\phi_X,
h(n,1)+\phi(n,1))=0,\,\,\, \lim_{n\to\infty}D_c(\phi_Y,
h(n,2)+\phi(n,2))=0;\\\label{Pdcen-1+2}
&&\lim_{n\to\infty}D_c(\phi_Z, h(n,i)+\phi(Z,
n,i))=0;\\\label{Pdcen-1+3}
&&D_c^e(\phi_X,
\phi_Z)=\lim_{n\to\infty}D_c(\phi(n,1),
\phi(Z,n,1));\\
\label{Pdcen-1+4}
&& D_c^e(\phi_Y,
\phi_Z)=\lim_{n\to\infty}D_c(\phi(n,2), \phi(Z,
n,2))\andeqn\\
\label{Pdcen-1+5}
&& \lim_{n\to\infty}D_c(h(n,1)+\phi(Z,
n,1),h(n,2)+\phi(Z,n,2))=0,
\eneq
the spectrum of $h(n,1)$ is $X_n'$ and the spectrum of
$h(n,2)$ is $Y_n',$ that of $\phi(n,1)$ is in $X,$ that of
$\phi(n,2)$ is in $Y,$ $\phi(Z, n,i)$ is in $Z,$ where $X'_n$ is a
finite subset of $X\cap Z$ and $Y'_n$  is a  finite subset of $Z\cap
Y$ which are
 $\ep_n$-dense in $X\cap Z$ and in $Y\cap Z$  with $\lim_{n\to\infty}\ep_n=0.$
 Since $X\cap Y\subset Z,$ we may assume, without loss of generality,
 that $X_n'\cap Y_n'$ is $\ep_n$-dense in $X\cap Y.$
We write
\beq\label{Pdcen-2}
h(n,i)(f)=\sum_{j=1}^{r(n,i)} f(\zeta(n,j,i))q(n,j,i)\tforal f\in C(\Omega),
\eneq
where $\{\zeta(n,1,1), \zeta(n,2,1),...,\zeta(n,r(n,1),1)\}=X_n',$
$\{\zeta(n,1,2), \zeta(n,2,2),...,\zeta(n,r(n,2), 2)\}=Y_n'$ and
$\{q(n,1,i)), q(n,2,i),...,q(n,r(n,i),i)\}$ is a set of mutually orthogonal non-zero projections.
We may further assume that
\beq\label{Pdcen-3}
\zeta(n,j,1)=\zeta(n,j,2),\,\,\,j=1,2,...,k(n)\le r(n,1),\,r(n,2),
\eneq
where $\{\zeta(n,1,1),\zeta(n,2,1),...,\zeta(n,k(n),1)\}$ is $\ep_n$-dense in $X\cap Y.$
Let $X_n$ be the spectrum of $\phi(n,1)$ and $Y_n$ be the spectrum
of $\phi(n,2),$ $n=1,2,....$
Without loss of generality, we may assume that
$X_n'\subset X_n$ and $Y_n'\subset Y_n,$ $n=1,2,....$
Note that, without changing the sums $h(n,i)+\phi(n,i),$ $ h(n,i)+\phi(Z, n,i)$ and (\ref{Pdcen-1})--(\ref{Pdcen-1+4}),
one can choose smaller $q(n,j,i),$ $j=1,2,...,r(n,i),$ $i=1,2$ and $n=1,2,....$
We may assume
that, since $A$ is simple and has (SP), we may assume
that $r(n,1)=r(n,2)$ and
\beq\label{Pdcen-4}
[q(n,j,1)]=[q(n,j',2)],\,\,\, j,j'=1,2,....,k(n),\,\,\, n=1,2,....
\eneq
To simplify the notation, we may further assume that
\beq\label{Pdcen-5}
q(n,j,1)=q(n,j,2),\,\,\, j=1,2,...,k(n),\,\,\,n=1,2,....
\eneq
Put
\beq\label{Pdcen-6}
\phi(n,i)'(f)&=&\sum_{j=k(n)+1}^{r(n,i)}f(\zeta(n,j,i))q(n,j,i)+\phi(n,i)(f),\\
\phi(Z,n)(f)&=&\sum_{j=k(n)+1}^{r(n,1)}f(\zeta(n,j,i))q(n,j,1)+\phi(Z,n,1)(f)\andeqn\\
\phi(Z,n,2)'(f)&=& \sum_{j=k(n)+1}^{r(n,1)}f(\zeta(n,j,2))q(n,j,2)+\phi(Z,n,2)(f)
\eneq
for all $f\in C(\Omega).$
It follows that
\beq\label{Pdcen-8}
\limsup_{n\to\infty} D_c(\phi(n,1)', \phi(Z,n))\le \limsup_{n\to\infty}D_c(\phi(n,1), \phi(Z,n,1))
=D_c^e(\phi_X, \phi_Z).
\eneq

By (\ref{Pdcen-1+5}), (\ref{Pdcen-5}),  and the proof of  \ref{Pdced00},
\beq\label{Pdcen-7}
\lim_{n\to\infty} D_c(\phi(Z,n),\phi(Z,n,2)')=0.
\eneq
It follows that
\beq\label{Pdcen-9}
&&\limsup_{n\to\infty} D_c(\phi(n,2)', \phi(Z,n))\\
&&\le
\limsup_{n\to\infty}D_c(\phi(n,2)', \phi(Z,n,2)')+\lim_{n\to\infty}D_c(\phi(Z,n,2)',\phi(Z,n))\\
&&\le \limsup_{n\to\infty} D_c(\phi(n,2), \phi(Z,n,2))=
D_c^e(\phi_Y, \phi_Z).
\eneq
However, by (\ref{Pdcen-8}) and (\ref{Pdcen-9}),
\beq\label{Pdcen-10}
D_c^e(\phi_X, \phi_Y)
&\le & \limsup_{n\to\infty}D_c(\phi(n,1)', \phi(n,2)')\\
 &\le & \limsup_{n\to\infty}(D_c(\phi(n,1)',\phi(Z,n))+D_c(\phi(n,2)',\phi(Z,n)))
 \\
 &\le & D_c^e(\phi_X,\phi_Z)+D_c^e(\phi_Z, \phi_Y).
\eneq

\end{proof}

\begin{df}\label{dceDxy}
{\rm
Let $A$ be a unital \CA\, and let $x,\, y\in A$ be two normal elements with
${\rm sp}(x)=X$ and ${\rm sp}(y)=Y,$ respectively.
Define $\phi_X, \phi_Y: C(X\cup Y)\to A$ {to be}  unital \hm s defined by
$\phi_X(f)=f(x)$ and $\phi_Y(f)=f(y)$ for all $f\in C(X\cup Y).$
We will use the notation $D_c^e(x, y)$ for
$D_c^e(\phi_X, \phi_Y).$
}
\end{df}

\section{Approximate unitary equivalence}

 The purpose of this section is to present  \ref{Uni2} and  \ref{Uni1}.
 In the case that $A$ is a unital simple \CA\, with $TR(A)=0,$ much more general
 results were presented in \cite{Lntrans}. However, in the spirit of \cite{Lnoldjfa},
 the exact condition for {when} two normal elements are approximately unitarily equivalent
 can be obtained in unital simple \CA\, $A$ with real rank zero, stable rank one and
  with weakly unperforated $K_0(A).$ We are also interested in \ref{Uni3}.

 The following is proved in \cite{Lnoldjfa}.

\begin{thm}\label{TappT}
Let  $\ep>0.$
For any  unital simple \CA\, $A$ of real rank zero with (IR) and any normal element $x\in A$ with $\|x\|\le 1$ such that
\beq\label{TappT-1}
\lambda-x\in {\rm Inv}_0(A)
\eneq
for all $\lambda\in \C$ with ${\rm dist}(\lambda, {\rm sp}(x))\ge \ep/8,$ there
is a normal element with finite spectrum $x_0\in A$ such that
\beq\label{TaapT-2}
\|x-x_0\|<\ep.
\eneq
\end{thm}

\begin{proof}
This was exactly proved in the proof of the Theorem of \cite{Lnoldjfa}.
Note that the set
$$
F_1=\{\lambda\in \C: {\rm dist}(\lambda, {\rm sp}(x))<r\}
$$
in that proof is chosen for $r=\ep/8.$

\end{proof}

\begin{lem}\label{Meas1}
Let $A$ be a unital separable simple \CA\, with real rank zero, stable rank one and
weakly unperforated $K_0(A),$ let $X$ be a compact subset of the plane and let
$\Delta: (0,1)\to (0,1)$ be an increasing function such that
$\lim_{r\to 0}\Delta(r)=0.$  Then, for any $\ep>0,$ there exists $d>0$ with
$d<\ep/128,$  there exists a finite subset $\{f_1,f_2,...,f_n\}\subset C(\overline{X_{d/2}})$
of mutually orthogonal projections with $\sum_{i=1}^n f_i=1_{C(\overline{X_{d/2}})},$  a finite
subset ${\cal H}\subset C(\overline{X_{d/2}})_+$   satisfying the following:
if $h: C(\overline{X_{{d/2}}})\to A$ is a homomorphism  such that
\beq\label{Meas1-1}
\mu_{\tau\circ h}(O)\ge \Delta(r)
\eneq
for any open balls $O$ of $X$ with radius $r\ge \ep/32,$ { then for any homomorphism}
$\phi: C(\overline{X_{\ep/128}})\to A$ such that
\beq\label{Meas1-2}
\phi_{*0}([f_i])=h_{*0}([f_i]),\,\,\, i=1,2,...,n,\\\label{Meas1-2+}
{\lambda-h(z), \lambda-\phi(z)\in {\rm Inv}_0(A)\,\,\,{\rm if}\,\,\, {\rm dist}(\lambda, X)\ge d,}\, \tand\\
|\tau\circ h(g)-\tau\circ \phi(g)|<(1/4)\Delta(\ep/32) \tforal g\in {\cal H},
\eneq
then there exists a unitary $u\in A$ such that
\beq\label{Meas1-3}
\|u^*h(z)u-\phi(z)\|<\ep,
\eneq
where $z\in C(\overline{X_{d/2}})$ is the identity function on $\overline{X_{d/2}}.$
\end{lem}

 \begin{proof}
Let $\ep>0.$ We choose $\dt>0$ which is required by  \ref{Lappdcu}
for $\ep/8$ (with $\overline{X_{\ep/16}}=\Omega$).
Let
\beq\label{Meas1-n}
d=\min\{\dt/8, \ep/2^{21}\}.
\eneq
Let $\{f_1',f_2',..., f_n'\}\subset C(\overline{X_{\ep/128}})$ be a subset of projections
  be as required by  \ref{Pclosem}
for $\Delta,$ $\overline{X_{d/2}}$ (in place of $X$) and $\ep/2$ (instead of $\ep$).
Define $f_i=f_i'|_{\overline{X_{d/2}}},$ $i=1,2,...,n.$
{Now assume $A,$ $h$ and $\phi$ be as stated above..}
By applying  \ref{Pclosem}, one has
\beq\label{Meas1-4}
D_{c }(h,\phi)<\ep/2.
\eneq
It follows from   \ref{TappT} that,
if (\ref{Meas1-2+}) holds,
there are normal elements $x_0$ and $y_0$ with finite spectrum
such that
\beq\label{Meas1-5}
\|h(z)-x_0\|<\min\{\ep/16,\dt\}\andeqn \|\phi(z)-y_0\|<\min\{\ep/16,,\dt\}.
\eneq
By  \ref{Lappdcu},
we have  that
\beq\label{Meas1-6}
D_{c}(h(z), x_0)<\ep/8\andeqn D_{c }(\phi(z), y_0)<\ep/8
\eneq
Therefore
\beq\label{Meas1-7}
D_{c }(x_0, y_0)<3\ep/4.
\eneq
{Since $A$ is simple and has real rank zero and stable rank one with
weakly unperforated $K_0(A),$ $K_0(A)$ has Riesz interpolation property
by a theorem of Zhang \cite{Zh},} by  \ref{MLfsp}, there exists a unitary $u\in A$ such that
\beq\label{Meas1-8}
\|u^*x_0u-y_0\|<3\ep/4.
\eneq
Combining this with (\ref{Meas1-5}), we conclude that
\beq\label{Meas1-9}
\|u^*h(z)u-\phi(z)\|<\ep.
\eneq

\end{proof}

\begin{lem}\label{Lext}
Let $A$ be a unital  infinite dimensional  separable simple \CA\, with real rank zero, stable rank one and with weakly unperforated $K_0(A)$ and
$X\subset \C$ be a compact subset.
Let $p_1, p_2,...,p_n\in C(X)$ be $n$ mutually orthogonal
projections with $\sum_{i=1}^n p_i=1_{C(X)}$ such that
$\{[p_1,], [p_2],...,[p_n]\}$ generates a subgroup
$G$ of $K_0(C(X)).$
Suppose that $\kappa_0: G\to K_0(A)$ is an order preserving
\hm\, with $\kappa_0([1_{C(X)}])=[1_A]$ and with
$\kappa_0([p_i])>0,$ $i=1,2,...,n,$ and
$\kappa_1: K_1(C(X))\to K_1(A)$ is a \hm. Then
there is a unital monomorphism $\phi: C(X)\to A$ such that
\beq\label{Lext-1}
\phi_{*0}|_G=\kappa_0\tand \phi_{*1}=\kappa_1.
\eneq
\end{lem}

\begin{proof}
Since $A$ is simple and has real rank zero and stable rank one with
weakly unperforated $K_0(A),$  as in the proof of \ref{Meas1}, $K_0(A)$ has Riesz interpolation property
 and $\rho_A(K_0(A))$ is
dense in $\Aff(T(A)).$
It follows from \cite{EG} that there is a unital simple AH-algebra $B$
with no dimension growth and {real rank zero} such that
\beq\label{Lext-2}
(K_0(A), K_0(A)_+, [1_A], K_1(A))=(K_0(B), K_0(B)_+, [1_B], K_1(B)).
\eneq
It follows from 4.6 of \cite{LnK} that there exists a unital embedding $\imath: B\to A$ which carries the above identification. We will also use the fact that $\rho_B(K_0(B))=\rho_A(K_0(A))$ is dense in $\Aff(T(A)).$ Therefore it suffices
to show that the lemma holds for $A=B.$
There are mutually orthogonal nonzero projections $e_1,e_2,...,e_n\in B$ such that $\kappa_0([p_i])=e_i,$ $i=1,2,...,n.$
Let $X_i$ be the clopen subset of $X$ corresponding to
the projection $p_i,$ $i=1,2,..., n.$
Each $e_iBe_i$ is an infinite dimensional  unital simple \CA\, with $TR(e_iBe_i)=0.$ Therefore there is a monomorphism $\psi_i: C(X_i)\to
e_iBe_i,$ $i=1,2,...,n.$
Define a unital monomorphism $\psi: C(X)\to B$ by
$\psi(f)=\sum_{i=1}^n \psi_i(fp_i)$ for all $f\in C(X).$
Note that
\beq\label{Lext-2n}
\kappa_0=\psi_{*0}|_G.
\eneq
We also have that $\psi_{*0}$ and $\psi^{\sharp}: C(X)_{s.a.}\to \Aff(T(B))$ are compatible and $\psi^{\sharp}$ is strictly positive.
It follows from Corollary 5.3 of \cite{LnMZ} that there is a unital monomorphism
$\phi: C(X)\to B$ such that
\beq
\phi_{*0}=\psi_{*0}\andeqn\\
\phi_{*1}=\kappa_1.
\eneq
Lemma follows.
\end{proof}

\begin{lem}\label{Ldig}
Let $A$ be a unital simple \CA\, of real rank zero,
 stable rank one with weakly unperforated $K_0(A),$ let $X\subset \C$ be
a compact subset of the plane and
let $\Delta: (0,1)\to (0,1)$ be a non-decreasing function
such that $\lim_{t\to 0}\Delta(t)=0.$
For any $1>r_0>0,$ any $\ep>0,$ any $\eta>0,$  any $\eta_1>0$ with $\eta_1<r_0/4,$  any $\eta_2>0$  and any finite
subset ${\cal G}\subset C(\overline{X_{\eta_1}})_+,$ where
${\overline{X_{\eta_1}}=\overline{\{\lambda\in \C: {\rm dist}(\lambda, X)<\eta_1\}}},$
there is
a finite subset
${\cal H}\subset C(X)_{s.a.}$
  satisfying the following:
If $x,\,y\in A$ are   normal elements with ${\rm sp}(x),\,{\rm sp}(y)\subset X$ such that
\beq\label{Ldig-1}
|\tau\circ \phi(g)-\tau\circ \psi(g)|<\eta/2\tforal g\in {\cal H}\tand\\\label{Ldig-1+}
\mu_{\tau\circ \phi}(O)\ge \Delta(r)\tforal \tau\in T(A)
\eneq
for all open balls $O$ of $X$ with radius $r\ge r_0,$ where
$\phi, \psi: C(X)\to A$ are defined by $\phi(f)=f(x)$ and
$\psi(f)=f(y)$ for all $f\in C(X),$ respectively,
then there exists $\{\lambda_1, \lambda_2,...,\lambda_n\}\subset X$
which is $r_0$-dense in $X,$ non-zero mutually orthogonal projections
$\{e_1,e_2,...,e_n\}\subset A,$ two normal elements $x_0, y_0\in eAe,$
where $e=1-\sum_{i=1}^ne_i$ and a unitary $u\in A$ such that
\beq\label{Ldig-2}
&&\|x-(\sum_{i=1}^n \lambda_i e_i +x_0)\|<\ep/2,\,\,\,
\|u^*yu-(\sum_{i=1}^n \lambda_i e_i+y_0)\|<\ep/2\\\label{Ldig-3}
&&|\tau\circ \phi_0(g)-\tau\circ \psi_0(g)|<\eta\tforal g\in {\cal G}\tand
\tforal \tau\in T(A),\\\label{Ldig-4}
&&{\rm sp}(x_0), {\rm sp}(y_0)\subset \overline{X_{\eta_1}},\\\label{Ldig-4n}
&&\tau(\sum_{i=1}^n e_i)<\eta_2\tforal \tau\in T(A)\tand\\
&&\mu_{\tau\circ \phi_0}(O)\ge (1/2)\Delta(r/6)
\eneq
for all open balls $O\subset \overline{X_{\eta_1}}$ with radius $r\ge 3r_0$ and for
all $\tau\in T(A),$ where $\phi_0, \psi_0: C(\overline{X_{\eta_1}})\to A$ is defined by
$$
\phi_0(f)=\sum_{i=1}^nf(\lambda_i)e_i+f(x_0)\tand
\psi_0(f)=\sum_{i=1}^n f(\lambda_i)e_i+f(y_0)
$$
for all $f\in C(\overline{X_{\eta_1}}).$

\end{lem}

\begin{proof}
To simplify the notation, we may assume that
$X$ is a subset of the unit disk.
Note that, since $A$ has real rank zero and stable rank one,
so does $pAp$  for any projection $p\in A.$ It follows that $pAp$  has (IR) (see \cite{FR}).
Let $1>\ep>0$ be given.
Let $\ep_1>0$ be such that $\ep/4>\ep_1>0.$ By \cite{FR}, there exists $\dt_1>0$  such that, for any \CA\, $D$ with (IR), any element $z\in D$ with $\|z\|\le 2$ and
\beq\label{Ldig-10}
\|z^*z-zz^*\|<\dt_1,
\eneq
then there is a normal element $z_0\in D$ such that
\beq\label{Ldig-11}
\|z_0-z\|<\ep_1.
\eneq
Let $\eta, \eta_1, \eta_2>0$ be given and  a finite subset ${\cal G}\subset C(\overline{X_{\eta_1}})_+$ be given.
Denote by $\phi': C(\overline{X_{\eta_1}})\to A$ the \hm\, defined by
$\phi'(f)=f(x)$ for all $f\in C(\overline{X_{\eta_1}}).$
Since $\eta_1<r_0/4,$ every open ball of $\overline{X_{\eta_1}}$ of  radius $r>r_0$ contains
an open balls of $X$ of radius $r/2.$ It follows that
\beq\label{Ldig-11+}
\mu_{\tau\circ \phi'}(O)\ge \Delta(r/2)
\eneq
for all open balls  $O\subset \overline{X_{\eta_1}}$ with radius $r>r_0$ and for all
$\tau\in T(A).$

We will applying Lemma 2 of \cite{Lnjot94}. Note that since $A$ has real rank zero, non-zero projections $p_k$ described in that lemma exists. Thus, we obtain non-zero mutually orthogonal projections $p_1,p_2,...,p_n\in A$ and
$p_1',p_2',...,p_n'\in A$ such that
\beq\label{Ldig-12}
\|x-(\sum_{i=1}^n\lambda_ip_i+pxp)\|<\min\{\dt_1/16,\ep_1/16\}\andeqn\\
\|y-(\sum_{i=1}^n \lambda_ip_i'+p'yp')\|<\min\{\dt_1/16, \ep_1/16\},
\eneq
{where $p=\sum_{i=1}^np_i,p'=\sum_{i=1}^np'_i$.}

Since ${\rm sp}(x)$ and ${\rm sp}(y)$ {are} $r_0$-dense by (\ref{Ldig-1+}), we may assume
that $\{\lambda_1,\lambda_2,...,\lambda_n\}$ is $r_0$-dense.
Since $A$ is simple and has real rank zero, there are possibly smaller non-zero projections $e_i\le p_i$ such that $e_i\lesssim p_i',$ $i=1,2,...,n.$
In other words, since $A$ has stable rank one, there is a unitary
$u\in A$ such that
\beq\label{Ldig-13}
\tau(\sum_{i=1}^n e_i)<\eta_2\tforal \tau\in T(A),\\
\|x-(\sum_{i=1}^n\lambda_ie_i+
\sum_{i=1}^n\lambda_i(p_i-e_i)+pxp)\|<\min\{\dt_1/16,\ep_1/16\}\andeqn\\
\|u^*yu-(\sum_{i=1}^n\lambda_i e_i+y_0')\|<\min\{\dt_1/16,\ep_1/16\},
\eneq
where $y_0'=u^*y_0u+\sum_{i=1}^n{\lambda_i}(u^*p_i'u-e_i).$
Put $x_0'=\sum_{i=1}^n\lambda_i(p_i-e_i)+pxp.$
Let $e=1-\sum_{i=1}^ne_i.$
Then
\beq\label{Ldig-14}
\|(x_0')^*(x_0')-(x_0')(x_0)^*\|<\dt_1\andeqn
\|(y_0')^*(y_0')-(y_0')(y_0')^*\|<\dt_1.
\eneq
By the choice of $\dt_1,$ by applying \cite{FR}, there exist normal elements
$x_0, y_0\in eAe$ such that
\beq\label{Ldig-15}
\|x_0-x_0'\|<\ep_1\andeqn \|y_0-y_0'\|<\ep_1{.}
\eneq
It follows that
\beq\label{Ldig-16}
\|x-(\sum_{i=1}^n \lambda_i e_i+x_0)\|<\ep_1\andeqn
\|u^*yu-(\sum_{i=1}^n\lambda_i e_i +y_0)\|<\ep_1{.}
\eneq
Define $\phi_0, \psi_0: C(\overline{X_{\eta_1}})\to A$ by
$$
\phi_0(f)=\sum_{i=1}^nf(\lambda_i)e_i +f(x_0)\andeqn\\
\psi_0(f)=\sum_{i=1}^nf(\lambda_i)e_i+f(y_0)
$$
for all $f\in C(\overline{X_{\eta_1}}).$
Now, we will choose $\ep_1$ sufficiently small to begin with.
First, by applying Lemma 3.4 of \cite{LnHAH}, we will choose ${\cal H}$ sufficiently large and $\sigma$ sufficiently small (independent of $A$ and normal elements given) so  that
\beq\label{Ldig-17}
\mu_{\tau\circ\phi_0}(O)\ge (1/2)\Delta(r/6)
\eneq
for all open balls $O$ of $\overline{X_{\eta_1}}$ of radius $r\ge 3r_0,$ if
 (\ref{Ldig-3}) holds.  In particular, we choose
 ${\cal H}\supset {\cal G}$ and $\sigma<\eta/2.$
 Since ${\cal G}$ is finite and given, with sufficiently smaller $\ep_1,$ we also have, by (\ref{Ldig-16}), and by assumption (\ref{Ldig-3}),
 \beq\label{Ldig-18}
|\tau\circ \phi_0(g)-\tau\circ \psi_0(g)|<\eta\tforal g\in {\cal G}
\eneq
and for all $\tau\in T(A).$
\end{proof}

\begin{thm}\label{Uni1}
Let $A$ be a unital separable simple \CA\, of real rank zero,
 stable rank one with weakly unperforated $K_0(A),$ let $X\subset \C$ be
a compact subset of the plane and
let $\Delta: (0,1)\to (0,1)$ be a non-decreasing function
such that $\lim_{t\to 0}\Delta(t)=0.$
For any $\ep>0$ there is a finite subset of unitaries ${\cal V}\subset C(X),$ a finite
subset  of projections $\{p_1, p_2,...,p_n\}\subset C(X),$ a finite subset
${\cal H}\subset C(X)_{s.a.},$ $\sigma>0$ and $r_0>0$ satisfying the following:
If $x,\,y\in A$ are   normal elements with ${\rm sp}(x),{\rm sp}(y)\subset X$ such that
\beq\label{Uni1-1}
\phi_{*0}([p_i])=\psi_{*0}([p_i])\\
\phi_{*1}|_{\cal V}=\psi_{*1}|_{\cal V},\\
|\tau\circ \phi(g)-\tau\circ \psi(g)|<\sigma\tforal g\in {\cal H}\tand\\
\mu_{\tau\circ \phi}(O)\ge \Delta(r)\tforal \tau\in T(A)
\eneq
for all open balls $O$ of $X$ with radius $r\ge r_0,$ where
$\phi, \psi: C(X)\to A$ are defined by $\phi(f)=f(x)$ and
$\psi(f)=f(y)$ for all $f\in C(X),$ respectively.
then there exists a unitary $u\in A$ such that
\beq\label{Uni1-2}
\|u^*xu-y\|<\ep.
\eneq
\end{thm}

\begin{proof}
To simplify notations, we may assume that $X$ is a compact subset of the unit disk.
Let $\Delta_1(r)=(1/64)\Delta(r/12)$ for all $r\in (0,1).$
Let $\ep>0.$ Choose ${\cal F}=\{z\},$ where $z$ is the identity function on the unit disk.

Let $\ep_0=\ep/{64}.$    Choose $r_0=\ep_0/16.$   Let $d>0$ with
$d<\ep_0/2^{20}$ be required by  \ref{Meas1} for $\ep_0$ (in place of $\ep$) and $\Delta_1$
(in place of $\Delta$). Put $Y=\overline{X_{d/2}}.$
Let $\{f_1, f_2,...,f_n\}\subset  C(Y)$ be a finite subset of mutually orthogonal
projections and  let ${\cal H}_1\subset C(Y)$  (in place of ${\cal H}$)  be required by
 \ref{Meas1}
for $\ep_0$ (in place of $\ep$) and for $\Delta_1$ (in place of $\Delta$).
We may assume that
$1_{C(Y)}=\sum_{i=1}^n f_i$ and
$f_i$ corresponding to $r_0/2^{14}$-connected components.
Note that $Y$ is homeomorphic to a finite CW complex in the plane.
Let $\{v_1,v_2,...,v_{n_1}\}\subset Y$ be a set of unitaries which generates $K_1(C(Y)).$

Choose $\ep_1>0$ satisfying the following:
if $x', y'$ be two normal elements in a unital \CA\, $B$ with ${\rm sp}(x'),\, {\rm sp}(y')\subset
Y$ and
$$
\|x'-y'\|<\ep_1,
$$
then
\beq\label{Uni1-3}
(\phi')_{*0}([f_i])&=&(\psi')_{*0}([f_i]),\,\,\,i=1,2,...,n\andeqn\\
(\phi')_{*1}&=&(\psi')_{*1},
\eneq
where $\phi', \psi': C(Y)\to B$ are defined by $\phi'(f)=f(x')$ and $\psi'(f)=f(y')$ for all $f\in C(Y).$
Let $\ep_2=\min\{\ep_1/4, \ep_0/2\}.$
Let $\eta=(1/2^{10})(\Delta_1(\ep_0/64)),$ $\eta_1=\min\{d/2,r_0/64\}$
and $\eta_2=\eta.$

Let ${\cal H}\subset { C(X)_{s.a}}$ be a finite subset   be required by  \ref{Ldig}
for $r_0,$ $\ep_2$ (in place of $\ep$), $\eta,$ $\eta_1,$ $\eta_2,$ ${\cal H}_1$ (in place
of ${\cal G}$) and $\Delta.$

Let $p_i=f_i|_{X},$ $i=1,2,...,n$ and let $u_j=v_j|_{X},$ $j=1,2,...,n_1.$
Put ${\cal V}=\{u_1, u_2,...,u_{n_1}\}.$
Let $\sigma=\eta/2.$
Now suppose that $x, y$ are two normal elements in $A$ satisfying the assumption for
the above ${\cal V},$ $\{p_1,p_2,...,p_n\},$ ${\cal H},$ $\sigma$ and $r_0.$

By  \ref{Ldig}, there  exists $\{\lambda_1, \lambda_2,...,\lambda_m\}\subset X$ which
is $r_0$-dense, there are non-zero mutually orthogonal projections
$\{e_1,e_2,...,e_m\}\subset A,$  two normal elements $x_0, y_0\in eAe,$
where $e=1-\sum_{i=1}^me_i$ and a unitary $w\in A$ such that
\beq\label{Uni-4}
&&\|x-(\sum_{i=1}^m \lambda_i e_i +x_0)\|<\ep_2/2,\,\,\,
\|w^*yw-(\sum_{i=1}^m \lambda_i e_i+y_0)\|<\ep_2/2,\\\label{Uni-5}
&&|\tau\circ \phi_0(g)-\tau\circ \psi_0(g)|<\eta\tforal g\in {\cal H}_1\tand
\tforal \tau\in T(A),\\\label{Uni-6}
&&{\rm sp}(x_0), {\rm sp}(y_0)\subset \overline{X_{\eta_1}},\\\label{Uni-7}
&&\tau(\sum_{i=1}^m e_i)<\eta_2\tforal \tau\in T(A)\tand\\\label{Uni-7+}
&&\mu_{\tau\circ \phi_0}(O)\ge (1/2)\Delta(r/6)
\eneq
for all open balls $O\subset \overline{X_{\eta_1}}$ with radius $r\ge 3r_0$ and for
all $\tau\in T(A),$ where $\phi_0, \psi_0: C(\overline{X_{\eta_1}})\to A$ is defined by
$$
\phi_0(f)=\sum_{i=1}^mf(\lambda_i)e_i+f(x_0)\andeqn
\psi_0(f)=\sum_{i=1}^m f(\lambda_i)e_i+f(y_0)
$$
for all $f\in C(\overline{X_{\eta_1}}).$

Since $A$ is a unital simple \CA\, with real rank zero, there are, for each $i,$  non-zero mutually orthogonal projections
$e_{i,0},$
$e_{i,1},e_{i,2}$ such that
\beq\label{Uni-8-1}
e_i=e_{i,0}+e_{i,1}+e_{i,2}\andeqn\, 9\tau(\sum_{i=1}^n(e_{i,1}+e_{i,2}) )< \tau(e_j)
\eneq
for all $\tau\in T(A)$ and $j=1,2,...,m.$
Define
\beq\label{Uni-8}
\phi_0'(f)=\sum_{i=1}^mf(\lambda_i)e_i,\,\,\,
\phi_{0,0}(f)=\sum_{i=1}^m f(\lambda_i)e_{i,0},\\
\phi_{0,1}(f)=\sum_{i=1}^m f(\lambda_i)e_{i,1}\,\, \andeqn
\phi_{0,2}=\sum_{i=1}^m f(\lambda_i)e_{i,2}
\eneq
for all $f\in C(Y).$

Put $P_1=\sum_{i=1}^m e_{i,1}$ and $P_2=\sum_{i=1}^ne_{i,2}.$
We have
\beq\label{Uni-9}
\tau(P_1+P_2)<\eta_2/8\tforal \tau\in T(A).
\eneq

It follows from  \ref{Lext} that there are unital monomorphisms
$H_1: C(Y)\to P_1AP_1$ and $H_2: C(Y)\to P_2AP_2$ such that
\beq\label{Uni-10}
(H_1)_{*0}([f_i])=(\phi_{0,1})_{*0}([f_i]),\,\,\, (H_2)_{*0}([f_i])=(\phi_{0,2})_{*0}([f_i]),\,\,\,i=1,2,...,n\andeqn\\\label{Uni-10+}
(H_1)_{*0}([v_j])=-(\phi_x)_{*1}([v_j])\andeqn (H_2)_{*1}([v_j])=(\phi_x)_{*1}([v_j]),\,\,\,j=1,2,...,n_1,
\eneq
where $\phi_x: C(Y)\to A$ is defined by $\phi_x(f)=f(x)$ for all $f\in C(Y).$
Let $x_1=H_1(z)+H_2(z),$ where $z$ is the identity function on $Y.$
Then, by \ref{TappT}, there are $\{\mu_1, \mu_2,....,\mu_{m_1}\}$ which is $r_0/2^{12}$-dense
and mutually orthogonal non-zero projections $\{e_{1,3}, e_{2, 3},..., {e_{m_1,3}}\}$
in $(P_1+P_2)A(P_1+P_2)$ such that
\beq\label{Uni-11}
\|x_1-\sum_{j=1}^{m_1}\mu_je_{j,3}\|<r_0/2^{12}.
\eneq
By \ref{Cabs}, there is a unitary $v\in (1-e)A(1-e)$ such that
\beq\label{Uni-12}
\|v^*\phi_0'(z)v-(x_1+\phi_{0,0}(z))\|<r_0/2^{11}.
\eneq

Define $\phi_x', \psi_y': C(Y)\to (1-P_2)A(1-P_2)$ by
\beq\label{Uni-13}
\phi_x'(f)=H_1(f)+f(x_0)+\phi_{00}(f)\andeqn
\psi_y'(f)=H_1(f)+f(y_0)+\phi_{00}(f)
\eneq
for all $f\in C(Y).$
We have, for all $\lambda\in \C,$
\beq\label{Uni-14}
\lambda-\phi_x'(z), \,\,\,\lambda-\psi_y'(z)\in {\rm Inv}_0((1-P_2)A(1-P_2))
\eneq
if $\lambda\not\in \overline{X_{d/2}}.$  By the choice of $\ep_1,$ (\ref{Uni-4}) and
(\ref{Uni-10}), and the fact that $A$ has stable rank one,
we check that
\beq\label{Uni-15}
(\phi_x')_{*0}([f_i])=(\psi_y')_{*0}([f_i]),\,\,\,i=1,2,...,n.
\eneq
For each open ball $O\subset \overline{X_{d/2}}$ with radius $r>r_0,$
we estimate that, by (\ref{Uni-7}) and (\ref{Uni-7+}),
\beq\label{Uni-16}
\mu_{\tau\circ \phi_x'}(O) &>&
\mu_{\tau\circ \phi_0}(O)-\tau(P_1)-\tau(P_2)\\
&\ge & (1/2)\Delta(r/6)-\eta_2\\
&\ge & (1/2)\Delta(r/6)-2^{-10}\Delta_1(\ep_0/64)\\
&>& 1/4\Delta(r/6)\ge \Delta_1(r)
\eneq
for all $r\ge 3r_0$ and for all $\tau\in T(A).$
It follows that
\beq\label{Uni-17}
\mu_{t\circ \phi_x'}(O)> \Delta_1(r)\tforal t\in T((1-P_2)A(1-P_2))
\eneq
and for all open balls {$O$} with radius $r>3r_0.$
For $f\in {\cal H}_1,$  by (\ref{Uni-7}) and (\ref{Uni-5})
\beq\label{Uni-18}
|\tau\circ \phi_x'(f)-\tau\circ \psi_y'(f)| &<&
|\tau\circ \phi_0(f)-\tau\circ \psi_0(f)|+2\tau(1-e)\\
&<& \eta+2\eta_2\le (3/2^{10})(\Delta_1(\ep_0/64))
\eneq
for all $\tau\in T(A).$ It follows that
\beq\label{Uni-19}
|t\circ \phi_x'(f)-t\circ \psi_y'(f)|<{(3/2^{10})(\Delta_1(\ep_0/64))\over{1-\eta}}<(1/4)\Delta(\ep_0/64)
\eneq
for all $t\in T((1-P_2)A(1-P_2)).$

 It follows from \ref{Meas1} that there is a unitary $w_0\in (1-P_2)A(1-P_2)$ such that
 \beq\label{Uni-20}
 \|w_0^*\phi_y'(z)w_0-\phi_x'(z)\|<\ep_0.
 \eneq
Put $w_1=v+e$ and $w_2=w_0+(1-P_2)$ and $u=w_1w_2^*w_1^*.$
Then, by (\ref{Uni-4}), (\ref{Uni-12}), (\ref{Uni-8}) and (\ref{Uni-12}) again,
{\beq\label{Uni-21}
x&\approx&\hspace{-0.1in}_{{\ep_2/2}} \,\,\,\,\phi_0(z)\approx_{r_0/2^{11}} w_1(x_1+\phi_{00}(z)+x_0)w_1^*\\
&=&\,\,\,\,\,\,\,\,\,w_1(H_2(z)+H_1(z)+\phi_{00}(z)+x_0)w_1^* \\
&\approx&\hspace{-0.1in}_{\ep_0} \,\,\,\,\,\,\,\,w_1(H_2(z)+w_2^*\phi_y'(z)w_2)w_1^*\\
&=&\,\,\,\,\,\,\,\,\, w_1w_2^*(H_2(z)+H_1(z)+\phi_{00}(z)+y_0)w_2w_1^* \\
&\approx&\hspace{-0.1in}_{r_0/2^{11}}w_1w_2^*(w_1^*(\phi_0'(z)+y_0)w_1w_2w_1^*\\
&\approx&\hspace{-0.1in}_{{\ep_2/2}}\,\,\,\,\,uyu^*.
\eneq}
But {$\ep_2/2+r_0/2^{11}+\ep_0+r_0/2^{11}+\ep_2/2<\ep.$}
\end{proof}

\begin{thm}\label{Uni2}
Let $A$ be a unital separable simple \CA\, of {real rank zero, stable rank one and weakly unperforated $K_0(A)$} and let $x\in A$ be a normal element with ${\rm sp}(x)=X.$ For any $\ep>0$ there is a finite subset ${\cal V}\subset K_1(C({\rm sp}(x))),$ a finite
subset ${\cal P}\subset K_0(C({\rm sp}(x))),$ a finite subset
${\cal H}\subset C({\rm sp}(x))_{s.a.},$ $\sigma>0$  satisfying the following:
If $y\in A$ is    normal element with ${\rm sp}(y)\subset X$ such that
\beq\label{Uni2-1}
\phi_{*0}|_{\cal P}&=&\psi_{*0}|_{\cal P}\\
\phi_{*1}|_{\cal V}&=&\psi_{*1}|_{\cal V}\tand\\
|\tau\circ \phi(g)-\tau\circ \psi(g)|&<&\sigma\tforal g\in {\cal H} \tforal \tau\in T(A),
\eneq
then there exists a unitary $u\in A$ such that
\beq\label{Uni2-2}
\|u^*xu-y\|<\ep.
\eneq
\end{thm}

\begin{thm}\label{Uni3}
Let $A$ be a unital separable simple \CA\, with real rank zero, stable rank one and with weakly unperforated $K_0(A).$ Let $x, y\in  A$ be two normal elements. Suppose that $D_c(x, y)=0$ and
$[\lambda-x]=[\lambda-y]$ in $K_1(A)$ for all $\lambda\not\in X\cup Y.$  Then \beq\label{Uni3-1}
{\rm dist}({\cal U}(x), {\cal U}(y))=0
\eneq
\end{thm}

\begin{proof}
Since $A$ is a unital simple \CA, the assumption
of $D_c(x, y)=0$ implies that ${\rm sp}(x)={\rm sp}(y)=X.$
Let $\phi,\psi: C(X)\to A$ be the unital monomorphisms induced by $x$ and $y,$ respectively. The assumption implies that
$\phi_{*1}=\psi_{*1}.$ It follows from
 \ref{Pcsameclose} that we also have $\phi_{*0}=\psi_{*0}.$
Thus the theorem follows from  \ref{Uni2}.
\end{proof}

\section{Distance between unitary orbits for normal elements
with non-zero $K_1$}

Let $A$ be a unital separable simple \CA\, with real rank zero, stable rank one and weakly unperforated $K_0(A).$                                \ref{Uni3} provides a clue how to described an  upper bound for  the distance between unitary orbits for normal elements in $A.$ If two normal elements $x, y\in A$ have the same spectrum and induce the same \hm\, from $K_1(C({\rm sp}(x)))$ to $K_1(A),$
then an upper bound for the distance between their unitary orbits can be similarly described. When they have different spectrum and with non-trivial $K_1$ information,
however, things are very different. This section  deals with the case
that $(\lambda-x)^{-1}(\lambda-y)\in {\rm Inv}_0(A)$ for all $\lambda\not\in
{\rm sp}(x)\cup {\rm sp}(y).$
Note that the assumption {allows} the case that $\lambda-x\not\in {\rm Inv}_0(A)$ for $\lambda\in  Y\setminus X$ and $\lambda-y\not\in {\rm Inv}_0(A)$ for $\lambda\in X\setminus Y.$  This can be done partly because we are able to
borrow a  Mayer-Vietoris Theorem.

\begin{df}\label{Dparing}
{\rm

Let $A$ be a unital simple \CA\, with real rank zero, stable rank one and
weakly unperforated $K_0(A)$ and let $\Omega$ be a compact metric space.
Let $F_1$ and $F_2$ be two finite subsets of $\Omega.$
Suppose that {$\kappa_{F_1}, \kappa_{F_2}\in H_{c,1}(C(\Omega), A)_+$} are
two elements  represented by two \hm s whose spectra are $F_1$ and $F_2,$ respectively.
Suppose also that $\kappa_{F_1}({f_O})$ and $ \kappa_{F_2}({f_O})$  {are projections} for all
open subsets $O\subset \Omega.$

Suppose that $F_1=\{x_1, x_2,...,x_m\}$ and
$F_2=\{y_1, y_2,...,y_n\}.$ Suppose that
\beq\label{Dd-6}
 D_c({\kappa_{F_1}, \kappa_{F_2}})=r.
\eneq

Then, as proved earlier in  \ref{MLXY}, {for any $\ep>0$,} there are
$a_{i,j}\in W(A),$ $1\le i\le n$ and $1\le j\le m$ such that
\beq\label{Dd-7}
\sum_{j=1}^na_{i,j}=\kappa_{F_1}([f_{\{x_i\}}]),\,\,\, \sum_{i=1}^m a_{i,j}=\kappa_{F_2}({[f_{\{y_j\}}]})\andeqn\\\label{Dd-7+}
|x_i-y_j|\le r{+\ep},\,\,\, {\rm if}\,\,\, a_{i,j}\not=0
\eneq

 By a paring of {$\kappa_{F_1}$ and $\kappa_{F_2}$}  we mean a subset $R{(\kappa_{F_1}, \kappa_{F_2})}\subset \{1,2,...,m\}\times \{1,2,...,n\}$ of those
 pairs of $(i,j)$ such that (\ref{Dd-7}) and  (\ref{Dd-7+}) hold.
}
\end{df}

 \begin{df}\label{Dhub}
{\rm Given a pair of {$\kappa_1$ and $\kappa_2$} with spectra $X$ and $Y,$   we say
that the pair has a {\it hub} at    $X\cap Y$, if $X=\sqcup_{i=1}^{m_1}S_i$ and $Y=\sqcup_{k=1}^{m_2}G_k,$ where $\{S_1,S_2,...,S_{m_1}\}$ is a set of mutually disjoint clopen subsets of $X$ and
$\{G_1, G_2,...,G_{m_2}\}$ is a set of mutually disjoint clopen subsets of $Y,$
there exists $\ep_0>0$ such that, for any $0<\ep<\ep_0,$
there are finite $\ep$-approximations $\kappa_{F_1}$ of $\kappa_1$ and
$\kappa_{F_2}$ of $\kappa_2$ satisfying the following:
There is a pairing  {$R(\kappa_{F_1}, \kappa_{F_2})$} such that,
for each  pair $(t,k)$  with $S_t\cap G_k\not=\emptyset,$  there is a pair $(i,j)\in R(\kappa_{F_1}, \kappa_{F_2})$ such that  $x_i, y_j\in S_t\cap G_k.$
}
\end{df}

Obvious examples that { any pairs $(\kappa_1,\kappa_2)$ have  a hub  at $X\cap Y$ are when $X\cap Y=\emptyset,$ or when $X=Y$ are connected.}
More examples will be presented in  \ref{MCC}.

Let $x, y\in A$ be two normal elements with $X={\rm sp}(x)$ and
 $Y={\rm sp}(y).$ Denote by  $\phi_X: C(X)\to A$ and
 $\phi_Y: C(Y)\to A$ {the homomorphisms} induced by $x$ and $y.$ We say the {\it  pair $(x, y)$
 has a hub at $X\cap Y,$} if the pair ${(\phi_X, \phi_Y)}$ has a hub at $X\cap Y.$

\begin{lem}\label{Lcuapp}
Let $A$ be a unital separable simple \CA\, with real rank zero, stable rank one and with weakly unperforated $K_0(A),$  and let $x,\, y\in A$ be two normal elements with $X={\rm sp}(x)$ and $Y={\rm sp}(y).$  For any $\ep>0,$  any finite subset ${\cal G}_X\subset C(X)_{s.a.}$ and any finite
subset ${\cal G}_Y\subset C(Y)_{s.a.},$
there exist mutually orthogonal projections
$\{e_1,e_2,...,e_n\}\subset A$ with $\sum_{i=1}^n e_i=1_A,$ $\lambda_1, \lambda_2,...,\lambda_n\in {\rm sp}(x)$ and $\mu_1, \mu_2,...,\mu_n\in {\rm sp}(y)$
such that
\beq\label{Lcuapp-1}
\max\{|\tau\circ g(x)-\tau\circ g(x_1)|: g\in {\cal G}_X\}<\ep/2,\\\label{Lcuapp-1+}
\max\{|\tau\circ g({y})-\tau\circ g(y_1)|: g\in {\cal G}_Y\}<\ep/2\tforal \tau\in T(A),\\\label{Lcuapp-2}
D_c(x, x_1)\le D_e^c(x,x_1)<\ep/2,\,\,\, D_c(y, y_1)\le D_c^e(y,y_1)<\ep/2,\\\label{Lcuapp-3}
D_c^e(x_1,y_1)<D_c^e(x,y)+\ep\tand\|x_1-y_1\|<D_c(x,y)+\ep,
\eneq
where
\beq\label{Lcuapp-4}
x_1=\sum_{i=1}^n \lambda_ie_i,\,\,\, y_1=\sum_{i=1}^n \mu_i e_i
\eneq
and
\beq\label{Lcuapp-4+}
\max_{1\le i\le n}|\lambda_i-\mu_i|\le D_c(x, y)+\ep
\eneq
Moreover, if $X\cap Y\not=\emptyset,$ for any $\sigma>0$ and $\eta>0,$
we may require that
\beq\label{Lcuapp-4++}
x_1=\sum_{i=1}^{m_0} \lambda_i e(i,0)+x_{1,1}\andeqn y_1=\sum_{i=1}^{m_0}\lambda_i e(i,0)+y_{1,1},\\\label{Lcuapp-4+++}
\sum_{i=1}^{m_0}\tau(e(i,0))<\sigma\tforal \tau\in T(A),
D_c(x_{1,1},y_{1,1})\le D_c^e(x, y)+\ep
\eneq
where $\{e(1,0),e(2,0),...,e(m_0,0)\}$ is a set of mutually orthogonal projections, $\{\lambda_1, \lambda_2,...,\lambda_{m_0}\}$ is $\eta$-dense
in $X\cap Y,$ $x_{1,1},y_{1,1}\in (1-\sum_{i=1}^{m_0}e(i,0))A(1-\sum_{i=1}^{m_0}e(i,0))$ are normal elements
with finite spectrum in $X$ and $Y,$ respectively,

In the above,  if $X=\sqcup_{i=1}^{m_1} F_j$ and
$Y=\sqcup_{k=1}^{m_2} G_k,$ where $F_1,F_2, ..., F_{m_1}$ are $\eta/2$-connected components of $X$ and $G_1, G_2,...,G_{m_2}$ are $\eta/2$-connected  components of $Y,$ we may assume that $\{\lambda_i\}$ is $\eta$-dense in $X$ and
$\{\mu_i\}$ is $\eta$-dense in {$Y.$ In} particular,  we may require
that $\{\lambda_i\}\cap F_j\not=\emptyset,$  $\{\mu_i\}\cap G_k\not=\emptyset.$
Moreover, if $F_j\cap G_k\not=\emptyset,$ we may further assume that there exist
{$\lambda_{i(j)}, \mu_{i(k)}\in F_j\cap G_k$  such that $|\lambda_{{i(j)}}-\mu_{i(k)}|<\eta.$}

Furthermore, if the pair $(x, y)$ has a hub at  $X\cap Y,$
then, we may require that,
 for each pair $(j,k)$ with $F_j\cap G_k\not=\emptyset,$
 there are ${\lambda_i}, \mu_{i}\in F_j\cap G_k.$


\end{lem}

\begin{proof}
The main part of this lemma follows from  \ref{Pepfinite}.
In fact that the existence of $x_1$ and $y_2$ satisfy everything up to (\ref{Lcuapp-2}) follows immediately from  \ref{Pepfinite}.
By the assumption, 
$$
D_c(x_1,y_1)\le D_c(x_1,x)+D_c(x,y)+D_c(y,y_1)<D_c(x, y)+\ep.
$$

Note that ${\rm sp}(x_1)\cap Y, Y\cap X\subset X.$ It follows from
 \ref{Pdced0} that
\beq\label{Lcuappn-1}
D_c^e(x_1, y)\le D_c^e(x_1,x)+D_c^e(x,y).
\eneq
Similarly
\beq\label{Lcuappn-2}
D_c^e(x_1,{ y_1})\le D_c^e(x_1,y)+D_c^e(y,y_1).
\eneq
Therefore
\beq\label{Lcuapp-3+}
D_c^e(x_1,y_1)\le D_c^e(x,y)+D_c^e(x_1, x)+D_c^e(y,y_1).
\eneq

{Applying \ref{MLfsp}, there exists a unitary $u$, such that $\|x_1-u^*y_1u\|<D_c(x,y)+\ep$. Let $y_1$ be replaced by $u^*y_1u$, we may assume (\ref{Lcuapp-3}) holds. 

Further  by applying the proof}
 of  \ref{MLfsp}, (\ref{Lcuapp-4}) and (\ref{Lcuapp-4+}) hold. 
The second part of the statement with (\ref{Lcuapp-4++}) and (\ref{Lcuapp-4+++})  follows from the definition of $D_c^e(\cdot,\cdot)$ and \ref{dcedig}.

The third and  fourth parts of the statement follow from  (\ref{Lcuapp-4++}) and the fact that the finite $\eta$-approximations
of ${\psi_X}$ and ${\psi_Y}$ can be made for arbitrarily small $\eta.$

Suppose, in addition, that  the pair $(x, y)$ has a hub at   $X\cap Y.$
For each pair $(j,k)$  with $F_j\cap G_k\not=\emptyset,$
we may assume that, by choosing sufficiently better finite approximation, without loss of generality,
that, there is ${\lambda_{i'}}\in F_j\cap G_k$ and there is $\mu_{i''}\in F_j\cap G_k.$ By the assumption that  the pair $(x, y)$ has a hub at  $X\cap Y$ and its definition, we may further assume,  there
are ${\lambda_i,\mu_i}\in F_j\cap G_k.$
\end{proof}

\begin{cor}\label{Ccuapp}
Let $A$ be a unital separable simple \CA\, of real rank zero,
stable rank one and weakly unperforated $K_0(A)$ and let $x, y\in A$ be  normal elements.
Then
\beq\label{Ccuapp-1}
D_c(x, y)\le D_c^e(x,y)\le 2D_c(x, y).
\eneq
\end{cor}

\begin{proof}
We will prove the second inequality.

Put $X={\rm sp}(x)$ and $Y={\rm sp}(y).$
 By  \ref{Lcuapp}, there are two sequences of normal elements $x_k, y_k\in A$
with
\beq\label{Ccuapp-2n}
&&\lim_{k\to\infty}D_c^e(x_k,x)=0,\,\,\,\lim_{k\to\infty}D_c^e(y_k, y)=0,\\
&&x_k=\sum_{i=1}^{m(k)}\lambda(k,i)p(k,i)\andeqn
y_k=\sum_{i=1}^{m(k)}\mu(k,i)p(k,i),
\eneq
where $\lambda(k,i)\in X$ and  $\mu(k,i)\in Y,$
\beq\label{Ccuapp-2+}
\lim_{k\to \infty}\max_{\{1\le i\le m(k)\}}|\lambda(k,i)-\mu(k,i)|\le D_c(x,y)
\eneq
and
$\{p(k,1),p(k,2),...,p(k,m(k))\}$ is a sequence of mutually orthogonal non-zero projections in $A$ with $\sum_{i=1}^{m(k)}p(k,i)=1_A,$ $k=1,2,....$
Note that
\beq\label{Ccuapp-3n}
\lim_{k\to\infty}D_c(x_k, y_k)=D_c(x,y).
\eneq
We may write
\beq\label{<2Dc-1-}
x_k=\sum_{i=1}^{m(k)}\lambda(k, i)p'(k,i)+\sum_{i=1}^{m(k)}\lambda(k, i)p''(k,i)\andeqn\\
y_k=\sum_{i=1}^{m(k)}\mu(k, i)p'(k,i)+\sum_{i=1}^{m(k)}\mu(k, i)p''(k,i),
\eneq
where $p'(k,i),\, p''(k,i)\not=0$ and $p'(k,i)+p''(k,i)=p(k,i),$ $i=1,2,...,m(k),$ $k=1,2,....$
To simplify the notation, we may write  that
\beq\label{<2Dc-1}
x_k=\sum_{i=1}^{2m(k)}\lambda(k, i)p(k,i)\andeqn
y_k=\sum_{i=1}^{2m(k)}\mu(k, i)p(k,i),\\\label{<2Dc-2}
\lambda(k,m(k)+i)=\lambda(k,i)\andeqn \mu(k,m(k)+i)=\mu(k,i),\,\,\,i=1,2,...,m(k),
\eneq
where $p(k,i)\not=0$ for all $i$ and $k.$
Without loss of generality,
we may assume that $S_k=\{\lambda(k,1), \lambda(k,2),...,\lambda(k,r(k))\}\subset X\cap Y$ and
$\lambda(k,j)\not\in Y,$ $r(k)<j\le 2m(k).$
Let\\ $T_k=\{\mu(k,\af(1)), \mu(k,\af(2)),...,\mu(k,{\af(g(k)))}\}\subset X\cap Y$ and $\mu(k,j)\not\in X$ if $j\not=\af(i)$ ($1\le i\le g(k)$).
We may also assume
that $S_k$ and $T_k$ both are  $\ep_k$-dense in $X\cap Y$ and $\lim_{k\to\infty} \ep_k=0.$
A standard argument allows us to assume, without loss of generality,
that $\lambda(k,i)=\mu(k,\af(i)),$ $i=1,2,...,f(k),$ where
$f(k)\le \min\{r(k), g(k)\}$ and
$W_k=\{\lambda(k,i): 1\le i\le f(k)\}$ is
$\dt_k$-dense in $X\cap Y$ and $\lim_{k\to\infty}\dt_k=0.$ By (\ref{<2Dc-1}) and (\ref{<2Dc-2}), we may assume that
$\af(i)>{f}(k),$ $i=1,2,..., f(k).$  In particular,
\beq\label{<2Dc-3}
\{\af(i): 1\le i\le f(k)\}\cap \{1,2,...,f(k)\}=\emptyset.
\eneq

Since $A$ is simple and has (SP),
there is a sequence of finite sets of non-zero projections
$e(k,i)\le p(k,i)$ such that
\beq\label{Ccuapp-4n}
&&p(k,i)-e(k,i)\not=0, [e(k,i)]=[e(k,1)],\,\,\,i=1,2,...,f(k),\\
&&\andeqn \sum_{i=1}^{m(k)}\tau(e(k,i))<1/k
\tforal \tau\in T(A),
\eneq$k=1,2,....$
Let $u_k\in U(A)$ be a sequence of unitaries
such that
\beq\label{Ccuapp-6n}
&&u_k^*e(k,i)u_k=e(k, \af(i)),\,\,\,
u_ke(k, \af(i))u_k^*=e(k,i),\,\,\,i=1,2,...,f(k),\\
&&u_k^*(p(k,i)-e(k,i))u_k=(p(k,i)-e(k,i)),\\
&&{u_k(p(k,\af(i))-e(k,\af(i)))u_k^*}
=(p(k,\af(i))-e(k,\af(i))),\\
&&\,\,\,i=1,2,...,f(k)\andeqn\\
&&u_k^*p(k,j)u_k=p(k,j),\,\,\, \text{if}\,\,\, j\not\in {\{i:1\le i\le f(k)\}}.
\eneq
Define
\beq\label{Ccuapp-7n}
x_{0,k}&=&\sum_{i=1}^{f(k)} \lambda(k,i)e(k,i),\\
x_{1,k}&=&\sum_{i=1}^{f(k)}\lambda(k,i)(p(k,i)-e(k,i))+\sum_{i=f(k)+1}^{2m(k)}\lambda(k,i)p_{k,i}
\eneq
\beq\label{<2Dc-4}
y_{0,k}&=&\sum_{i=1}^{f(k)}\mu(k,\af(i))e(k,\af(i))\\
y_{1,k}&=&\sum_{i=1}^{f(k)}\mu(k,i)(p(k,\af(i))-e(k,\af(i)))+
\sum_{i=f(k)+1}^{2m(k)}\mu(k,i)p(k,i).
\eneq
Note that
$$
u_ky_{0,k}u_k^*=\sum_{i=1}^{f(k)}\mu(k,\af(i))u_ke(k,\af(i))u_k^*=x_{0,k}.
$$
Note now that $\lambda(k,i)=\mu(k,\af(i)),$ $i=1,2,...,f(k),$
{$W_k$} is $\dt_k$-dense in $X\cap Y$ and
\beq\label{Ccuapp-8n}
\lim_{k\to\infty}\sup\{\tau(\sum_{i=1}^{f(k)}e(k,i)): \tau\in T(A)\}=0
\eneq
Moreover we still have  that 
\beq\label{Ccuapp-9}
\lim_{k\to\infty} D_c(x,x_{0,k}+x_{1,k})=0 \andeqn
\lim_{k\to\infty} D_c(y,y_{0,k}+y_{1,k})=0.
\eneq
Therefore
\beq\label{Ccuapp-10}
D_c^e(x,y)\le \liminf_{k\to\infty} D_c(x_{1,k}, y_{1,k}).
\eneq
Note also, by (\ref{<2Dc-3})
\beq\label{CCuapp-n10}
x_{1,k}
&=&
\sum_{i=1}^{f(k)}\lambda(k,i)(p(k,i)-e(k,i))+\sum_{i=1}^{f(k)}\lambda(k, \af(i))(p(k, \af(i))-e(k, \af(i)))\\&&+\sum_{i=1}^{f(k)}\lambda(k, \af(i))e(k,\af(i))+
\sum_{j\not\in\{i, \af(i): 1\le i\le f(k)\}}\lambda(k,j)p(k,j) \andeqn\\
y_{1,k}
&=& \sum_{i=1}^{f(k)}\mu(k,\af(i))(p(k,\af(i))-e(k,\af(i)))+\sum_{j\not\in\{\af(i):1\le i\le f(k)\}}\mu(k,j)
p(k,j)\\
&=&\sum_{i=1}^{ f(k)}\mu(k,i)(p(k,i)-e(k,i))
+\sum_{i=1}^{f(k)}\mu(k,\af(i))(p(k,\af(i))-e(k,\af(i)))\\
&&+\sum_{i=1}^{f(k)}\mu(k,i)e(k,i)
+\sum_{j\not\in \{i, \af(i): 1\le i\le f(k)\}}\mu(k,j)p(k,j).
\eneq
Since, for $i=1,2,..., f(k),$ $[e_{k,i}]=[e_{k, \af(i)}],$ $\lambda(k,i)=\mu(k, \af(i))$ and
\beq\label{Ccuapp-11}
&&\hspace{-0.4in}\limsup_{k\to\infty}|\lambda(k,\af(i))-\mu(k,i)|\\
&\le & \limsup_{k\to\infty}(|\lambda(k, \af(i))-\mu(k,\af(i))|+
|\mu(k, \af(i))-\mu(k,i)|)\\
&=& \limsup_{k\to\infty}|\lambda(k, \af(i))-\mu(k,\af(i))|+\limsup_{k\to\infty}
|\lambda(k, i)-\mu(k,i)|\\
&\le & D_c(x, y)+D_c(x, y).
\eneq
It follows that (by comparing the last expresses of $x_{1,k}$ and $y_{1,k}$ above and using (\ref{Ccuapp-2+}))
\beq\label{Ccuapp-12}
D_c(x_{1,k}, y_{1,k})\le 2D_c(x,y).
\eneq
Therefore, by (\ref{Ccuapp-10}),
\beq\label{Ccuapp-13}
D_c^e(x,y)\le 2D_c(x, y).
\eneq
\end{proof}

The following is a useful observation for the proof of the main results in this section.

\begin{lem}\label{Lspindx}
Let $A$ be a unital \CA\, and let $X$ and $Y$ be two compact subsets of the plane. Suppose that $x, y\in A$ are two normal elements with
${\rm sp}(x)=X$ and ${\rm sp}(y)=Y$ and suppose that $\phi_X: C(X)\to A$ and
$\phi_Y: C(Y)\to A$ are induced unital monomorphisms by $x$ and by $y,$
respectively.
Suppose also that $[\lambda-x]=[\lambda-y]$ in $K_1(A)$ for all
$\lambda\not\in X\cup Y.$
Then
\beq\label{Lspindx-1}
(\phi_X\circ \imath_1)_{*1}=0,
\eneq
where $I=\{f\in C(X): f|_{X\cap Y}=0\}$ and $\imath_1: I\to C(X)$
is the embedding.

{Consequently, if $X\cap Y=\emptyset,$ then  $I=C(X)$  and  $(\phi_X)_{*1}=0.$}
\end{lem}
In this case (\ref{Lspindx-1}) means that

\begin{proof}
Let $\pi_X: C(X\cup Y)\to C(X)$ and $\pi_Y: C(X\cup Y)\to C(Y)$ be
quotient maps. Define $\psi_1=\phi_X\circ \pi_X$ and
$\psi_2=\phi_Y\circ \pi_Y.$
The assumption implies that
\beq\label{Lspindx-2}
(\psi_1)_{*1}=(\psi_2)_{*1}.
\eneq
Put
$$
J=\{f\in C(X\cup Y): f|_Y=0\}.
$$
Note that $J\cong C_0(X\cup Y\setminus Y).$ But
$X\cup Y\setminus Y= X\setminus X\cap Y.$ Therefore
 there is a natural  isomorphism $h: I\to J.$  Let  $\imath_2: J\to C(X\cup Y)$
be the embedding.
Then
\beq\label{Lspindx-3}
\pi_X\circ \imath_2\circ h=\imath_1.
\eneq
Thus
\beq\label{Lspindx-4}
(\phi_X)_{*1}\circ (\imath_1)_{*1}&=&(\phi_X)_{*1}\circ (\pi_X)_{*1}\circ (\imath_2\circ h)_{*1}\\
&=& (\psi_2)_{*1}\circ (\imath_2\circ h)_{*1}.
\eneq
But
\beq\label{Lspindx-5}
\psi_2\circ \imath_2=0.
\eneq
Therefore
\beq\label{Lspindx-6}
(\phi_X\circ \imath_1)_{*1}=0.
\eneq
\end{proof}

\begin{lem}\label{LdecK}
Let $A$ be a unital  infinite dimensional simple \CA\, with real rank zero, stable rank one and with unperforated $K_0(A),$ let $x\in A$ be a normal
element and let $\phi_X: C(X)\to A$ be the unital monomorphism induced by $x,$ where
$X={\rm sp}(x).$  Suppose that $y\in A$ is another normal element
such that ${\rm sp}(y)=Y$ and
$[\lambda-x]=[\lambda-y]$ in $K_1(A)$ for all $\lambda\not\in X\cup Y$ and
suppose that $X\cap Y\not=\emptyset.$

Suppose also that functions $f_1, f_2,...,f_n\in C(X\cap Y)$ are $n$ mutually
orthogonal projections with $1_{C(X\cap Y)}=\sum_{i=1}^n f_i.$
Then,
for any non-zero projection
$e\in A$ {and}  any $n$ mutually orthogonal nonzero projections such that
$\sum_{i=1}^ne_i=e,$
there is a normal element $x_1\in eAe$ with ${\rm sp}(x_1)=X\cap Y$ satisfying
the following:
for any normal elements $x_0, y_0\in (1-e)A(1-e)$ with finite spectrum in $X$ and $Y,$ respectively,
\beq\label{LdecK-1}
f_i(x_1)&=&e_i,\,\,\,i=1,2,...,n,\\
(\psi_1)_{*1} &=&(\phi_{X})_{*1} \tand\\\label{Ldeck-1+}
 (\psi_2)_{*1}&=&(\phi_{Y})_{*1},
\eneq
where $\psi_1: C(X)\to A$ and $\psi_2: C(Y)\to A$ are defined by $\psi_1(f)=f(x_0+x_1)$ for all $f\in C(X)$ and
$\psi_2(f)=f(y_0+x_1)$ for all $f\in C(Y).$
\end{lem}

\begin{proof}
Let $\pi_X: C(X\cup Y)\to C(X),$ $\pi_Y: C(X\cup Y)\to C(Y),$
$\pi_{X\cap Y}^X: C(X)\to C(X\cap Y)$ and $\pi_{X\cap Y}^Y: C(Y)\to C(X\cap Y)$ be quotient maps.
By the Universal Coefficient Theorem, there are
$\kappa_1\in KK(C(X), A)$ such that $\kappa_1|_{K_1(C(X))}=-(\phi_X)_{*1}$ and $\kappa_1|_{K_0(C(X))}=0$ and $\kappa_2\in KK(C(Y), A)$ such that
$\kappa_2|_{K_1(C(Y))}=(\phi_Y)_{*1}$ and $\kappa_2|_{K_0(C(Y))}=0.$
Consider the pull back:
\beq\label{Ldecn-1}
\begin{array}{ccc}
C(X\cup Y) &\overset{\pi_X}{\longrightarrow} & C(X)\\
\Big\downarrow_{\pi_Y} & &\Big\downarrow_{\pi_{X\cap Y}^X}\\
C(Y) &\overset{\pi_{X\cap Y}^Y}{\longrightarrow} & C(X\cap Y).\\
\end{array}
\eneq
By a Mayer-Vietoris Theorem (see, for example, 21.5.1 of \cite{Bla}), one has the following  six-term exact sequence:
{\scriptsize
\beq\nonumber
\begin{array}{ccccc}
KK(C(X\cap Y), A) &\overset{(-[\phi_{X\cap Y}^X], [\phi_{X\cap Y}^Y])}{\longrightarrow} & KK(C(X),A)\oplus KK(C(Y), A)
&\overset{[\phi_X]+[\phi_Y]}{\longrightarrow} & KK(C(X\cup Y), A)\\
\Big\uparrow&&&&\Big\downarrow\\
KK^1(C(X\cup Y), A)&\overset{[\phi_X]+[\phi_Y]}{\longleftarrow}&KK^1(C(X),A)\oplus KK^1(C(Y),A)&\overset{(-[\phi_{X\cap Y}^X], [\phi_{X\cap Y}^Y])}{\longleftarrow}&KK^1(C(X\cap Y), A).\\
\end{array}
\eneq
}

 By the assumption and the proof of  \ref{Lspindx},
 $$
 (\phi_X)_{*1}\circ (\pi_X)_{*1}=(\phi_X\circ \pi_X)_{*1}=(\phi_Y\circ \pi_Y)_{*1}=(\phi_Y)_{*1}\circ (\pi_Y)_{*1}.
 $$
 It follows that
 \beq\label{Ldecn-3}
 ([\phi_{X\cap Y}^X]+[\phi_{X\cap Y}^Y])(\kappa_1, \kappa_2)=0.
 \eneq
 The exactness of the Mayer-Vietoris sequence above shows that there
 is $\kappa_3\in KK(C(X\cap Y), A)$ such that
 \beq\label{Ldecn-4}
 (-[\phi_X], [\phi_Y])(\kappa_3)=(\kappa_1, \kappa_2),
 \eneq
or
\beq\label{Ldecn-5}
-[\phi_{X\cap Y}^X](\kappa_3)=\kappa_1 \andeqn [\phi_{X\cap Y}^Y](\kappa_3)=\kappa_2.
\eneq
Let $\kappa_3|_{K_1(C(X\cap Y))}=\lambda$ be as an element in $Hom(K_1(C(X\cap Y), K_1(A)).$
Then (\ref{Ldecn-5}) implies
that
\beq\label{Ldecn-6}
\lambda\circ (\phi_{X\cap Y}^X )_{*1}=(\phi_X)_{*1}\andeqn \lambda\circ (\phi_{X\cap Y}^Y)_{*1}=(\phi_Y)_{*1}
\eneq
It follows from  \ref{Lext}  that there is
a unital monomorphism $\psi_1': C(X\cap Y)\to eAe$ such that
\beq\label{Ldeck-9}
(\psi_1')_{*1}={\lambda}\andeqn \psi_1'(f_i)=e_i,\,\,\, i=1,2,...,n.
\eneq
Let $z_{X\cap Y}\in C(X\cap Y)$ be the identity function on $X\cap Y.$
Choose $x_1\in eAe$ such that $x_1=\psi_1'(z_{X\cap Y}).$
Choose any normal element $x_0\in B\subset eAe,$ where
$B$ is a finite dimensional \SCA\, of $eAe$ with $1_B=e$ such that
${\rm sp}(x_0)\subset X.$
Define $\psi'': C(X)\to (1-e)A(1-e)$ by
$\psi''(f)=f(x_1)$ for all $f\in C(X)$ and define $\psi_1: C(X)\to A$ by
$$
\psi_1(f)=f(x_0+x_1).
$$
Then, since $x_0\in B,$  $\psi''(f)\in B$ for all $f\in B.$ It follows
that $(\psi'')_{*1}=0.$
Therefore, by (\ref{Ldecn-6}),
\beq\label{Ldeck-10}
(\psi_1)_{*1}={\lambda}\circ (\pi_{X\cap Y}^X)_{*1}=(\phi_X)_{*1}.
\eneq
Define $\psi_2: C(Y)\to A$ by $\phi_2(g)=g(x_1+y_0)$ for all $g\in C(Y)$ and for any
normal element $y_0\in B$ with ${\rm sp}(y_0)\subset Y.$
We also have
\beq\label{Ldeck-11}
(\psi_2)_{*1}={\lambda}\circ (\pi_{X\cap Y}^Y)_{*1}=(\phi_Y)_{*1}.
\eneq

\end{proof}

Let $A$ be a unital simple \CA\, with $T(A)\not=\emptyset.$
Let $\phi_X: C(X)\to A$
be a unital monomorphism. Denote by
$\phi_X^{\sharp}: C(X)_{s.a.}\to \Aff(T(A))$  the unital affine continuous map
induced by $\phi.$
If $s: C(X)\to \C$ is a state of $C(X),$ denote by $\mu_s$ the probability Borel measure
induced by $s.$

\begin{thm}\label{MT1}
Let $A$ be a unital separable simple \CA\, with real rank zero, stable rank one and with weakly unperforated $K_0(A).$  Let $x, \, y\in A$ be two normal
elements with ${\rm sp}(x)=X$ and ${\rm sp}(y)=Y.$  Denote $Z=X\cup Y.$
Suppose that $[\lambda-x]=[\lambda-y]$ in $K_1(A)$
for all $\lambda\not\in Z.$

\noindent
{\rm (1)}\,\,\, Then
\beq\label{MT1-1}
{\rm dist}({\cal U}(x), {\cal U}(y))\le D_c^e(x, y).
\eneq

\noindent
{\rm (2)}\,\,\, Moreover, if the pair $(x, y)$ has a hub at $X\cap Y,$ then
\beq\label{MT1-1n}
{\rm dist}({\cal U}(x), {\cal U}(y))\le D_c(x, y).
\eneq
\end{thm}

\begin{proof}
Denote by $\phi_X: C(X)\to A$ the unital monomorphism defined
by $\phi_X(f)=f(x)$ for all $f\in C(X)$ and
$\phi_Y: C(Y)\to A$ defined by $\phi_Y(f)=f(y)$ for all $f\in C(Y).$
Let $\ep>0.$  Let ${\cal V}_1\subset K_1(C(X))$ (in place of ${\cal V}$) be a finite subset,
${\cal P}_1\subset K_0(C(X))$ (in place of ${\cal P})$ be a finite subset,
${\cal H}_1\subset C(X)_{s.a.}$ (in place of ${\cal H}$) be a finite subset and
let $\sigma_1>0$ (in place of $\sigma$) be required by  \ref{Uni2} for $\ep/16$ and
$x.$

Let  ${\cal V}_2\subset K_1(C(Y))$ (in place of ${\cal V}$) be a finite subset,
${\cal P}_2\subset K_0(C(Y))$ (in place of ${\cal P})$ be a finite subset,
${\cal H}_2\subset C(Y)_{s.a.}$ (in place of ${\cal H}$) be a finite subset and
let $\sigma_2>0$ (in place of $\sigma$) be required by  \ref{Uni2} for $\ep/16$ and
$y.$

Without loss of generality, we may assume that ${\cal H}_1$ and ${\cal H}_2$ are in the unit balls
of $C(X)$ and $C(Y),$ respectively.
Moreover, we may assume
that
\beq\label{MT1p-1-1}
{\cal P}_1=\{f_1, f_2,...,f_{m_1}\}\andeqn
{\cal P}_2=\{g_1,g_2,...,g_{m_2}\},
\eneq
where $f_i\in C(X)$ and $g_i\in C(X)$ which are projections
such that
\beq\label{MT1p-1-2}
1_{C(X)}=\sum_{i=1}^{m_1}f_i\andeqn 1_{C(Y)}=\sum_{i=1}^{m_2}g_i.
\eneq

Denote by $F_1, F_2,...,F_{m_1}$ the clopen sets of $X$ corresponding to
projections $f_1, f_2,...,f_{m_1},$ and by $G_1, G_2,...,G_{m_2}$ the clopen subsets of $Y$ corresponding to projections $g_1,g_2,...,g_{m_2}.$

Let $X{\cap} Y=\sqcup_{j=1}^I S_j,$ where $S_1, S_2,...,S_I$ are  distinct $\ep/8$-connected components of $X\cap Y.$ In particular, ${\rm dist}(S_I, S_j)\ge \ep/8.$
If $X\cap Y=\emptyset,$ {this} notation simply means $I=0$ (see \ref{Lspindx}).

Let $\eta>0$ be such that $\eta<\ep.$
By applying  \ref{Lcuapp}, there are mutually orthogonal projections
$e_1,e_2,...,e_n\in A$  with $\sum_{i=1}^n e_i=1_A,$
$\lambda_1, \lambda_2,...,\lambda_n\in X$ and $\mu_1,\mu_2,...,\mu_n\in Y$ such that
\beq\label{MT1p-1}
&&\max\{|\tau(g(x))-\tau(g(x_1))|: g\in {\cal H}_1\} <\sigma_1/2\tforal \tau\in T(A),\\\label{MT1p-2}
&&\max\{|\tau(g(y))-\tau(g(y_1))|: g\in {\cal H}_2\} <\sigma_2/2\tforal \tau\in T(A),\\\label{MT1p-3}
&&D_c(x, x_1)\le D_c^e(x,x_1)<\eta/4,\\\label{MT1p-4}
&&D_c(y, y_1)\le D_c^e(y,y_1)<\eta/4,\\\label{MT1p-5}
&&D_c^e(x_1,y_1)<D_c^e(x,y)+\eta/4\andeqn \|x_1-y_1\|<D_c(x, y)+\eta/4\\\label{MT1p-5+}
&&{\max_{1\le i\le n} |\lambda_i-\mu_i| <D_c(x, y)+\eta/4,}
\eneq
where
\beq\label{MT1p-6}
x_1=\sum_{i=1}^n\lambda_i e_i\andeqn y_1=\sum_{i=1}^n \mu_ie_i.
\eneq
By  the proof of   \ref{Lcuapp} (by choosing even smaller $\dt$) and by  \ref{Pcsameclose}, we may assume that
\beq\label{MT1p-6+}
[f_i(x_1)]=[f_i(x)],\,\,\,i=1,2,...,m_1\andeqn\\\label{MT1p-6++}
[g_j(y_1)]=[g_j(y)]\,\,\,\, j=1,2,...,m_2
\eneq

We first consider {case} (1).
Since the case that  $X\cap Y=\emptyset$ will be dealt with in case {(2)}, we will assume that $X\cap Y\not=\emptyset.$
By the second part of  \ref{Lcuapp}, we may also assume that,
\beq\label{MT1nn-1}
&&x_1=\sum_{i=1}^I\lambda_ie_i^{(0)}+x_2',\,\,\, y_2=\sum_{i=1}^I\lambda_ie_i^{(0)}+y_2',\\\label{MT1nn-2}
&&D_c(x_2',y_2')<D_c^e(x, y)+\ep/4,\,\,\,
\tau(\sum_{i=1}^Ie_i^{(0)})<\min\{\sigma_1/2, \sigma_2/2\}
\eneq
for all $ \tau\in T(A),$
where $\{e_1^{(0)}, e_2^{(0)},...,e_I^{(0)}\}$ is a set of mutually orthogonal non-zero projections, $\lambda_i\in S_i,$ $i=1,2,...,I,$
$x_2',y_2'\in (1-p_0)A(1-p_0)$ are normal elements with finite spectrum in
$X$ and $Y,$ respectively, and where $p_0=\sum_{i=1}^Ie_i^{(0)}.$

Let $h_j=\chi_{S_j}\in C(X\cap Y),$ $j=1,2,...,I.$

By applying  \ref{LdecK}, there is a normal element $x_0\in p_{0}Ap_{0}$ with
${\rm sp}(x_0)=X\cap Y$ such that
\beq\label{MT1p-9}
(\psi_1)_{*1}|_{{\cal V}_1}&=&(\phi_X)_{*1}|_{{\cal V}_1},\\\label{MT1p-10}
(\psi_2)_{*1}|_{{\cal V}_2}&=&(\phi_Y)_{*1}|_{{\cal V}_2},\\\label{MT1p-10+}
h_j(x_0)&=& e_{j}^{(0)},\,\,\, j=1,2,...,I,
\eneq
where {$\psi_1: C(X)\to p_{0}Ap_{0}$} is defined by
$\psi_1(f)=f(x_0)$ for all $f\in C(X)$ and
{$\psi_2: C(Y)\to p_{0}Ap_{0}$} is defined by $\psi_2(f)=f(x_0)$ for all $f\in C(Y).$
Now consider $x_3=x_0+x_2'$ and $y_3=x_0+y_2'.$
Note that ${\rm sp}(x_3)\subset X$ and ${\rm sp}(y_3)\subset Y.$
Define $\psi_3: C(X)\to A$ by $\psi_3(f)=f(x_3)$ for all $f\in C(X)$ and
$\psi_4: C(Y)\to A$ by $\psi_4(f)=f(y_3)$ for all $f\in C(Y).$
Since $x_2'$  and $y_2'$ have  finite spectra, by (\ref{MT1p-9}) {and (\ref{MT1p-10})}, we have
\beq\label{MT1p-11}
(\phi_X)_{*1}|_{{\cal V}_1}=(\psi_3)_{*1}|_{{\cal V}_1}\andeqn
(\phi_Y)_{*1}|_{{\cal V}_2}=(\psi_4)_{*1}|_{{\cal V}_2}.
\eneq
For each $i,$ if $F_i\cap Y=\emptyset,$
i.e., $F_i\cap G_k=\emptyset$ for all $k,$
we compute that
\beq\label{MT1p-12}
[\psi_3(f_i)]=[f_i(x_3)]=[f_i(x_2')]=[f_i(x_1)]=[f_i(x)]
\eneq
$i=1,2,...,I.$
If $F_i\cap Y\not=\emptyset,$ let $H_i$
be the subset of $\{j: j=1,2,...,I\}$ such that $h_j\le f_i.$
We then have
\beq\label{MT1p-13}
[\psi_3(f_i)]=\sum_{j\in H_i}[e_{j}^{(0)}]+[f_i(x_2')]=[f_i(x_1)]=[f_i(x)].
\eneq
Similarly, if $G_i\cap X=\emptyset,$
\beq\label{MT1nn-4}
[\psi_4(g_i)]=[g_i(y_3)]=[g_i(y_2')]=[g_i(y_1)]=[g_i(y)].
\eneq
If $G_i\cap X\not=\emptyset,$ let $H_i'$ be the subset of $\{j: j=1,2,...,I\}$ such that $h_j\le g_i.$ We then have
\beq\label{MT1nn-5}
[\psi_4(g_i)]=\sum_{j\in H_i'}[e_j^{(0)}]+[g_i(y_2')]=[g_i(y_1)]=[g_i(y)].
\eneq
In other words,
\beq\label{MT1p-14}
(\psi_3)_{*0}|_{{\cal P}_1}=(\phi_X)_{*0}|_{{\cal P}_1}\andeqn
(\psi_4)_{*0}|_{{\cal P}_2}=(\phi_Y)_{*0}|_{{\cal P}_2}.
\eneq
By applying  \ref{Uni2}, using (\ref{MT1p-11}), (\ref{MT1p-14}),
(\ref{MT1p-1}), (\ref{MT1p-2}) and (\ref{MT1nn-2}), we obtain a unitary  $u_1, u_2\in A$ such that
\beq\label{MT1n-21}
\|u_1^*xu_1-x_3\|<\ep/16\andeqn \|u_2^*yu_2-y_3\|<\ep/16.
\eneq
By (\ref{MT1nn-2}) and by  \ref{MLfsp} there is a unitary $u_3\in (1-p_0)A(1-p_0)$ such that
\beq\label{MT1n-22}
\|u_3^*x_2'u_3-y_2'\|<D_c^e(x,y)+\ep/16.
\eneq
Put $u_4=p_0+u_3.$ Then
\beq\label{MT1n-23}
\|u_4^*x_3u_4-y_3\|&=&\|(x_0+u_3^*x_2'u_3)-(x_0+y_2')\|=\|u_3^*x_2'u_3-y_2'\|\\
&<& D_c^e(x, y)+\ep/16.
\eneq
Therefore
\beq\label{MT1n-24}
{\rm dist}({\cal U}(x), {\cal U}(y))<D_c^e(x, y)+5\ep/8.
\eneq
This proves the case (1).

Now we turn to case (2).

If $X\cap Y=\emptyset,$ by the assumption that
$[\lambda-x]$ and $[\lambda-y]$ are the same in $K_1(A)$ for all $\lambda\not\in X\cup Y,$
$\lambda-x\in {\rm Inv}_0(A)$ for all $\lambda\not\in X$ and $\lambda-y\in {{\rm Inv}}_0(A)$ for all
$\lambda\not\in Y.$ Thus this special case has been
proved in  \ref{Toriginal}.

Thus we will then assume again $X\cap Y\not=\emptyset.$
Some of the argument above will be repeated.
In this case,  we may assume that, if $F_j\cap G_k\not=\emptyset,$ there are
at least one $i$ such that  $\lambda_i, \mu_i\in F_j\cap G_k.$
Note that, in this case,  $F_j\cap G_k$ is a non-empty clopen subset of $X\cap Y.$ In fact $X\cap Y$ is a disjoint union of those
$F_j\cap G_k.$ Call them $T_1, T_2,...,T_k.$ Then $k\le n.$
We may assume that $\{r(1), r(2),...,r(k)\}\subset \{1,2,...,n\}$ such that
$\lambda_{r(j)}, \mu_{r(j)}\in T_j,$ $j=1,2,...,k.$

Since $A$ is simple and infinite dimensional, one can find a non-zero
projection $e_{r(j)}^{(0)}\le e_{r(j)}$ such that
$[e_{r(j)}^{(0)}]=[e_{r(1)}^{(0)}]$ in $K_0(A),$ $j=1,2,...,r_k,$ and
\beq\label{MT1p2-7}
\tau(\sum_{j=1}^ke_{r(j)}^{(0)})<\min\{\sigma_1/2,\sigma_2/2\}\tforal \tau\in T(A).
\eneq
Let $p_0=\sum_{j=1}^k e_{r(j)}^{(0)}, $ $p=1_A-p_0,$    $p_i=e_i-e_i^{(0)},$ if $i\in \{r(1),r(2),...,r(k)\}$ and $p_i=e_i,$ if
$i\not\in \{r(1),r(2),...,r(k)\}.$
Put
\beq\label{MT1p2-8}
x_2=\sum_{i=1}^n \lambda_ip_i\andeqn y_2=\sum_{i=1}^n \mu_i p_i.
\eneq
{Note,  by (\ref{MT1p-5+}), that
\beq\label{MT1p2-8+}
\|x_2-y_2\|\le \max_{1\le i\le n}|\lambda_i-\mu_i|<D_c(x, y)+\eta/4.
\eneq
}
Let $h_j'=\chi_{T_j},$ $j=1,2,...,k.$

By applying  \ref{LdecK}, there is a normal element $x_0\in p_0Ap_0$ with
${\rm sp}(x_0)=X\cap Y$ such that
\beq\label{MT1p2-9}
(\psi_1)_{*1}|_{{\cal V}_1}&=&(\phi_X)_{*1}|_{{\cal V}_1},\\\label{MT1p2-10}
(\psi_2)_{*1}|_{{\cal V}_2}&=&(\phi_Y)_{*1}|_{{\cal V}_2},\\\label{MT1p2-10+}
h_i'(x_0)&=& e_{r(i)}^{(0)},\,\,\, i=1,2,...,k,
\eneq
where $\psi_1: C(X)\to p_0Ap_0$ is defined by
$\psi_1(f)=f(x_0)$ for all $f\in C(X)$ and
$\psi_2: C(Y)\to p_0Ap_0$ is defined by $\psi_2(f)=f(x_0)$ for all $f\in C(Y).$
Now consider $x_3=x_0+x_2$ and $y_3=x_0+y_2.$
Note that ${\rm sp}(x_3)\subset X$ and ${\rm sp}(y_3)\subset Y.$
Define $\psi_3: C(X)\to A$ by $\psi_3(f)=f(x_3)$ for all $f\in C(X)$ and
$\psi_4: C(Y)\to A$ by $\psi_4(f)=f(y_3)$ for all $f\in C(Y).$
Since $x_2$  and $y_2$ have  finite spectra, by (\ref{MT1p-9}), we have
\beq\label{MT1p2-11}
(\phi_X)_{*1}|_{{\cal V}_1}=(\psi_3)_{*1}|_{{\cal V}_1}\andeqn
(\phi_Y)_{*1}|_{{\cal V}_2}=(\psi_4)_{*1}|_{{\cal V}_2}
\eneq
For each $i,$ if $F_i\cap Y=\emptyset,$ i.e., $i\not\in \{r(1),r(2),...,r(k)\},$
we compute that
\beq\label{MT1p2-12}
[\psi_3(f_i)]=[\psi_1(f_i)]=[\phi_X(f_i)].
\eneq
If $F_i\cap Y\not=\emptyset,$ we also have
\beq\label{MT1p2-13}
[\psi_3(f_{i})]&=&\sum_{h_j'\le f_i}[e_{r(j)}^{(0)}]+(\sum_{h_j'\le f_i}[e_{r(j)}-e_{r(j)}^{(0)}]+\sum_{\lambda_j\in F_{i}, j\not=r(j)}[e_i])\\
&=& [\sum_{\lambda_i\in F_{i}}e_j]=[\phi_X(f_{i})],
\eneq
$j=1,2,...,k.$
In other words,
\beq\label{MT1p2-14}
(\psi_3)_{*0}|_{{\cal P}_1}=(\phi_X)_{*0}|_{{\cal P}_1}.
\eneq
By applying  \ref{Uni2}, using (\ref{MT1p2-14}), (\ref{MT1p2-11}) and
(\ref{MT1p-1})  and (\ref{MT1p2-7}), we obtain a unitary such that
\beq\label{MT1p-15}
\|u^*xu-x_3\|<\ep/16.
\eneq
On the other hand, for each $j,$
$\psi_2(g_j)=\sum_{h_i'\le g_j}h_i'(x_0).$

It follows that
\beq\label{MT1p-16-}
[\psi_2(g_j)]=\sum_{\mu_{r(i)}\in G_j}[e_{r(i)}^{(0)}].
\eneq
Therefore,
\beq\label{MT1p2-16}
[\psi_4(g_j)] &=& [(\psi_2)(g_j)]+[g_j(y_2)]\\
&=&{ \sum_{\mu_{r(i)}\in G_j}[ e_{r(i)}^{(0)}]
+\sum_{\mu_{r(j)}\in G_j}[e_{r(i)}-e_{r(i)}^{(0)}]+\sum_{\mu_s\in G_j,s\not\in\{r(i): 1\le i\le r(k)\}} [e_s]}\\
&=&\sum_{\mu_s\in G_j}[e_s]
=[\phi_Y(g_j)].
\eneq
It follows that
\beq\label{MT1p2-18}
(\psi_4)_{*0}|_{{\cal P}_2}=(\phi_Y)_{*0}|_{{\cal P}_2}.
\eneq
It follows from (\ref{MT1p2-18}), (\ref{MT1p2-11}), (\ref{MT1p-2}), (\ref{MT1p2-7}) and  \ref{Uni2} that there is a unitary $v\in A$ such that
\beq\label{MT1p2-19}
\|v^*yv-y_3\|<\ep/16.
\eneq
We also have {(using (\ref{MT1p2-8+}))}
\beq\label{MT1p2-20}
\|x_3-y_3\|=\|(x_0+x_2)-(x_0+y_2)\|=\|x_2-y_2\|<D_c(x, y)+\ep/4.
\eneq
It follows that
\beq\label{MT1p2-21}
\|u^*xu-v^*yv\|<\ep/16+D_c(x, y)+\ep/4+\ep/16<
D_c(x, y)+\ep/2.
\eneq
Therefore
\beq\label{MT1p2-22}
{\rm dist}({\cal U}(x), {\cal U}(y))<D_{c}(x, y)+\ep/2
\eneq
for all $\ep>0.$ The theorem follows.
\end{proof}

\begin{rem}\label{RM1}
{\rm
It is probably helpful to be reminded that
$$
D_c^e(x, y)\le \min\{D^T(x, y),2D_c(x,y)\}.
$$
and if  both $X$ and $Y$ are connected, $D_c^e(x,y)=D_c(x,y).$
}
\end{rem}

\begin{cor}\label{MCdig}
Let $A$ be a unital separable  simple infinite dimensional \CA\, with real rank zero, stable rank one and weakly unperforated $K_0(A)$ and let $x\in A$ be a normal element with ${\rm sp}(x)=X.$
Then, for any $\ep>0,$ any $\sigma>0$ and any finite subset $\{{\xi_1}, \xi_2,...,\xi_k\}\subset  {\rm sp}(x),$ there is a set of mutually
orthogonal non-zero projections $\{e_1, e_2,...,e_k\}$ of $A$ and a
normal element $x_0\in (1-p)A(1-p)$ with ${\rm sp}(x_0)=X$ such that
\beq\label{MCdig-1}
\|x-(x_0+\sum_{i=1}^k\xi_i e_i)\|<\ep,\\\label{MCdig-2}
\tau(\sum_{i=1}^ke_i)<\sigma\tforal \tau\in T(A),\\\label{MCdig-3}
(\phi_1)_{*1}=(\phi_2)_{*1},
\eneq
where $p=\sum_{i=1}^k e_i,$ $\phi_1, \phi_2: C(X)\to A$
is defined by $\phi_1(f)=f(x)$ and
$\phi_2(f)=f(x_0)+\sum_{i=1}^k f(\xi_i)e_i$ for all $f\in C(X).$
\end{cor}

\begin{proof}
 {This}  is merely a refinement of that of  \ref{Ldig}. The issue is that we now insist that ${\rm sp}(x_0)=X.$
The proof is contained in the proof of the case (1) in the proof of  \ref{MT1} { by letting $X=Y$}. With the notation in the proof of the case (1) in  \ref{MT1},
we have $x_3=x_0+x_2',$ where $x_2'$ has finite spectrum
but ${\rm sp}(x_2')$ can be $\ep/16$-dense in $X.$ Note that ${\rm sp}(x_0)=X.$ We may assume, without loss of generality, that
$x_2'=\sum_{i=1}^k \xi_k p_i+x_2'',$
where $\{p_1, p_2,...,p_k\}$ is a set of mutually orthogonal non-zero projections and
where $x_2''$ is a normal element in $(1-\sum_{i=1}^k p_i)A(1-\sum_{i=1}^k p_i)$ with ${\rm sp}(x_2'')\subset X.$
Since $A$ is simple and has the property (SP), there are non-zero
projections $e_i'\le p_i,$ $i=1,2,...,k,$ such that
\beq\label{MCdign-1}
\sum_{i=1}^k\tau(e_i')<\sigma\tforal \tau\in T(A).
\eneq
We still have, as in (\ref{MT1n-21}), a unitary $u_1\in A$ such that 
\beq\label{MCdign-2}
\|u_1^*xu_1-x_3\|<\ep/16
\eneq
 Then
we have
\beq\label{MCdign-3}
\|x-(u_1(x_0+x_2'')u_1^* +\sum_{i=1}^k\xi_i u_1e_i'u_1^*)\|<\ep/16.
\eneq
Choose the new $x_0$ to be $u_1(x_0+x_2'')u_1^*$ and $e_i$ to be $u_1e_i'u_1^*.$

\end{proof}

\begin{cor}\label{MCC}
Let $A$ be a unital separable  simple \CA\, of real rank zero, stable rank one and weakly unperforated $K_0(A),$ let $x, y\in A$ be two normal
elements with ${\rm sp}(x)=X$ and ${\rm sp}(y)=Y.$
Then the pair $(x, y)$ has hub at $X\cap Y,$ if one of the following holds:

{\rm (1)} $X=Y$  is connected;

{\rm (2)} $X\cap Y$ is connected and it contains an open ball with radius $D_c(x,y);$

{\rm (3)}  for every connected component $S$ of $X,$  either $S=X\cap Y$ or
 ${\rm dist}({\xi}, X\cap Y)> D_c(x,y)$ for all $\xi\in S;$

{\rm (4)} $X\cap Y=\emptyset.$

Consequently,  
$$
{\rm dist}({\cal U}(x), {\cal U}(y))\le D_c(x, y),
$$
if one of the above conditions holds.
\end{cor}

\begin{proof}
It is clear that  (1) and (4) follow from the definition immediately.
It is also clear that (3) holds since $X\cap Y$ must be  connected and no point
in $X\cap Y$  can pair with any point outside $X\cap Y$ with a distance
no more than $D_c(x, y).$

To see (2), let $X=\sqcup_{i=1}^{m_1} F_i$ and
$Y=\sqcup_{j=1}^{m_2} G_j,$ where $\{F_1,F_2,...,F_{m_1}\}$ and
$\{G_1, G_2,...,G_{m_2}\}$ are of mutually disjoint clopen sets.
Since $X\cap Y$ is connected,  we may assume that $X\cap Y=F_1\cap G_1,$ { and }
 $F_i\cap G_j=\emptyset,$ if $(i,j)\not=(1,1).$

Let
$$
d=\min\{{\rm dist}(G_1, G_j): j=2,3,...,m_2\}>0.
$$
Choose $\ep_0=d/16.$   Let  $0<\ep<\ep_0.$
Suppose that
$F=\{\lambda_1,\lambda_2,..., \lambda_K\}\subset X$ and $G=\{\mu_1,\mu_2,...,\mu_L\}\subset Y$
are finite subsets,  $x_1=\sum_{i=1}^K \lambda_ke_i$  and
$y_1=\sum_{j=1}^L \mu_jp_j,$
where $\{e_1,e_2,...,e_K\}$ and $\{p_1,p_2,...,p_L\}$ are two sets of mutually orthogonal projections in $A$ such that $\sum_{i=1}^Ke_i=1_A=\sum_{j=1}^Lp_j,$
and such that
\beq\label{MCC-10}
D_c(x_1, x)<\ep\andeqn D_c(y_1, y)<\ep.
\eneq
Therefore
\beq\label{MCC-11}
D_c(x, y)-2\ep<D_c(x_1, y_1)<D_c(x, y)+2\ep.
\eneq
Then there is $\lambda_i\in F_1\cap G_1=X\cap Y$ such that
$$
{\rm dist}(\lambda_i, G_j)>(D_c(x, y)-\ep)+d> D_c(x_1,y_1)+d-3\ep>D_c(x,y).
$$
for all $j\not=1.$  In other words there is $j$ such that
$\mu_j\in F_1\cap G_1$ and $(i,j)\in {R_{x_1, y_2}}$ (see \ref{Dparing}).
Therefore the pair $(x,y)$ has a hub at $X\cap Y.$

\end{proof}

\begin{cor}
Let $A$ be a unital separable  simple \CA\, of real rank zero, stable rank one and weakly unperforated $K_0(A),$ let $x, y\in A$ be two normal
elements with ${\rm sp}(x)=X$ and ${\rm sp}(y)=Y.$
If $X$ or $Y$ is connected and if $[\lambda -x]=[\lambda -y]$ in $K_1(A)$ for all $\lambda \not\in X\cup Y$, then
$$
{\rm dist}({\cal U}(x), {\cal U}(y))\le D_c(x, y).
$$
\end{cor}
\begin{proof}
If $X$ or $Y$ is connected, by \ref{dceconnect},
$D_c(x, y)=D_c^e(x,y).$ The corollary follows from  \ref{MT1}.

\end{proof}
\section{Distance {between} unitary orbits of normal elements
with different $K_1$ maps}

In this section we will show that, without the condition
that $[\lambda-x]=[\lambda-y]$ in $K_1(A)$ {for
$\lambda\not\in X\cup Y$} in the statement of \ref{MT1},
$D_c(x,y)$ alone may have little to do with  the distance of the unitary orbits of $x$ and $y$ as \ref{Knot=1} shows. However,   \ref{MT2}
provides us some {description} of the upper as well as lower bound for the distance
between unitary orbits of normal elements.

Let $A$ be a unital \CA\, and let $x,\, y\in A$ be two normal elements.
Let $X={\rm sp}(x)$ and $Y={\rm sp}(y).$
Denote by $d_H(X,Y)$ the Hausdorff distance between the subset
$X$ and $Y.$
Define
\beq\label{Drhodist}
\hspace{-0.2in}\rho(x, y)=\max\{d_H(X, Y), \rho_1(x,y)\},
\eneq
where
\beq\label{Drho1}
\rho_1(x,y)=\sup\{{\rm dist}(\lambda, X)+{\rm dist}(\lambda, Y): \lambda\not\in X\cup Y, (\lambda-x)(\lambda-y)^{-1}\not\in {\rm Inv}_0(A)\}.
\eneq
Let
\beq\label{Dhx}
\rho_x(x,y)&=&\sup\{{\rm dist}(\lambda, X): \lambda\not\in X\cup Y,
(\lambda-x)^{-1}(\lambda-y)\not\in {\rm Inv}_0(A)\}\andeqn\\
\rho_y(x,y)&=&\sup\{{\rm dist}(\lambda, Y): \lambda\not\in X\cup Y,
(\lambda-x)^{-1}(\lambda-y)\not\in {\rm Inv}_0(A)\}.
\eneq

The following is a result of Ken Davidson (\cite{Dv2}). The proof is exact the same as that in \cite{Dv2}.

\begin{prop}\label{Dvpro}
Let $A$ be a unital \CA\, and let $x, \, y\in A$ be two normal elements.
Then
\beq\label{Dvpro-1}
{\rm dist}({\cal U}(x), {\cal U}(y))\ge \rho(x, y).
\eneq
\end{prop}

\begin{thm}\label{Knot=1}
Let $A$ be a unital, infinite dimensional, separable  simple
\CA\, with  real rank zero, stable rank one,  weakly unperforated $K_0(A)$ and
$K_1(A)\not=\{0\}.$
Then

{\rm (1)} for any unitary $u_1\in A$  with ${\rm sp}(u_1)=\T$ (the unit circle),
there is a   unitary $u_2\in A$ such that $[u_1]\not=[u_2]$ in $K_1(A),$
\beq\label{Knot=1-1}
D_c(u_1, u_2)=0\tand {\rm dist}({\cal U}(u_1), {\cal U}(u_2))=2;
\eneq

{\rm (2)} For any compact subset $X\subset \C$ such that $\C\setminus X$ is not connected and  for any
normal element $x\in A$ with ${\rm sp}(x)=X,$ there exists a
normal element $y\in A$ such that
\beq\label{Knot=1-2}
&&{\rm dist}({\cal U}(x),{\cal U}(y))\\
&&\ge 2\sup\{{\rm dist}(\lambda, {\rm sp}(x)):
\lambda\,\,\, {\rm in\,\,\, bounded\,\,\, components\,\,\,\, of }\,\,\, \C\setminus {\rm sp}(x)\}\\
&&\tand D_c(x, y)=0.
\eneq

\end{thm}

\begin{proof}
It is clear that (1) follows from (2). So we will prove (2).
As in the beginning of the proof of  \ref{Lext}, there is a unital simple AH-algebra $B$ with slow dimension growth
and with real rank zero such that
$$
(K_0(B), K_0(B)_+, [1_B], K_1(B))=(K_0(A), K_0(A)_+, [1_A], K_1(A))
$$
and since $\rho_B(K_0(B))$ and $\rho_A(K_0(A))$ are dense in
$\Aff(T(B))$ and $\Aff(T(A)),$ respectively,  the above also gives an affine homeomorphism
from $\Aff(T(B))$ to $\Aff(T(A))$ which is compatible with the above identification.
We will use this fact in the proof of (2).

Let $x\in A$ be a normal element with ${\rm sp}(x)=X.$
Put
$$
d= \sup\{{\rm dist}(\lambda, {\rm sp}(x)):
\lambda\,\,\, {\rm in\,\,\, bounded\,\,\, components\,\,\,\, of }\,\,\, \C\setminus {\rm sp}(x)\}.
$$

Let $S$ be the union of all bounded components of $\C\setminus {\rm sp}(x).$
Then
$$
\sup\{|\lambda|: \lambda\in S\}\le \|x\|.
$$
In particular, $d\le \|x\|.$
There is $\lambda_0\in S$ such that
${\rm dist}(\lambda_0, {\rm sp}(x))=d_0>0.$  So $d\ge d_0.$
The set
\beq\label{Knot=1-7+1}
S_1=\{\xi\in S: \|x\|\ge {\rm dist}(\xi, {\rm sp}(x))\ge d_0\}
\eneq
is compact. It follows that there is $\lambda\in S_1$ such that
\beq\label{Knot=1-7+2}
{\rm dist}(\lambda, {\rm sp}(x))=d.
\eneq

Let $\phi_1: C(X)\to A$ be the unital monomorphism defined by $\phi(f)=f(x)$ for all $f\in C(X).$
By \cite{LnMZ}, since $K_1(A)\not=\{0\},$ there exists a unital monomorphism $\psi: C(X)\to B\subset A$ such that
$\psi_{*0}=\phi_{*0},$ $\tau\circ \psi=\tau\circ \phi$ for all $\tau\in T(B)=T(A)$ and
$[\lambda-\psi(z)]\not=[\lambda-x]$ in $K_1(A),$
where
$z: X\to X$ is the identity function.

Let $y\in {\cal U}(\psi(z)).$
It follows from   \ref{Ltp=close} and  \ref{dtP} that
\beq\label{Knot=1-9-1}
D_c(x, y)=0.
\eneq
Since $(\lambda-x)^{-1}(\lambda-\psi(z))\not\in {\rm Inv}_0(A),$
by  \ref{Dvpro},
\beq\label{Knot=1-9}
{\rm dist}({\cal U}(x), {\cal U}(y))\ge 2{\rm dist}(\lambda, {\rm sp}(x))
=2d.
\eneq

\end{proof}

%

{Please note that the above   \ref{Knot=1}  does not follow from  the following theorem.}

\begin{thm}\label{MT2}
Let $A$ be a unital separable simple \CA\, with real rank zero, stable rank one and weakly unperforated $K_0(A)$ and let $x, y\in A$ be two normal elements.

Then
\beq\label{MT2-1}
\rho(x,y)&\le&  {\rm dist}({\cal U}(x), {\cal U}(y))\le
\min\{D_1, D_2\},
\eneq
where
\beq\label{MT2-1+1}
D_1&=&\max\{D^T(x,y),\max\{\rho_x(x,y),\rho_y(x,y)\}\} +\min\{\rho_x(x,y),\rho_y(x,y)\},\\
D_2&=&D_c^e(x,y)+2\min\{\rho_x(x,y),\rho_y(x,y)\}.
\eneq

\end{thm}

\begin{proof}
Let $X={\rm sp}(x)$ and $Y={\rm sp}(y).$
Let  $d=D^T(x,y),$  $d/2>\ep>0$ and let
\beq\label{MT2-5}
S=\{\lambda\in \C: \lambda\not\in X\cup Y,\,\,\, (\lambda-x)^{-1}(\lambda-y)\not\in {\rm Inv}_0(A)\}.
\eneq
Note that the closure ${\bar S}$ of $S$ is compact.
Let $\xi_1,\xi_2,...,\xi_{L'}$ be a finite subset of $X\cup Y\cup {\bar S}$ such that it is
$\ep/32$-dense in $X\cup Y\cup {\bar S}.$
Let $N_1, N_2,...,N_L$ be all possible finite unions of $O(\xi_i, \ep/32)'s$
such that $N_j\cap Y\not=Y$ for $j=1,2,...,L.$ For each $i,$
let
\beq\label{MT2-4-1}
\eta_i=\inf\{d_\tau(f_{(N_i)_{d+\ep/32}}(x))-d_\tau(f_{N_i}(y)): \tau\in T(A)\}.
\eneq
It follows from  \ref{d<d} that $\eta_i>0,$ $i=1,2,...,L.$ Choose
\beq\label{MT2-4}
0<\eta<\min\{\ep/4,\,\,\min\{\eta_i: 1\le i\le L\}/8\}.
\eneq


Let $\dt>0$ with $\dt<\min\{\ep/2^{10},\eta/16\}.$
Let $g_i\in C(X\cup Y\cup {\bar S})$ be such that $0\le g_i(t)\le 1,$
$g_i(t)=1$ if $t\in (N_i)_{d+\ep/8},$ $g(t)=0$ if $t\not\in (N_i)_{d+\ep/4},$
$i=1,2,...,L.$

There are distinct points $\zeta_1, \zeta_2,...,\zeta_K\in {\bar S}$ such
that
\beq\label{MT2-7}
\cup_{i=1}^K O(\zeta_i, \dt/4)\supset {\bar S}.
\eneq

Let  $S_1, S_2,...,S_K$ be compact subsets of ${\bar S}$ such that
\beq\label{MT2-8}
\zeta_i\in S_i\andeqn {\rm diam}(S_i)<\dt/2,\,\,\, i=1,2,...,K.
\eneq
Since $sp(x)$ is compact, there are $\lambda_1,\lambda_2,...,\lambda_K\in {\rm sp}(x)$ 
such that
\beq\label{MT2-9}
{\rm dist}(\lambda_i, \zeta_i)={\rm dist}({\rm sp}(x), \zeta_i),
\eneq
$i=1,2,...,K.$

Let ${\cal H}=\{z, g_i: 1\le i\le L\},$ where $z$ represents the identity
function on $X\cup Y\cup {\bar S}.$

By  \ref{MCdig} and  \ref{dcedig},
there are nonzero mutually orthogonal projections
$\{e_1,e_2,...,e_K,e_{K+1},...,e_k\},$ a unitary $w$
in $A$
and  normal elements
$x_0\in (1-Q_1)A(1-Q_1)$ with ${\rm sp}(x_0)=X$
and $y_0\in (1-Q_2)A(1-Q_2)$ with ${\rm sp}(y_0)=Y$
satisfy the following:
\beq\label{MT2-10}
&&\|f(x)-(f(x_0)+\sum_{i=1}^kf(\lambda_i) e_i)\|<\dt/4\tforal f\in {\cal H},\\\label{MT2-10+1}
&&\|f(y)-w^*(f(y_0)+\sum_{i=K+1}^kf(\lambda_i)e_i)w\|<\dt/4\tforal f\in {\cal H}\\\label{MT2-10+2}
&&[\lambda-x]=[\lambda-x_1]\rforal \lambda\not\in X\andeqn\\\label{MT2-10+3}
&&\tau(\sum_{i=1}^ke_i)<\eta/2,\\
&&D_c(x_0+\sum_{i=1}^K\lambda_ie_i,y_0)<D_c^e(x,y)+\dt/4
\eneq
for all $\tau\in T(A),$
where
$Q_1=\sum_{i=1}^ke_i$  and
$Q_2=\sum_{i=K+1}^ke_i,$
$x_1=x_0+\sum_{i=1}^k\lambda_ie_i.$
and
$\{\lambda_{K+1},\lambda_{K+2},...,\lambda_k\}$ is $\dt/4$-dense in $X\cap Y.$
As in the proof of  \ref{Lext},
there are normal elements $h_i\in e_iAe_i$ such that
${\rm sp}(h_i)=S_i,$   $\lambda-h_i\in {\rm Inv}_0(e_iAe_i)$
for all $\lambda\not\in S_i,$ {$i=1,2,...,K.$}

Define
\beq\label{MT2-13}
x_2=x_0+\sum_{i=1}^K h_i+\sum_{i=K+1}^k\lambda_i e_i
\eneq
It follows that
\beq\label{MT2-14}
\|x-x_2\| &\le & \|x-x_1\|+\|x_1-x_2\|\\
&<& \dt/4+\|\sum_{i=1}^K\lambda_i e_i-\sum_{i=1}^K h_i\|\\
&\le & \dt/4+\max\{\|\lambda_ie_i-h_i\|:1\le i\le K\}\\\label{MT2-14+}
&<&\dt/4+\rho_x(x,y).
\eneq

Let $Z=X\cup {\bar S}.$ Define $\psi: C(\Omega)\to A$ by $\psi(f)=f(x_2)$
for all $f\in C(\Omega)$ and define
$\psi_Y: C(\Omega)\to  A$ by $\psi_Y(g)=g(y)$ for $g\in C(\Omega).$
Let $\lambda\not\in Z\cup Y=X\cup Y\cup {\bar S}.$
By the assumption and (\ref{MT2-10+1}),
\beq\label{MT2-15}
[\lambda-x_2]=[\lambda-y]\rforal \lambda\not\in Z\cup Y.
\eneq
Since $\dt<\eta/16,$ by (\ref{MT2-10}),
\beq\label{MT2-16}
\tau(f_i(x_1))>\tau(f_i(x))-\eta/16\rforal \tau\in T(A),
\eneq
$i=1,2,...,L.$
Let
\beq\label{MT2-16n}
d_1=d_H({\rm sp}(x_2), Y)=\max\{d_H(X, Y), \rho_y(x,y)\}.
\eneq
Let $O\subset Y\cup X\cup {\bar B}$ be an open subset with $O\cap Y\not=Y.$
If $O_{\ep/2}\cap Y=Y,$
since $A$ is simple,
\beq\label{MT2-16n+1}
d_\tau(\psi_Y(f_O))<d_\tau(\psi_Y(f_{O_{\ep/2}})) {\rforal \tau\in T(A).}
\eneq
But we also have that   $O_{d_1+\ep}\cap Z{\supset} \,{\rm sp}(x_2).$
Then $\psi(f_{O_{d_1+\ep}})=1_A.$ It follows that
\beq\label{MT2-16n+2}
d_\tau(\psi_Y(f_O))<d_\tau(\psi_Y(f_{O_{\ep/2}}))=
d_\tau( \psi(f_{O_{d_1+\ep}})).
\eneq
If $O_{\ep/2}\cap Y\not=Y,$
let $O_{\ep/32}\cap \{\xi_1,\xi_2,...,\xi_{L'}\}=\xi_{k_1}, \xi_{k_2},...,\xi_{k_l}.$
Then $O\subset \cup_{j=1}^l O(\xi_{k_j}, \ep/32)\subset O_{\ep/16}.$ It follows that
there is $j$ such that
\beq\label{MT2-17}
O\subset N_j\subset (N_j)_{d+\ep/16}\subset O_{d+\ep/8}.
\eneq

By  (\ref{MT2-10+3}), (\ref{MT2-16}),  (\ref{MT2-17}) and (\ref{MT2-4-1}), we have
\beq\label{MT2-17+}
\hspace{-0.6in}d_\tau(\psi(f_{O_{d+\ep}}))-d_\tau(\psi_Y(f_O)) &>& d_\tau(f_{O_{d+\ep}}(x_0))-d_\tau(\psi_Y(f_{N_j}))-\eta/2\\
&\ge & \tau(f_{O_{d+\ep}}(x_1))-d_\tau(\psi_Y(f_{N_j}))-\eta/2-\eta/2\\
&\ge & \tau({g_j}(x_1))-d_\tau(\psi_Y(f_{N_j}))-\dt/4-\eta\\
&> &  \tau({g_j}(x))-d_\tau(\psi_Y(f_{N_j}))-\eta/16-\dt/4-\eta\\
&\ge & d_\tau(f_{(N_j)_{d+\ep/16}}(x))-d_\tau(\psi_Y(f_{N_j}))-17\eta/16 -\dt/4\\\label{MT2-17++}
&\ge & \eta_j-17\eta/16-\eta/64>0
\eneq
for all $\tau\in T(A).$
By (\ref{MT2-17+})-(\ref{MT2-17++}) and (\ref{MT2-16n})
\beq\label{MT2-18}
D^T(x_2,y)\le
\max \{D^T(x, y)+\ep, \rho_y(x,y)\}.
\eneq
It follows from this,  (\ref{MT2-15}),  \ref{dce<DT} and  \ref{MT1} that
\beq\label{MT2-19}
{\rm dist}({\cal U}(x_2), {\cal U}(y))\le \max \{D^T(x, y)+\ep, \rho_y(x,y)\}.
\eneq
Combining this with (\ref{MT2-14}) and (\ref{MT2-14+}),
we have ($\dt<\ep/2^{10}$)
\beq\label{MT2-20}
{\rm dist}({\cal U}(x), {\cal U}(y))\le \max \{D^T(x, y)+\ep, \rho_y(x,y)\}+\rho_x(x, y)+\ep/2^{12}
\eneq
for all $\ep>0.$
Therefore
\beq\label{MT2-21}
{\rm dist}({\cal U}(x), {\cal U}(y))\le \max \{D^T(x, y), \rho_y(x,y)\}+\rho_x(x, y).
\eneq
Since we may switch the position of $x$ and $y,$ we conclude that
\beq\label{MT2-22}
\hspace{-0.3in}{\rm dist}({\cal U}(x), {\cal U}(y))\le \max \{D^T(x, y), \max\{\rho_x(x,y),\rho_y(x,y)\}\}+\min\{\rho_x(x, y), \rho_y(x, y)\}.
\eneq

On the hand,
we have
\beq\label{MT2-23}
D_c^e(x_2, y) &\le &D_c(x_0+\sum_{i=1}^K h_i, y_0)\\
&\le & \rho_x(x, y)+D_c(x_0+\lambda_ie_i, y_0)\le \rho_x(x,y)+D_c^e(x, y)+\dt/4.
\eneq
It follows from  \ref{MT1} that
\beq\label{MT2-24}
{\rm dist}({\cal U}(x_2), {\cal U}(y))\le \rho_x(x,y)+D_c^e(x, y)+\dt/4.
\eneq
By (\ref{MT2-14+}),
\beq\label{MT2-25}
{\rm dist}({\cal U}(x), {\cal U}(y))\le D_c^e(x,y)+2\rho_x(x,y)+\dt/4.
\eneq
Since we can exchange $x$ with $y$ in the above proof, finally, we conclude
that
\beq\label{MT2-26}
{\rm dist}({\cal U}(x), {\cal U}(y))\le D_c^e(x, y)+2\min{\{\rho_x(x,y),\rho_y(x,y)\}}.
\eneq
\end{proof}

In some  special case below exact formula for distance can be stated.

\begin{cor}\label{CMT2}
Let $A$ be a unital separable simple \CA\, of real rank zero, stable rank one and weakly unperforated $K_0(A)$ and let $x,\, y\in A$ be two normal elements.
If  $D_c(x,y)=0,$
then
\beq\label{CMT2-1}
{\rm dist}({\cal U}(x), {\cal U}(y))=\rho_1(x,y).
\eneq
If $X=Y$ and $X$ is connected, then
\beq\label{CMT2-n}
{\rm dist}({\cal U}(x), {\cal U}(y))\le \max\{D_c(x,y), (1/2)\rho_1(x,y)\}+(1/2)\rho_1(x,y).
\eneq

\end{cor}

\begin{proof}
In this case ${\rm sp}(x)={\rm sp}(y)$ and $d_H(X,Y)=0.$
Therefore $\rho_x(x,y)=\rho_y(x,y)$ and
$$
\rho_1(x, y)=\rho_x(x, y)+\rho_y(x,y)=2\rho_x(x, y).
$$
{In case that $X$ is connected, by  \ref{dtP}, $D_c(x, y)=D^T(x, y).$ }
Thus the corollary follows from  \ref{MT2}.
\end{proof}

\section{Lower bound}
Last section gives both upper bound and lower bound for the distance
between unitary orbits of normal elements. However, the lower bound
are all given by the bounded components of $\C\setminus X\cup Y$ which
give different $K_1$-information of the corresponding normal elements.
In this section, we will discuss the lower bound  {for distance between unitary orbits
 of normal elements} who have the same $K_1$-information outside of $X\cup Y.$

\begin{thm}\label{BoundT}
There exists a constant $C>0$ satisfying the following:
Let $A$ be a unital separable  AF-algebra and let $x,\, y\in A$ be two normal elements.  Then
\beq\label{BoundT-1}
C \cdot D_c (x, y)\le {\rm dist}({\cal U}(x), {\cal U}(y))\le D_c(x, y).
\eneq

\end{thm}

\begin{proof}
Let $C=c^{-1}$ be in the statement of Theorem 4.2 of \cite{Dv}.
Without loss of generality, we may assume that $\|x\|, \|y\|\le 1.$
It follows from  \ref{Toriginal} that
it suffices to show that
\beq\label{BoundT1-2}
{\rm dist}({\cal U}(x), {\cal U}(y))\ge C\cdot D_c(x, y).
\eneq
We will show that
\beq\label{BoundT1-2+}
\|x-y\|\ge C\cdot D_c(x, y),
\eneq

Put $d=D_c(x, y).$
Let $\ep>0.$ It follows from \cite{Lnoldjfa} that there are $\lambda_1, \lambda_2,...,\lambda_n\in {\rm sp}(x),$
$\mu_1, \mu_2,...,\mu_m\in {\rm sp}(y),$ two sets of mutually orthogonal
non-zero projections $\{p_1,p_2,...,p_n\}$ and $\{q_1,q_2,...,q_m\}$ in $A$
such that
\beq\label{BoundT1-3}
\|x-\sum_{i=1}^n \lambda_i p_i\|<\ep/16\andeqn \|y-\sum_{j=1}^m \mu_j q_j\|<\ep/16.
\eneq
Put $x_1=\sum_{i=1}^n \lambda_i p_i$ and $y_1=\sum_{j=1}^m \mu_j q_j.$
Without loss of generality, by the virtue of \ref{Lappdcu}, we may also assume that
\beq\label{BoundT1-4}
D_c(x, x_1)<\ep/16\andeqn D_c(y, y_1)<\ep/16.
\eneq
Since $D_c(\cdot, \cdot)$ is a metric,
\beq\label{BoundT1-5}
D_c(x_1, y_1)\ge D_c(x,y)-\ep/8.
\eneq
Let $\ep>\dt>0$ be given.
Since $A$ is an AF-algebra, there is a finite dimensional \CA\, $B\subset A$ such that
there are mutually orthogonal projections $\{p_1',p_2',...,p_n'\}$ and $\{q_1', q_2',...,q_m'\}$ in $B$ such that
\beq\label{BoundT1-6}
\|p_i-p_i'\|<\dt/16n\andeqn \|q_j-q_j'\|<\dt/16m,\,\,\,
\eneq
$i=1,2,...,n$ and $j=1,2,...,m.$
Put $x_2=\sum_{i=1}^n \lambda_ip_i'$ and $y_2=\sum_{j=1}^m \mu_jq_j'.$
Therefore, by the virtue of  \ref{Lappdcu} and choosing sufficiently small $\dt,$
\beq\label{BoundT1-7}
\|x_1-x_2\|<\ep/16,\,\,\, \|y_1-y_2\|<\ep/16\andeqn\\
D_c(x_1, x_2)<\ep/16\andeqn D_c(y_1,y_2)<\ep/16.
\eneq
It follows from (\ref{BoundT1-5}) and (\ref{BoundT1-7}) that
\beq\label{BoundT1-8}
D_c(x_2, y_2)\ge d-\ep/4.
\eneq
This has to hold in $B$ too. By Theorem 4.2 of \cite{Dv},
\beq\label{BoundT1-9}
{\rm dist}(x_2, y_2)\ge C (d-\ep/4).
\eneq
It follows that
\beq\label{BoundT1-10}
\|x-y\|\ge C(d-\ep/4)-\ep/4=C\cdot d- C\ep/4-\ep/4
\eneq
for any $\ep>0.$ 

{So we get  (\ref{BoundT1-2+}).}
{ For any unitary $u$, by (\ref{BoundT1-2+}), we have
$$
\|u^*xu-y\|\ge C\cdot D_c(u^*xu,y)=C\cdot D_c(x,y).
$$
That implies (\ref{BoundT1-2}) holds.}
\end{proof}

\begin{lem}\label{LowerL1}
Let $A$ be a unital simple \CA\, with $TR(A)=0,$ let $x, y\in A$ be two normal elements. Let $\eta>0.$ Suppose that
\beq\label{LowerL1-1}
\tau(f(x))>\tau(g(y))
\eneq
for some positive functions $f\in C(\overline{X_\eta})$ and $g\in C(\overline{Y_\eta})$ and for some
$\tau\in T(A).$
Then, for any $\ep>0,$ there is a projection $p\in A,$ and  a finite dimensional
\SCA\, $B\subset A$ with $1_B=p,$ and normal elements
$x_0, y_0\in (1-p)A(1-p),$ $x_1, y_1\in B$  such that
${\rm sp}(x_1)\subset \overline{X_\eta},$ ${\rm sp}(y_1)\subset \overline{Y_\eta}$ with
\beq\label{LowerL1-2}
&&\|x-(x_0+x_1)\|<\ep,\,\,\,\|y-(y_0+y_1)\|<\ep\\
&&t_0(f(x_1))>t_0(g(y_1))
\eneq
for some $t_0\in T(B).$
\end{lem}

\begin{proof}
Let $\tau(f(x))>\tau(g(y))$ for some $\tau\in T(A)$ and let 
$d=\tau(f(x))-\tau(g(y))>0.$ 
Fix a separable \SCA\ $C\subset A$ such that
$x, y\in C.$
Since $A$ has tracial rank zero, there exists a sequence of projections
$\{p_n\}$ and a sequence of finite dimensional \SCA s $\{B_n\}$ such that
$1_{B_n}=p_n,$ $n=1,2,....,$ such that
\beq\label{LowerL1-3}
\lim_{n\to\infty}\|p_nc-cp_n\|=0\tforal c\in C;\\
\lim_{n\to\infty}{\rm dist}(p_ncp_n, B_n)=0\andeqn\\\label{LowerL1-3+}
\lim_{n\to\infty}\max\{t(1-p_n): t\in T(A)\}=0.
\eneq
It follows from \cite{Lnalm} and \cite{FR} that there exists normal elements
$x_n^{(0)}, y_n^{(0)}\in (1-p_n)A(1-p_n),$ $x_n^{(1)},y_n^{(1)}\in B_n$
such that
\beq\label{LowerL1-4}
\lim_{n\to\infty}\|x-(x_n^{(0)}+x_n^{(1)})\|=0
\andeqn \lim_{n\to\infty}\|y-(y_n^{(0)}+y_n^{(1)})\|=0.
\eneq
Therefore, we may assume that ${\rm sp}(x_n^{(0)}), {\rm sp}(x_n^{(1)})\subset \overline{X_\eta}$ and ${\rm sp}(y_n^{(0)}), {\rm sp}(y_n^{(1)})\subset \overline{Y_\eta}.$
It follows that
\beq\label{LowerL1-5}
\lim_{n\to\infty}\|f(x)-(f(x_n^{(0)})+f(x_n^{(1)}))\|=0\andeqn
\lim_{n\to\infty}\|g(y)-(g(y_n^{(0)})+g(y_n^{(1)}))\|=0.
\eneq
By (\ref{LowerL1-3+}), we may assume that
\beq\label{LowerL1-6}
t(1-p_n)<d/4\tforal t\in T(A).
\eneq
It follows, for all sufficiently large $n,$ that
\beq\label{LowerL1-7}
\tau(f(x_n^{(1)}))>\tau(f(x))-d/4-d/4>\tau(g(y_n^{(1)}))
\eneq
Note $B_n$ is a finite direct sum of simple \CA s.
If, for all tracial states $t\in T(B_n),$
\beq\label{LowerL1-8}
t(f(x_n^{(1)}))\le t(g(y_n^{(1)})),
\eneq
then, by \cite{CP}, there is a sequence $\{z_{k,n}\}\subset B_n$ such that
\beq\label{LowerL1-9}
\sum_{k=1}^{\infty}z_{k,n}^*z_{k,n}=f(x_n^{(1)})\andeqn
\sum_{k=1}^{\infty}z_{k,n}z_{k,n}\le g(y_n^{(1)}).
\eneq
Since $B_n\subset A,$ this would imply that
\beq\label{LowerL1-10}
\tau(f(x_n^{(1)}))\le \tau(g(y_n^{(1)}))\tforal \tau\in T(A)
\eneq
which contradicts with (\ref{LowerL1-7}).

\end{proof}

\begin{thm}\label{LowerT1}
There is a constant $C>0$ satisfying the following:
Let $A$ be a unital separable simple \CA\, with $TR(A)=0$ and let $x,\, y\in A$ be two normal elements. Then
\beq\label{LowerT1-1}
{\rm dist}({\cal U}(x), {\cal U}(y))\ge C\cdot  D_T(x,y).
\eneq
If $[\lambda-x]= [\lambda-y]$ in $K_1(A)$  for all $\lambda\not\in {\rm sp}(x)\cup {\rm sp}(y),$
then
\beq\label{Ln2}
D_c^e(x, y)\ge {\rm dist}({\cal U}(x), {\cal U}(y))\ge C\cdot  D_T(x,y).
\eneq
\end{thm}

\begin{proof}
Note that the second part of the theorem follows from the first part and
 \ref{MT1}. So we will only prove the first part of the theorem.

Let ${C}$ be the constant $c^{-1}$ in the statement of Theorem 4.2 of \cite{Dv}.
Let $r=D_T(x, y)$ and let $1/2>\ep>0.$
Suppose that $K>0$ such that $\|x\|,\,\|y\|\le K$ and $D$ is a closed ball with the center at the origin and radius larger than $K.$
Denote by $\phi_X, \phi_Y: C(D)\to A$ the unital \hm s
defined by $\phi_X(f)=f(x)$ and $\phi_Y(f)=f(y)$ for all $f\in C(D).$
We will show that
\beq\label{LowerT1-2-}
\|x-y\|\ge {Cr}.
\eneq

There is an open subset $O$ of  $D$ such that
\beq\label{LowerT1-2}
d_\tau(\phi_X(f_O))>  d_\tau(\phi_Y(f_{O_{s_1}})),
\eneq
for some $\tau\in T(A),$ where $s_1=r-\ep/8.$
It follows that
\beq\label{LowerT1-3}
d_\tau(\phi_X(f_O))>\tau(\phi_Y(g)),
\eneq
where $g\in C(D)$ such that $0\le g\le 1,$ $g(\xi)=1$ if
${\rm dist}(\xi, O)<r-\ep/2$ and $g(\xi)=0$ if ${\rm dist}(\xi, O)\ge r-\ep/4.$
There is $\dt>0$ such that
\beq\label{LowerT1-4}
\tau(f_\dt(\phi_X(f_O)))>\tau(\phi_Y(g)).
\eneq
By  \ref{LowerL1}, there is a projection $p\in A,$ a finite dimensional
\CA\, $B\subset A$ with $1_B=p,$ normal elements $x_0, y_0\in (1-p)A(1-p),$
$x_1, y_1\in B$ with ${\rm sp}(x_0),{\rm sp}(y_0), {\rm sp}(x_1), {\rm sp}(y_1)\subset D$ such that
\beq\label{LowerT1-5}
\|x-(x_0+x_1)\|<\ep/16,\,\,\,\|y-(y_0+y_1)\|<\ep/16\andeqn\\
t_0(f_\dt(f_O(x_1)))>t_0(g(y_1))
\eneq
for some $t_0\in B.$  Therefore
\beq\label{LowerT1-6}
d_{t_0}(f_O(x_1))\ge t_0(f_\dt(f_O(x_1)))>d_{t_0}(f_{O_{r-\ep/2}}{(y_1)}){.}
\eneq
It follows that, in $B,$
\beq\label{LowerT1-7}
D_c(x_1, y_1)\ge r-\ep/2.
\eneq
It follows from Theorem 2.4 of \cite{Dv} that
\beq\label{LowerT1-8}
\|x_1-y_1\|\ge C(r-\ep/2).
\eneq
It follows from (\ref{LowerT1-5}) that
\beq\label{LowerT1-9}
\|x-y\| &\ge & \|(x_0+x_1)-(y_0+y_1)\|- \|x-(x_0+x_1)\|-\|(y_0+y_1)-y\|\\
&\ge & \|x_1-y_1\|-\ep/8\ge Cr-C\ep/2-\ep/8.
\eneq
Hence
\beq\label{LowerT1-10}
\|x-y\|\ge C\cdot D_T(x, y).
\eneq

\end{proof}

\begin{lem}\label{Lprojnrm}
Let $A$ be a unital simple \CA\, with $TR(A)=0.$
Suppose that $p, q\in A$ are two non-zero projections such that
\beq\label{Lprojn-1}
\tau(p)>\tau(q)
\eneq
for some $\tau\in T(A).$
Then
\beq\label{Lprojn-2}
\|(1-q)p\|=1
\eneq
\end{lem}

\begin{proof}
We first apply  \ref{LowerL1}.
Let $x=p,$ $y=q$ and  $f=g$ be identity function on $[0,1].$
Then, for any $\ep>0,$ there are a non-zero projection $e\in A$ and
a finite dimensional \CA\, $B\subset A$ with $1_B=e,$ non-zero projections
 $p_1, q_1\in B$ and $p_0,\, q_0\in (1-e)A(1-e)$
 such that
\beq\label{Lprojn-3}
\|p-(p_0+p_1)\|<\ep/4,\,\,\, \|q-(q_0+q_1)\|<\ep/4\andeqn
t_0(p_1)>t_0(q_1)
\eneq
for some $t_0\in T(B).$
Since $B$ is finite dimensional,
we write $B=M_{n_1}\oplus M_{n_2}\oplus\cdots \oplus M_{n_k}.$
Accordingly, we may write
\beq\label{Lprojn-4}
p_1=(p_{1,1},..., p_{1,n_k}),\,\,\, q_1=(q_{1,1},...,q_{1, n_k}),\\
\eneq
where $p_{1,i},q_{1,i},\in M_{n_i},$ $i=1,2,...,n_k.$
The last condition in (\ref{Lprojn-3}) implies, for some $i,$
\beq\label{Lprojn-5}
{\rm rank} p_{1,i}>{\rm rank} q_{1,i}.
\eneq
Let $M_{n_i}$ act on $H_i$ (${\rm dim} H_i=n_i$).
Then, by counting the rank,
$(1_{B_i}-q_{1,i})H_i\cap p_{1,i}H_i\not=\{0\}.$ Therefore
\beq\label{Lprojn-6}
\|(1_{B_i}-q_{1,i})p_{1,i}\|=1.
\eneq
It follows that
\beq\label{Lprojn-7}
\|(e-q_1)p_1\|=1
\eneq
Let $\xi\in (1_{B_i}-q_{1,i})H_i\cap p_{1,i}H_i$ be a unit vector.
Define a projection $e_0\in B(H_i)=M_{n_i}$ by $e_0(x)=\langle x, \xi \rangle \xi$
for all $x\in H_i.$  Then $e_0\in M_{n_i}\subset A$ is a non-zero projection.  Moreover,
\beq\label{Lpronjn-12}
e_0\le e-q_1\,\,\,{\rm and }\,\,\,e_0\le p_1.
\eneq
Therefore
\beq\label{Lprojn-8}
\|(1-q)p\|&\ge &\|1-(q_0+q_1)(p_0+p_1)\|-\ep/2\\
&\ge & \|e(1-(q_0+q_1)(p_0+p_1))e\|-\ep/2=\|(e-q_1)p_1\|-\ep/2\\
&=& 1-\ep/2{.}
\eneq
It follows that (\ref{Lprojn-2}) holds.


\end{proof}

\begin{thm}\label{TLower}
Let $A$ be a separable  simple \CA\, with  $TR(A)=0$ and let $x,
y\in A$ be two normal  elements. Then \beq\label{TLower-1} {\rm
dist}({\cal U}(x), {\cal U}(y))\ge {d_T}(x,y). \eneq

If $A$ is a finite dimensional \CA\, then \beq\label{Tlower-1n} {\rm
dist}({\cal U}(x), {\cal U}(y))\ge {d_c}(x,y). \eneq
\end{thm}

\begin{proof}
Let $0<d<d_T(x,y).$ Note that in a finite dimensional \CA\,
${d_c}(x, y)={d_T}(x, y).$ Let $\ep>0.$ We assume that
$\ep<{{d_T}(x, y)-d\over{4}}.$ We also assume that $\|x\|,
\|y\|>2\ep.$ Denote $X={\rm sp}(x)$ and $Y={\rm sp}(y).$ Choose  any
pair of $x' \in {\cal U}(x)$ and $y'\in {\cal U}(y).$ Since ${\rm
sp}(x')=X,$  ${\rm sp}(y')=Y,$ ${d_c}(x', y')={d_c}(x, y),$
and ${d_T}(x',y')={d_T}(x,y),$ to simplify the notation,
without loss of generality, it suffices to show that $\|x-y\|\ge
{d_T}(x, y)$. 

By the assumption, {there is an open disc $O=O(\lambda, \eta)$ such that}
\beq\label{Tlower-2n}
d_\tau(f_O(x))> d_{\tau}(f_{O_{d+{\ep}}}(y))
\eneq
for some $\tau\in T(A)$ (including the case that $A$ is finite dimensional).

Let $e_1$ be the spectrum projection of $x$ corresponding to open set ${O_{\eta+\ep}}$
and $e_2$ be the spectrum projection of $y$ corresponding to
$O_d$ in $A^{**}.$

Denote by $z=y(1-e_2)=(1-e_2)y.$ Then
\beq\label{TLower-3}
{\rm sp}_{(1-e_2)A^{**}(1-e_2)}(z)=Y\setminus O_d\cap Y
\eneq
(as an element in $(1-e_2)A^{**}(1-e_2)$). The inequality (\ref{Tlower-2n}) implies that
$Y\not=O_d\cap Y.$
In particular,
\beq\label{Tlower-4}
{\rm dist}(\lambda, {\rm sp}_{(1-e_2)A^{**}(1-e_2)}(z))\ge d.
\eneq

We also note that
\beq\label{Tlower-5}
\|xe_1-\lambda e_1\|<{\eta+\ep}.
\eneq
Therefore
\beq\label{Tlower-5+}
\|(1-e_2)(y-\lambda)e_1\| &\le & \|(1-e_2)(y-x)e_1\|+\|(1-e_2)(x-\lambda)e_1\|\\
&<& \|(1-e_2)(y-x)e_1\|+{\eta+\ep}. \eneq It follows that
\beq\label{Tlower-6} \|(1-e_2)(y-x)e_1\|>
\|(1-e_2)(y-\lambda)e_1\|-{\eta-\ep}. \eneq One has
\beq\label{Tlower-7}
(1-e_2)(y-\lambda)e_1 &=&(y-\lambda)(1-e_2)e_1\\&=&(1-e_2)(y-\lambda)(1-e_2)e_1\\
   &=& (z-\lambda)(1-e_2)e_1{.}
\eneq
Let $z_1$ be the inverse of $z-\lambda$ in $(1-e_2)A^{**}(1-e_2).$
Then
\beq\label{Tlower-8}
\|(1-e_2)e_1\|\le \|z_1(z-\lambda)(1-e_2)e_1\|\le \|z_1\|\|(z-\lambda)(1-e_2)e_1\|.
\eneq
It follows that
\beq\label{Tlower-9}
\hspace{-0.5in}
\|(z-\lambda)(1-e_2)e_1\| &\ge & {\|(1-e_2)e_1\|\over{\|z_1\|}}\\
&=& {\rm  dist}(\lambda, {\rm sp}_{(1-e_2)A^{**}(1-e_2)}(z))\|(1-e_2)e_1\|\\
&\ge& {(\eta+d)}\|(1-e_2)e_1\|. \eneq
By (\ref{Tlower-6}) and (\ref{Tlower-7}), one concludes that
\beq\label{Tlower-12}
\|y-x\|\ge \|(1-e_2)(y-x)e_1\| &>&\|(1-e_2)(y-\lambda)e_1\|-{\eta-\ep}\\
&\ge & \|(z-\lambda)(1-e_2)e_1\| -{\eta-\ep}\\\label{Tlower-12+}
&\ge& {(\eta+d)}\|(1-e_2)e_1\|-{\eta-\ep}.
\eneq
If $A$ is finite dimensional, then $e_1, e_2\in A.$
By (\ref{Tlower-2n}),
 \beq\label{Tlower-13}
  {\rm rank} e_1>{\rm rank}e_2.
\eneq
As in the proof of  \ref{Lprojnrm}, this implies that
$\|(1-e_2)e_1\|=1.$ From this and from   (\ref{Tlower-12}) to (\ref{Tlower-12+}),
\beq\label{Tlower-14} \|x-y\|\ge {d}. \eneq
It follows that \beq\label{Tlower-15} \|x-y\|\ge {d_c}(x, y).
\eneq Let $e_0$ be the spectral projection of $x$ corresponding to
the closed set $\{\xi\in \C: {\rm dist}(\xi, \lambda)\le \eta\}$
and $e_3$ be the  spectral projection of $y$ corresponding to the
open subset $O_{d+{\ep}}$ in $A^{**}.$ Note that $e_0$ is a
closed projection and $e_3$ is an open projection. If $A$ is a
simple infinite dimensional \CA\, with $TR(A)=0,$ by \cite{Bro},
there are projections $p_1, q_1\in A$ such that
\beq\label{Tlower-16} e_0\le  q_1\le e_1\andeqn e_2\le  p_1 \le e_3.
\eneq By (\ref{Tlower-2n}), \beq\label{Tlower-17}
\tau(q_1)>\tau(p_1). \eneq
 It follows from  \ref{Lprojnrm} that
\beq\label{Tlower-18} \|(1-p_1)q_1\|=1{.} \eneq By
(\ref{Tlower-12}), \beq\label{Tlower-19}
\|x-y\|&\ge& {(\eta+d)}\|(1-e_2)e_1\|-{\eta-\ep}\\
&\ge &  {(\eta+d)}\|(1-p_1)(1-e_2)e_1q_1\|-{\eta-\ep}\\
&=&{(\eta+d)}\|(1-p_1)q_1\|-{\eta-\ep}\\
&\ge & {d{-\ep}}.
\eneq
The theorem follows.
\end{proof}

\begin{rem}
{ Suppose that  $x,y$ are normal elements in a separable simple $C^*$-algebra $A$ with $TR(A)=0,$ and ${\rm sp}(x)$ is a subset 
of a straight line $L_1$ and ${\rm sp}(y) $ is a subset of another straight line. 
  If $L_1$ and $L_2$ are parallel, by  applying  \ref{TLower},} one can show that 
$
d_c(x,y)=D_c(x,y).
$
Hence by  \ref{TLower} and  3.6,$$\dist(\U(x),\U(y))=D_c(x,y)=d_c(x,y).$$
If $L_1$ and $L_2$ are perpendicular,  one  can  also show
$
\dist(\U(x),\U(y))=D_c(x,y).
$
However, in  \ref{Rd=D} of  section 2, there are $x,y$ in finite dimesional $C^*$-algebra, with $x=x^*,y=iy^*$  and
$$
d_c(x,y)=1<\sqrt 2=D_c(x,y),
$$
So we get an example such that
$$
d_c(x,y)<\dist(\U(x),\U(y))=D_c(x,y).
$$
\end{rem}
\begin{cor}\label{AFT1bc}
Let $A$ be a unital  AF-algebra and let $x,\, y\in A$ be two normal elements.
Suppose that
$D_c(x,y)={d_c(x,y)}.$
Then
\beq\label{AFTlb-1n}
{\rm dist}({\cal U}(x),{\cal U}(y))= D_c(x,y).
\eneq
\end{cor}
\begin{cor}\label{MT3}
Let $A$ be a unital simple separable \CA\, with $TR(A)=0$ and let $x,\, y\in A$ be two normal
elements with ${\rm sp}(x)=X$ and ${\rm sp}(y)=Y.$
Then
\beq\label{MT3-0}
\max\{C\cdot D_T(x,y),d_T(x, y),\rho_1(x,y)\}\le {\rm dist}({\cal U}(x), {\cal U}(y))\le \min\{D_1,D_2\},
\eneq
where
\beq\nonumber
D_1&=&\max\{D^T(x,y),\max\{\rho_x(x,y),\rho_y(x,y)\}\}+\min\{\rho_x(x,y),\rho_y(x,y)\}\tand\\
D_2&=&\max\{D_c^e(x,y)+\min\{\rho_x(x,y),\rho_y(x,y)\}\}+ \min\{\rho_x(x,y),\rho_y(x,y)\}.
\eneq
Suppose that
\beq\label{MT3-1}
[\lambda-x]=[\lambda-y]\,\,\, {\rm in}\,\,\, K_1(A)
\eneq
for all $\lambda\not\in X\cup Y.$
Then
\beq\label{MT3-2}
\max\{C\cdot D_T(x,y),d_T(x, y)\}\le {\rm dist}({\cal U}(x), {\cal U}(y))\le D_c^e(x, y).
\eneq
\end{cor}

\end{document}